\documentclass{mscs} 
\usepackage[applemac]{inputenc}
\usepackage[english]{babel}
\usepackage{pdfsync,hyperref}
\let\proof\relax

\usepackage{verbatim}
\usepackage{amsmath,amsthm,amssymb,bussproofs,tikz,stmaryrd,mathtools}
\usepackage[emu]{cmll}
\usepackage{fouriernc}
\usepackage{hyperref}
\usepackage{subfigure}
\usepackage{enumitem} 
\hypersetup{
	pdfdisplaydoctitle		= 	true,			
	pdfinfo={
		Author	=	{Thomas Seiller},
		Title		=	{A Correspondence between Maximal Abelian Sub-Algebras and Linear Logic Fragments},
		Subject	=	{Linear Logic, Geometry of Interaction, Complexity, Operator Algebras, Maximal Abelian Sub-Algebras},
		Keywords	=	{Theoretical computer Science, Linear Logic, Operator algebras},
		Creator	=	{PDFLaTeX},
		Producer	=	{PDFLaTeX},
		Licence 	= 	{Creative Commons BY-NC-SA  3.0 or later},
		Lang		=	{en}
	}
}

\usepackage{macros}

\usepackage{aliascnt}
\newcommand{\mynewtheorem}[2]{
  \newaliascnt{#1}{rubric}
  \newtheorem{#1}[#1]{#2}
  \aliascntresetthe{#1}
  \expandafter\def\csname #1autorefname\endcsname{#2}
}

\theoremstyle{plain}
\mynewtheorem{theorem}{Theorem}
\mynewtheorem{corollary}{Corollary}
\mynewtheorem{lemma}{Lemma}
\mynewtheorem{proposition}{Proposition}

\theoremstyle{definition}
\mynewtheorem{definition}{Definition}

\theoremstyle{remark}
\mynewtheorem{remark}{Remark}

\renewcommand{\eqp}{\sim_{\infhyp}}
\newcommand{\barEx}{\textnormal{E}\bar{\textnormal{x}}}

\title{A Correspondence between Maximal Abelian Sub-Algebras and Linear Logic Fragments}
\shorttitle{A Correspondence between MASAs and Linear Logic Fragments}

\author[Thomas Seiller]
    {T\ls H\ls O\ls M\ls A\ls S\ns S\ls E\ls I\ls L\ls L\ls E\ls R$^{\dagger}$ \thanks{This work was partly supported by the ANR-10-BLAN-0213 Logoi.}\\
      \addressbreak $^{\dagger}$ I.H.E.S., Le Bois-Marie, 35, Route de Chartres, 91440 Bures-sur-Yvette, France\\ \href{mailto:seiller@ihes.fr}{seiller@ihes.fr}
      }
\date{December 15th, 2014; Revised September 23rd, 2015}

\usepackage{color}
\begin{document}

\maketitle

\begin{abstract}
We show a correspondence between a classification of maximal abelian sub-algebras (MASAs) proposed by Jacques Dixmier \cite{dixmier} and fragments of linear logic. We expose for this purpose a modified construction of Girard's hyperfinite geometry of interaction \cite{goi5}. The expressivity of the logic soundly interpreted in this model is dependent on properties of a MASA which is a parameter of the interpretation. We also unveil the essential role played by MASAs in previous geometry of interaction constructions.
\end{abstract}

\tableofcontents

\section{Introduction}

\subsection{Geometry of Interaction.} Geometry of Interaction is a research program initiated by Girard \cite{towards} a year after his seminal paper on linear logic \cite{ll}. Its aim is twofold: define a semantics of proofs that accounts for the dynamics of cut-elimination, and then construct realisability models for linear logic around this semantics of cut-elimination. 

The first step for defining a \emph{GoI model}, i.e. a construction that fulfills the geometry of interaction program, is to describe the set of mathematical objects $\mathcal{O}$ that will represent the proofs, together with a binary operation $\plug$ on this set of objects that will represent the cut-elimination procedure. This function should satisfy one key property: associativity, i.e. $(a\plug f)\plug b=a\plug(f\plug b)$. This property is a reformulation of the Church-Rosser property: if $f$ represents a function taking two arguments $a,b$, one can evaluate the function $f\plug a$ on the argument $b$ or equivalently evaluate the function $f\plug b$ on the argument $a$. Both these evaluation strategies should yield the same result, i.e. the operation $\plug$ is associative. The data of the set $\mathcal{O}$ together with the associative binary operation $\plug$ constitute the first step in defining a GoI model, as it provides the semantics of cut-elimination\footnote{Obviously, one needs this setting to satisfy additional properties in order to interpret full linear logic. We describe here the minimal requirements for obtaining a GoI model, which could be a model of multiplicative linear logic only.}. The majority of works dealing with geometry of interaction content themselves with this part of the geometry of interaction program. It has been a rich source of dynamical interpretation of proofs \cite{DanosRegnierIAM,LaurentTokenGoI,goi3} leading to numerous results about the execution of programs \cite{AbadiGonthierLevy92b,danosregnierlocalasynchronous} and computational complexity \cite{Lago,baillotpedicini,seiller-conl,seiller-lsp,lics-ptime}. 

It is however a crucial mistake to think the second part unimportant.  On the contrary, it is the second part of the program -- the reconstruction of logic from the dynamical interpretations of proofs -- that makes it a very innovative and strong tool, both from a technical and a philosophical point of view. This reconstruction of logic is a step further into the Curry-Howard correspondence: one reconstructs the logic of programs. This part of the GoI program is reminiscent of classical realisability \cite{krivine1,krivine2,streicher-kr,Miquel}. Indeed, the set $\mathcal{O}$ represents a set of programs and the operation $\plug$ describes how these programs compose and evaluate. By choosing a set $\Bot$ of \enquote{bad programs}, such as infinite loops, one can then build the logic that naturally arises from the notion of programs described by $\mathcal{O}$, $\plug$ and the set $\Bot$. This is done in two steps, the first being the definition the notion of types -- subsets of $\mathcal{O}$. Then every operation $\odot$ which allows one to construct new programs $\de{a\odot b}$ from two programs $\de{a}$ and $\de{b}$, lifts to an operation on types: $A\odot B$ is defined as the set of $\de{a\odot b}$ for all $\de{a}$ in $A$ and all $\de{b}$ in $B$. The set of bad programs $\Bot$ defines a notion of \emph{negation}: a program $\de{a'}$ has type $A^{\bot}$ if and only if for all program $\de{a}$ of type $A$ one has $\de{a\plug a'}\in\Bot$. As a consequence, the types and connectives are only descriptions of the structure of the set of programs considered. Let us notice that work in this direction has been directed at obtaining models of linear logic and therefore the notion of types is restricted to biorthogonally closed sets because linear negation is involutive; this is however a choice of design and not a requirement. For the same reasons, the connectives defined and studied in these models are those of linear logic, although many others may be considered.

\subsection{Geometry of Interaction and Maximal Abelian Sub-Algebras.} A major result in the geometry of interaction program was obtained by Girard about ten years ago. In previous work, he had described a way of representing cut-elimination by the so-called \emph{execution formula} $\Ex(A,B)$ between two operators $A,B$. This formula is actually an explicit solution to a functional equation involving $A$ and $B$, the \emph{feedback equation}, but this explicit solution is defined as an infinite series whose convergence can be insured only when the product $AB$ is nilpotent. Using techniques of operator algebras, Girard showed that this explicit solution admits a \enquote{continuation} $\barEx(\cdot,\cdot)$, i.e. can be extended so as to be defined on the whole unit ball of a von Neumann algebra. More precisely, the solution to the feedback equation, noted $\barEx(A,B)$, exists and is unique even when the product $AB$ is not nilpotent as soon as $A,B$ are operators of norm at most $1$; moreover, $\barEx(A,B)$ is an operator of norm at most $1$ in the von Neumann algebra generated by $A,B$. This implies that for every von Neumann algebra $\vn{M}$, there exists an operation $\barEx(\cdot,\cdot)$ which extends the execution formula on the unit ball $\vn{M}_{1}$ of $\vn{M}$. This means that the pair $(\vn{M}_{1},\barEx(\cdot,\cdot))$ fulfills the first part of the GoI program as it yields the set $\mathcal{O}=\vn{M}_{1}$ and an associative operation $\barEx(\cdot,\cdot):\mathcal{O}\times\mathcal{O}\rightarrow\mathcal{O}$. The first constructions proposed by Girard were therefore based on a smaller set of operators, a subset of the unit ball of $\B{\hil{H}}$, the algebra of all continuous linear maps on a Hilbert space $\hil{H}$. The \enquote{pole} $\Bot$ was then defined as the set of nilpotent operators, i.e. two operators are \emph{orthogonal} if and only if their product is nilpotent. After he worked out the general solution to the feedback equation, Girard defined a geometry of interaction in the hyperfinite factor of type {II}$_{1}$ -- the hyperfinite GoI model, where the pole $\Bot$ is chosen based on the determinant: two operators $A,B$ are orthogonal when $\det(1-AB)\neq0,1$. 

The present work stems from an attempt to obtain soundness results for the hyperfinite GoI model. As one can define in this model an exponential connective satisfying the \enquote{functorial promotion rule}, one would expect a soundness result for at least \emph{Elementary Linear Logic}, an exponentially-constrained fragment of linear logic which characterises elementary time functions. It turns out however that the interpretation of proofs depends on the choice of a \emph{Maximal Abelian Sub-Algebra (MASA)} $\vn{A}$ of the algebra $\vn{M}$, and that one can interpret more or less expressive fragments of linear logic according to the \enquote{size} of the algebra generated by the normaliser of $\vn{A}$: if this algebra is $\vn{A}$ itself (the minimal case) then no non-trivial interpretation of proofs exists, if this algebra is $\vn{M}$ (the maximal case) we can interpret soundly elementary linear logic. Although it might seem at first glance a specific feature of the hyperfinite GoI model, it turns out that a choice of MASA for the interpretation of proofs was already made in earlier GoI models \cite{goi1,goi2,goi3}. However, this choice was done implicitly in the definition of the models, and did not affect the expressivity of the interpretation: the algebra generated by the normaliser of any MASA in a type {I} algebra $\vn{M}$ is equal to $\vn{M}$. The passage from a type {I} algebra to a type {II} algebra thus only made clear the crucial role played by the MASAs, thanks to the rich variety of such sub-algebras in type {II} factors.

\subsection{Outline of the Paper}

We will first give an overview of some of the theory of von Neumann algebras. Although we do not expect the reader to learn and feel familiar with the theory from reading this short section, we hope it will give her a broad idea of the richness of the subject and a few intuitions about the objects it studies. This also gives us the opportunity to define notions and state results that will be used in the following sections.

In a second section, we offer a historical overview of the various GoI models defined by Girard. This overview gives us the opportunity to provide a homogeneous point of view on those since we define all the models using operator theory. It is moreover the occasion to pinpoint the implicit choice of a MASA which was made in earlier constructions; this choice, which was unimportant in these constructions but is crucial in the hyperfinite GoI model, has seemingly never been noticed before.

We then motivate the notion of \emph{subjective truth} which appears in Girard's hyperfinite GoI. We try to explain why it is necessary to have a notion of truth that depends on the choice of a MASA. We then define a variant of Girard's hyperfinite GoI model, which we call the \emph{matricial GoI model}. This variant of Girard's model makes a more explicit connection with MASAs and will be used to prove the main theorem of the paper. We show how a satisfying notion of truth depending on a MASA can be defined for this model and relate it with the notion of truth defined by Girard in his hyperfinite GoI model.

The last section is then concerned with the proof of the main theorem of the paper. We show that the expressivity of the fragment of linear logic one can soundly interpret is in direct correlation with the classification of the MASA proposed by Dixmier \cite{dixmier}. This shows, in particular, that no non-trivial interpretation exists if the MASA is \emph{singular}, and that any exponential connective can be interpreted if the MASA is \emph{regular} -- therefore one can soundly interpret elementary linear logic. This section shows also that in the intermediate case -- that of \emph{semi-regular} MASAs -- one can at least interpret multiplicative-additive linear logic but no general statement can be made concerning the interpretation of exponential connectives.


\section{von Neumann Algebras and MASAs}

\subsection{First Definitions and Results}

The theory of von Neumann algebras, under the name of \enquote{rings of operators}, was first developed by Murray and von Neumann in a series of seminal papers \cite{vonneumannbicommutant,murrayvn1,murrayvn2,vn1938,vn1940,murrayvn3,vn1943,vn1949}.

\subsubsection{The double commutant theorem.} A normed $\ast$-algebra is a normed algebra endowed with an antilinear isometric involution $(\cdot)^{\ast}$ which reverses the product:
$$(t^{\ast})^{\ast}=t~~~~\norm{t^{\ast}}=\norm{t}~~~~(t+u)^{\ast}=t^{\ast}+u^{\ast}~~~~ (\lambda t)^{\ast}=\bar{\lambda} t^{\ast}~~~~(tu)^{\ast}=u^{\ast}t^{\ast}$$
A normed $\ast$-algebra is a C$^{\ast}$-algebra when it is complete (i.e.\ it is a Banach algebra) and satisfies the \emph{C$^{\ast}$-identity} $\norm{t^{\ast}t}=\norm{t}^{2}$.

We denote by $\B{\hil{H}}$ the $\ast$-algebra of continuous (or equivalently, bounded) linear maps from the Hilbert space $\hil{H}$ to itself. This algebra can be endowed with the following three topologies:
\begin{itemize}
\item The norm topology, for which a net $(T_{\lambda})$ converges toward $0$ when the net $\norm{T_{\lambda}}$ converges to $0$ in $\complexN$;
\item The strong operator topology (SOT) which is the topology of pointwise convergence when $\hil{H}$ is considered endowed with its norm topology: a net $(T_{\lambda})$ converges toward $0$ when for all $\zeta\in \hil{H}$, the net $(\norm{T_{\lambda}\zeta})$ converges towards $0$ in $\complexN$;
\item The weak operator topology (WOT) which is the topology of pointwise convergence when $\hil{H}$ is considered endowed with its weak topology: a net $(T_{\lambda})$ converges toward $0$ when for all $\zeta,\eta \in \hil{H}$, the net $(\inner{T_{\lambda}\zeta}{\eta})$ converges towards $0$ in $\complexN$;
\end{itemize}

\begin{definition}[von Neumann algebra]
A von Neumann algebra is a $\ast$-sub-algebra of $\B{\hil{H}}$ which is closed for the strong operator topology (SOT).
\end{definition}

We now explain Murray and von Neumann's fundamental \enquote{double commutant theorem}. Pick $M\subset \B{\hil{H}}$. We define the commutant of $M$ (in $\B{\hil{H}}$) as the set $M'^{\B{\hil{H}}}=\{x\in\B{\hil{H}}~|~\forall m\in M, mx=xm\}$. We will in general omit to precise the ambient algebra and denote abusively $M'$ the commutant of $M$ if the context is sufficiently clear. We will denote by $M''$ the bi-commutant $(M')'$ of $M$.

The following theorem is the keystone of the von Neumann algebras theory. It is particularly elegant, since it shows an equivalence between a purely algebraic notion -- being equal to its bi-commutant -- and a purely topological notion -- being closed for the strong operator topology.

\begin{theorem}[Double Commutant Theorem, von Neumann \cite{vonneumannbicommutant}]
Let $M$ be a $\ast$-sub-algebra of $\B{\hil{H}}$ such that $1_{\B{\hil{H}}}\in M$. Then $M$ is a von Neumann algebra if and only if $M=M''$.
\end{theorem}

\begin{remark} Since the strong operator topology (SOT) is weaker than the norm topology, a von Neumann algebra $\vn{M}$ is also closed for the norm topology, and is also a C$^{\ast}$-algebra. Moreover, since $\vn{M}$, as a von Neumann algebra, is the commutant of a set of operators, it necessarily contains the identity operator in $\B{\hil{H}}$, and consequently is a unital C$^{\ast}$-algebra. One can therefore define the continuous spectral calculus for operators in $\vn{M}$.
\end{remark}

\subsubsection{Direct Integrals.}

Let $\vn{M}$ be a von Neumann algebra. We define the \emph{center} of $\vn{M}$ as the von Neumann algebra $\vn{Z(M)}=\vn{M\cap M}'$. 
\begin{definition}[Factor]
A \emph{factor} is a von Neumann algebra $\vn{M}$ whose center is trivial, i.e. such that $\vn{Z(M)}=\complexN.1_{\B{\hil{H}}}$.
\end{definition}

The study of von Neumann algebras can be reduced to the study of factors. This is one of the most important results of the theory, which is due to von Neumann \cite{vn1949}: he showed that every von Neumann algebra can be written as a \emph{direct integral} of factors. A direct integral is a direct sum over a continuous index set, in the same way an integral is a sum over a continuous index set. A complete exposition of this result can be found in the first book of the Takesaki series  \cite{takesaki1}, Section $\text{IV}$.8, page 264.

Here are the main ideas. If $\vn{A}$ is not a factor, its center $\vn{Z(A)}$ is a non-trivial commutative von Neumann algebra (i.e. different from $\complexN$). Suppose now that $\vn{Z(A)}$ is a \emph{diagonal} algebra, i.e. that there exists a countable set $I$ (which could be finite) and a family $(p_{i})_{i\in I}$ of pairwise disjoint minimal projections such that $\sum_{i\in I}p_{i}=1$. Then the algebras $p_{i}\vn{A}p_{i}$ are factors, and one has $\vn{A}=\bigoplus_{i\in I}p_{i}\vn{A}p_{i}$. However, in the general case, the center $\vn{Z(A)}$ does not need to be a diagonal algebra, and it can contain a \emph{diffuse} sub-algebra, i.e. a sub-algebra that does not have minimal projections. Then it is necessary to consider a continuous version of the direct sum: the direct integral.

\begin{definition}
Let $(X,\mathcal{B},\mu)$ be a measured space. A family $(\hil{H}_{x})_{x\in X}$ of Hilbert spaces is \emph{measurable} over $(X,\mathcal{B},\mu)$ when there exists a countable partition $(X_{i})_{i\in I}$ of $X$ such that for all $i\in I$:
\begin{equation*}
\exists \hil{K},~ \forall x\in X_{i}, \hil{H}_{x}=\hil{K}
\end{equation*}
where $\hil{K}$ is either equal to $\complexN^{n}$ ($n\in\naturalN$) or equal to $\ell^{2}(\naturalN)$.

A section $(\xi_{x})_{x\in X}$ ($\xi_{x}\in\hil{H}_{x}$) is measurable when its restriction to each element $X_{n}$ of the partition is measurable.
\end{definition}

\begin{definition}
Let $(\hil{H}_{x})_{x\in X}$ be a measurable family of Hilbert spaces over a measured space $(X,\mathcal{B},\lambda)$. The \emph{direct integral} $\dint_{X} \hil{H}_{x}d\lambda(x)$ is the Hilbert space whose elements are equivalence classes of measurable sections modulo almost everywhere equality, and the scalar product is defined by:
\begin{equation*}
\inner{(\xi_{x})_{x\in X}}{(\zeta_{x})_{x\in X}}=\dint_{X}\inner{\xi_{x}}{\zeta_{x}}d\lambda(x)
\end{equation*}
\end{definition}

In the same way commutative C$^{\ast}$-algebras are exactly the algebras of continuous functions from locally compact Hausdorff spaces to $\complexN$ (this is Gelfand's theorem, \cite{gelfand}), one can show that every commutative von Neumann algebra can be identified with the algebra $L^{\infty}(X,\mathcal{B},\lambda)$ of essentially bounded measurable functions on a measured space $(X,\mathcal{B},\lambda)$.

\begin{theorem}
Let $\vn{A}$ be a commutative von Neumann algebra. There exists a measurable family of Hilbert spaces $(\hil{H}_{x})_{x\in X}$ over a measured space $(X,\mathcal{B},\lambda)$ such that $\vn{A}$ is unitarily equivalent to the algebra $L^{\infty}(X)$ acting on the Hilbert space $\dint_{X} \hil{H}_{x}d\lambda(x)$.
\end{theorem}

We will not define here neither the notion of measurable family of von Neumann algebras, nor the one of direct integrals of von Neumann algebras. We only state the fundamental theorem mentioned above.

\begin{theorem}[von Neumann \cite{vn1949}, Takesaki \cite{takesaki2}, Theorem $\text{IV}$.8.21 page 275]
Every von Neumann algebra can be written as a direct integral of factors.
\end{theorem}

\subsubsection{Classification of factors.} The study of factors led to a classification based on the study of the set of projections and their isomorphisms (partial isometries). We recall that a projection is an operator $p$ such that $p=p^{\ast}=p^{2}$ (this is sometimes referred to as an \enquote{orthogonal projection}). If $\vn{M}$ is a von Neumann algebra, we will denote by $\Pi(\vn{M})$ the set of projections in $\vn{M}$. Since $\vn{M}$ is a sub-algebra of $\B{\hil{H}}$ for a given Hilbert space $\hil{H}$, the projections in $\Pi(\vn{M})$ are in particular projections in $\B{\hil{H}}$. As such, they are in correspondence with subspaces of $\hil{H}$: the projection $p$ corresponds to the closed subspace $p\hil{H}$. Two projections $p,q$ are \emph{disjoint} when $pq=0$, translating the fact that the two corresponding closed subspaces $p\hil{H}$ and $q\hil{H}$ are disjoint. Moreover, the set $\Pi(\vn{M})$ is endowed with a partial ordering inherited from the inclusion of subspaces: $p\preceq q$ if and only if $pq=p$ if and only if $p\hil{H}\subset q\hil{H}$.

Now, the idea of Murray and von Neumann \cite{murrayvn1} was to consider an equivalence relation on the set of projections. This equivalence relation depends on the algebra $\vn{M}$ and translates the fact that $\vn{M}$ contains an isomorphism between the corresponding subspaces. Namely, they define the equivalence as follows: two projections $p,q$ are \emph{Murray von Neumann equivalent in $\vn{M}$}, noted $p\sim_{\vn{M}} q$, when there exists an element $u\in \vn{M}$ such that $uu^{\ast}=p$ and $u^{\ast}u=q$. Notice that this implies that $u$ is a partial isometry.

The partial ordering $\preceq$ then induces a partial ordering $\precsim_{\vn{M}}$ on the equivalence classes of projections in $\vn{M}$, i.e. on the set $\Pi(\vn{M})/\sim_{\vn{M}}$.

\begin{remark}\label{remarqueprecalgopvn}
As we explained above, $p\preceq q$ means that $p\hil{H}$ is a closed subspace of $q\hil{H}$. The fact that $p\sim_{\vn{M}} q$ translates the fact that $p\hil{H}$ and $q\hil{H}$ are \emph{inner (w.r.t $\vn{M}$) isomorphic}, i.e. there exists an isomorphism between them which is an element of $\vn{M}$, or in other terms, the fact that they are isomorphic is witnessed by an element of $\vn{M}$. Consequently, the fact that $p\precsim_{\vn{M}} q$ translates the idea that $p\hil{H}$ is \emph{inner isomorphic} to a closed subspace of $q\hil{H}$, and therefore that $p\hil{H}$ is somehow \emph{smaller} than $q\hil{H}$ in the sense that an element of $\vn{M}$ witnesses the fact that it is smaller.
\end{remark}

\begin{definition}
A projection $p$ in a von Neumann algebra $\vn{M}$ is \emph{infinite} (in $\vn{M}$) when there exists $q\prec p$ (i.e. a proper sub-projection) such that $q\sim_{\vn{M}} p$. A projection is \emph{finite} (in $\vn{M}$) when it is not infinite (in $\vn{M}$).
\end{definition}

\begin{proposition}[Takesaki \cite{takesaki1}, Proposition $\text{V}$.1.3 page 291 and Theorem $\text{V}$.1.8 page 293]
Let $\vn{M}$ be a von Neumann algebra. Then $\vn{M}$ is a factor if and only if the relation $\precsim_{\vn{M}}$ is a total ordering.
\end{proposition}

To state the following definition and theorem, we will use a slight variant of the usual notion of order type: we distinguish the element denoted by $\infty$ from any other element, considering that $\infty$ represents a class of infinite projections. For instance, $\{0,1\}$ and $\{0,\infty\}$ should be considered as distinct since the first does not contain infinite elements contrarily to the second.

\begin{definition}[Type of a Factor]
Let $\vn{M}$ be a factor. We will say that:
\begin{itemize}
\item $\vn{M}$ is of type $\text{I}_{n}$ when $\precsim_{\vn{M}}$ has the same order type as $\{0,1,\dots,n\}$;
\item $\vn{M}$ is of type $\text{I}_{\infty}$ when $\precsim_{\vn{M}}$ has the same order type as $\naturalN\cup\{\infty\}$;
\item $\vn{M}$ is of type $\text{II}_{1}$ when $\precsim_{\vn{M}}$ has the same order type as $[0,1]$;
\item $\vn{M}$ is of type $\text{II}_{\infty}$  when $\precsim_{\vn{M}}$ has the same order type as $\realposN\cup\{\infty\}$;
\item $\vn{M}$ is of type $\text{III}$  when $\precsim_{\vn{M}}$ has the same order type as $\{0,\infty\}$, i.e. all non-zero projections are infinite.
\end{itemize}
\end{definition}

\begin{proposition}
There exists factors of all types. Moreover, $\precsim_{\vn{M}}$ cannot be of another order type as the ones listed above.
\end{proposition}

\begin{proof}
Existence of type {I} factors is clear; the algebra $\B{\hil{H}}$ with $\hil{H}$ a Hilbert space of dimension $k$ ($k\in\naturalN^{\ast}\cup\{\infty\}$) is a type {I}$_{k}$ factor. For the existence of type {II} and type {III} factors, we refer to the first volume of Takesaki's series \cite{takesaki1}, section $\text{V}$.7, page 362. For the second part of the proposition, we refer once again to the first volume of Takesaki's series \cite{takesaki1}, Theorem $\text{V}$.1.19 and Corollary $\text{V}$.1.20 pages 296-297.
\end{proof}

We can show that a factor of type $\text{I}_{n}$ is isomorphic to $\vn{M}_{n}(\complexN)$, the algebra of square matrices of size $n\times n$ with complex coefficients. A factor of type $\text{I}_{\infty}$ is isomorphic to $\B{\hil{H}}$, where $\hil{H}$ is an infinite-dimensional Hilbert space.

We will now define the notion of \emph{trace}. One of the important properties of type {II}$_{1}$ factors is the existence of a faithful normal finite trace, i.e. an adequate generalisation of the trace of matrices. Traces, in general, are not defined for all elements, but only for \emph{positive elements}.

\begin{definition}
Let $a$ be an operator in $\vn{M}$ a von Neumann algebra (more generally, a C$^{\ast}$-algebra). We say that $a$ is \emph{positive} if $\text{Spec}_{\vn{M}}(a)\subset \mathbb{R}_{+}$. We denote by $\vn{M}^{+}$ the set of positive operators in $M$.
\end{definition}

\begin{proposition}
We have $\vn{M}^{+}=\{u^{\ast}u~|~u\in \vn{M}\}$.
\end{proposition}

\begin{definition}
A \emph{trace} $\tau$ on a von Neumann algebra $\vn{M}$ is a function from $\vn{M}^{+}$ into $[0,\infty]$ satisfying:
\begin{enumerate}
\item $\tau(x+y)=\tau(x)+\tau(y)$ for all $x,y\in \vn{M}^{+}$;
\item $\tau(\lambda x)=\lambda\tau(x)$ for all $x\in \vn{M}^{+}$ and all $\lambda\geqslant 0$;
\item $\tau(x^{\ast}x)=\tau(xx^{\ast})$ for all $x\in \vn{M}$.
\end{enumerate}
We will moreover say that $\tau$ is:
\begin{itemize}
\item \emph{faithful} if $\tau(x)>0$ for all $x\not= 0$ in $\vn{M}^{+}$;
\item \emph{finite} when $\tau(1)<\infty$;
\item \emph{semi-finite} when for all element $x$ in $\vn{M}^{+}$ there exists $y\in\vn{M}^{+}$ such that $x-y\in \vn{M}^{+}$ and $\tau(y)<\infty$;
\item \emph{normal} when $\tau(\sup \{x_{i}\})=\sup \{\tau(x_{i})\}$ for all increasing bounded net $\{x_{i}\}$ in $\vn{M}^{+}$.
\end{itemize}
\end{definition}

\begin{theorem}[Takesaki \cite{takesaki1}, Theorem $\text{V}$.2.6 page 312]
If $\vn{M}$ is a finite factor (i.e. the identity is a finite projection), then there exists a faithful normal finite trace $\tau$. Moreover, every other faithful normal finite trace $\rho$ is proportional to $\tau$.

If $\vn{M}$ is of type $\text{II}_{1}$, we will refer to the unique faithful normal finite trace $\tr$ such that $\tr(1)=1$ as the \emph{normalised trace}.
\end{theorem}

\begin{remark} Since the set of positive operators in $\vn{M}$ generates the von Neumann algebra $\vn{M}$, a finite trace $\tau$ extends uniquely to a positive linear form on $\vn{M}$ that we will abusively write $\tau$ as well. In particular, every operator $a$ in a type $\text{II}_{1}$ factor has a finite trace.
\end{remark}

In order to define the notion of hyperfiniteness we need to define yet another topology on $\B{\hil{H}}$, the so-called $\sigma$-weak topology. This definition is based upon the notion of weak$^{\ast}$ topology: if $X$ is a space and $X^{\ast}$ is its dual, then the weak$^{\ast}$ topology on $X^{\ast}$ is defined as the topology of pointwise convergence on $X$. To define the $\sigma$-weak topology on $\B{\hil{H}}$ as a weak$^{\ast}$ topology, we moreover need to see $\B{\hil{H}}$ as the dual space of some other space. This is a well-known result which can be found in standard textbooks: the algebra $\B{\hil{H}}$ is the dual of the space of \emph{trace-class operators} that we will denote $\B{\hil{H}}_{\ast}$ and which is itself the dual space of the algebra of compact operators.

\begin{definition}\label{deftoposigmafaible}
Let $\hil{H}$ be a Hilbert space. The $\sigma$-weak topology on $\B{\hil{H}}$ is defined as the weak$^{\ast}$ topology induced by the predual $\B{\hil{H}}_{\ast}$ of $\B{\hil{H}}$.
\end{definition}

\begin{remark}
If $\hil{H}$ is an infinite-dimensional separable Hilbert space, $\B{\hil{H}}$ embeds into $\B{\hil{H\otimes H}}$ through the morphism $x\mapsto x\otimes 1$. One can show that the restriction of the weak operator topology (WOT) on $\B{\hil{H\otimes H}}$ coincides with the $\sigma$-weak topology on $\B{\hil{H}}$.
\end{remark}

\begin{definition}
A von Neumann algebra $\vn{M}$ is \emph{hyperfinite} if there exists a directed family $\vn{M}_{i}$ of finite-dimensional $\ast$-sub-algebras of $\vn{M}$ such that $\cup_{i}\vn{M}_{i}$ is dense in $\vn{M}$ for the $\sigma$-weak topology. 
\end{definition}

\begin{theorem}[Takesaki \cite{takesaki3}, Theorem $\text{XIV}$.2.4 page 97]
Two hyperfinite type $\text{II}_{1}$ factors are isomorphic. We will write $\finhyp$ the unique hyperfinite type $\text{II}_{1}$ factor.
\end{theorem}

\begin{theorem}[Takesaki \cite{takesaki3}, Theorem $\text{XVI}$.1.22 page 236]
Two hyperfinite type $\text{II}_{\infty}$ factors are isomorphic. In particular, they are isomorphic to the tensor product $\infhyp=\B{\hil{H}}\otimes \finhyp$.
\end{theorem}

\subsubsection{Sakai's Theorem and W$^{\ast}$-algebras.}\label{theoremsakai}

We have defined above the von Neumann algebras as sub-algebras of $\B{\hil{H}}$ where $\hil{H}$ is a separable Hilbert space. We therefore defined a von Neumann algebra as a \enquote{concrete} algebra, i.e. as a set of operators acting on a given space. As it is the case with C$^{\ast}$-algebras, which can be defined either concretely as a norm-closed sub-algebra of $\B{\hil{H}}$ or abstractly as an involutive Banach algebra satisfying the C$^{\ast}$-identity, there exists an abstract definition of von Neumann algebras. This important result is due to Sakai.

\begin{definition}
Let $\vn{M}$ be a von Neumann algebra. The \emph{pre-dual} $\vn{M}_{\ast}$ of $\vn{M}$ is the set of linear forms\footnote{We recall that a \emph{linear form on a vector space $V$} is a linear map from $V$ into $\complexN$, i.e. an element of the dual of $V$. When $V$ is a topological vector space, the elements of the topological dual of $V$ are therefore the continuous linear forms.} which are continuous for the $\sigma$-weak topology (\autoref{deftoposigmafaible}).
\end{definition}

\begin{proposition}[Takesaki \cite{takesaki1}, Theorem $\text{II}$.2.6, page 70]
Let $\vn{M}$ be a von Neumann algebra. There exists an isometric isomorphism between $\vn{M}$ and $(\vn{M}_{\ast})^{\ast}$ -- the dual (as a Banach space) of the pre-dual of $\vn{M}$.
\end{proposition}

The reciprocal statement was proved by Sakai \cite{sakai} and gives an exact characterisation of von Neumann algebras among C$^{\ast}$-algebras. A proof can be found in Takesaki \cite{takesaki1}, Theorem 3.5, page 133, and Corollary 3.9, page 135.

\begin{theorem}\label{theoremesakai}
A C$^{\ast}$-algebra $\vn{A}$ is a von Neumann algebra if and only if there exists a Banach algebra $B$ such that $\vn{A}$ is the dual of $B$: $\vn{A}=B^{\ast}$. The algebra $B$ is moreover unique (up to isomorphism).
\end{theorem}

One can then define von Neumann algebras \emph{abstractly}, i.e. as an abstract algebra \emph{vs} as an algebra of operators acting on a specific space. Such abstract algebras can then be \emph{represented} as algebras of operators.

\begin{definition}\label{representations}
A \emph{representation} of a von Neumann algebra $\vn{M}$ is a pair $(\hil{H},\pi)$ where $\pi:\vn{M}\rightarrow \B{\hil{H}}$ is a C$^{\ast}$-algebra homomorphism. If the homomorphism $\pi$ is injective, we say the associated representation is \emph{faithful}.
\end{definition}

\subsubsection{The Standard Representation.} One of the major results in the theory of von Neumann algebras is that every such algebra has a \enquote{standard representation}, i.e. a representation that satisfies a number of important properties. Namely, once realised that von Neumann algebras can be defined in an abstract way, the next step is to identify them with particularly satisfying concrete algebras. A proof of the following result can be found in Takesaki \cite{takesaki2}, Section $\text{IX}$.1, page 142.

\begin{theorem}[Haagerup \cite{haagerup}]\label{haagerup}
Let $\vn{M}$ be a von Neumann algebra. Then there exists a Hilbert space $\hil{H}$, a von Neumann algebra $\vn{S}\subset\B{\hil{H}}$, an isometric antilinear involution $J:\hil{H}\rightarrow\hil{H}$ and a cone $\vn{P}$ closed under $(\cdot)^{\ast}$ such that:
\begin{itemize}
\item $\vn{M}$ and $\vn{S}$ are isomorphic;
\item $J\vn{M}J=\vn{M}'$;
\item $JaJ=a^{\ast}$ for all $a\in\vn{Z(M)}$;
\item $Ja=a$ for all $a\in\vn{P}$;
\item $aJaJ\vn{P}=\vn{P}$ for all $a\in\vn{M}$.
\end{itemize}
The tuple $(\vn{S},\hil{H},J,\vn{P})$ is called the \emph{standard form} of the algebra $\vn{M}$.
\end{theorem}

Let us work out the case of a von Neumann algebra $\vn{M}$ endowed with a faithful normal semi-finite trace. In this case, we can describe a quite easy construction of the standard form of $\vn{M}$. We first define the ideal $\mathfrak{n}_{\tau}=\{x\in\vn{M}~|~\tau(x^{\ast}x)<\infty\}$ (notice that in the case of a finite algebra $\mathfrak{n}_{\tau}=\vn{M}$). We then consider the map $\innerbis{\cdot}{\cdot}$ from $\vn{M}$ to real numbers defined by:
\begin{equation*}
\innerbis{x}{y}=\tau(y^{\ast}x)
\end{equation*}
From the linearity of the trace and the anti-linearity of the involution, we can show that it is a sesquilinear form. Moreover, since $x^{\ast}x$ is a positive operator, we know that $\tau(x^{\ast}x)\geqslant 0$. Therefore, this defines a scalar product on $\mathfrak{n}_{\tau}$, and we can now define the Hilbert space $L^{2}(\vn{M},\tau)$ as the completion of $\mathfrak{n}_{\tau}$ ($\vn{M}$ when the algebra is finite) for the norm defined by $\norm{x}_{2}=\tau(x^{\ast}x)^{\frac{1}{2}}$.

One can then show that for every element $a\in\vn{M}$ and every $x\in\mathfrak{n}_{\tau}$:,
\begin{eqnarray*}
\norm{ax}_{2}&\leqslant&\norm{a}\norm{x}_{2}\\
\norm{xa}_{2}&\leqslant&\norm{a}\norm{x}_{2}
\end{eqnarray*}
We then denote by $\pi_{\tau}$ (resp. $\pi'_{\tau}$) the representation of $\vn{M}$ onto $L^{2}(\vn{M},\tau)$ by left (resp. right) multiplication.

We then notice that the operation $(\cdot)^{\ast}$ defines an isometry on $\mathfrak{n}_{\tau}$ for the norm $\norm{\cdot}_{2}$. It thus extends to an antilinear involution $J:L^{2}(\vn{M},\tau)\rightarrow L^{2}(\vn{M},\tau)$. One then shows that:
\begin{itemize}
\item $\pi_{\tau}$ (resp. $\pi'_{\tau}$) is a faithful representation (resp. antirepresentation\footnote{An antirepresentation is a representation that inverses multiplication: $\pi'_{\tau}(xy)=\pi'_{\tau}(y)\pi'_{\tau}(x)$.});
\item $\pi_{\tau}(\vn{M})'=\pi'_{\tau}(\vn{M})$ and $\pi'_{\tau}(\vn{M})'=\pi_{\tau}(\vn{M})$;
\item $J\pi_{\tau}(a)J=\pi_{\tau}'(a^{\ast})$ for all $a\in\vn{M}$.
\end{itemize}

\subsection{Maximal Abelian Sub-Algebras}

The purpose of this paper is to exhibit a remarkable correspondence between a classification of maximal abelian sub-algebras and fragments of linear logic. We will therefore need a number of definitions and results about such sub-algebras. The purpose of this section is to provide those together with a number of intuitions that should help the reader to grasp some subtleties of the theory. After defining what exactly is a maximal abelian sub-algebra, we will start by explaining the classification of such in type {I} factors, the simpler case. We will then go on with the case of type {II} algebras which is more involved.

\subsubsection{MASAs in type ${I}$ factors.} 

\begin{definition}
Let $\vn{M}$ be a von Neumann algebra. A \emph{maximal abelian sub-algebra} (MASA) $\vn{A}$ of $\vn{M}$ is a von Neumann sub-algebra of $\vn{M}$ such that for all intermediate sub-algebras $\vn{B}$, i.e. $\vn{A}\subset\vn{B}\subset\vn{M}$, if $\vn{B}$ is abelian then $\vn{A}=\vn{B}$.
\end{definition}

If $\vn{A}$ and $\vn{B}$ are MASAs in a von Neumann algebra $\vn{M}$, they can be " isomorphic" in three different ways: 
\begin{itemize}
\item they can be isomorphic as von Neumann algebras -- this is the weakest notion;
\item there can exists an automorphism $\Phi$ of $\vn{M}$ such that $\Phi(\vn{A})=\vn{B}$; we then say that $\vn{A}$ and $\vn{B}$ are conjugated;
\item there can exist a unitary operator\footnote{We recall that a unitary operator is an operator $u$ such that $uu^{\ast}=u^{\ast}u=1$.} $u\in\vn{M}$ such that $u\vn{A}u^{\ast}=\vn{B}$ -- this is the strongest notion; we then say that $\vn{A}$ and $\vn{B}$ are unitarily equivalent.
\end{itemize}

Let us quickly discuss the finite-dimensional case. We fix $\hil{H}$ a finite-dimensional Hilbert space of dimension $k\in\naturalN$. Then $\B{\hil{H}}$ is isomorphic to the algebra of $k\times k$ matrices. Picking a basis $\mathcal{B}=(b_{1},\dots,b_{k})$ of $\hil{H}$, one can define the sub-algebra $\vn{D}_{\mathcal{B}}$ of $\B{\hil{H}}$ containing all diagonal matrices in the basis $\mathcal{B}$. This algebra is obviously commutative, and it is moreover maximal as a commutative sub-algebra of $\B{\hil{H}}$: if $\vn{A}$ is a commutative sub-algebra of $\B{\hil{H}}$ containing $\vn{D}_{\mathcal{B}}$, then $\vn{A}=\vn{D}_{\mathcal{B}}$. A more involved argument shows that any maximal abelian sub-algebra of $\B{\hil{H}}$ is the diagonal algebra induced by a basis; this result is also a direct corollary of \autoref{masasbh}. These algebras $\vn{D}_{\mathcal{B}}$ where $\mathcal{B}$ is a basis of $\hil{H}$ are clearly pairwise isomorphic, as it suffices to map bijectively the bases one onto the other. They are in fact unitarily equivalent, as such a bijection induces a unitary operator. This shows that the distinctions we just made are useless in the finite-dimensional case: all MASAs are unitarily equivalent.

We will now state a classification result about maximal abelian sub-algebras of $\B{\hil{H}}$, which gives a complete answer to the classification problem of MASAs in type {I} factors. This theorem will be preceded by a proposition showing that all \emph{diffuse} MASAs in $\B{\hil{H}}$ are unitarily equivalent; this will be of use later on, as those MASAs of a type {II} factor $\vn{N}\subset\B{\hil{H}}$ which are also MASAs of $\B{\hil{H}}$ are necessarily diffuse.

\begin{proposition}[Sinclair and Smith \cite{FiniteVNAandMasas}]\label{masasbh}
Let $\vn{A}$ be a MASA of $\B{\hil{H}}$ which do not have (non-zero) minimal projections -- we say in this case that $\vn{A}$ is a diffuse MASA. Then there exists a unitary $U:\hil{H}\rightarrow L^{2}([0,1])$ such that $U\vn{A}U^{\ast}=L^{\infty}([0,1])$.
\end{proposition}

\begin{theorem}[Sinclair and Smith \cite{FiniteVNAandMasas}]\label{classificationmasasbh}
Let $\vn{A}$ be a MASA in $\B{\hil{H}}$. Then:
\begin{itemize}
\item either $\vn{A}$ is unitarily equivalent to $L^{\infty}([0,1])$ (diffuse case);
\item either $\vn{A}$ is unitarily equivalent to $\vn{D}$, a diagonal algebra (discrete case);
\item either $\vn{A}$ is unitarily equivalent to $\vn{D}\oplus L^{\infty}([0,1])$, where $\vn{D}$ is a diagonal algebra (mixed case);
\end{itemize}
\end{theorem}

Things are therefore clear concerning the MASAs in $\B{\hil{H}}$, as the previous theorem provides a complete classification of those. In the case of von Neumann algebras of type $\text{II}_{1}$ however, things are more complicated and such a complete classification does not exist in spite of the numerous works on the subject. 

\subsubsection{Dixmier's Classification.} We begin the discussion about MASAs of type $\text{II}_{1}$ von Neumann algebras by explaining Dixmier's classification \cite{dixmier}, which considers the algebra generated by the \emph{normaliser} of the MASA. Let us stress that this classification is not exhaustive. This presentation of Dixmier's classification will also give us the opportunity to state some results that will be of use in the rest of the paper.

\begin{definition}[normaliser]
Let $\vn{M}$ be a von Neumann algebra, and $\vn{A}$ a von Neumann sub-algebra of $\vn{M}$. We will denote by $\nor[\vn{M}]{\vn{A}}$ the \emph{normaliser} of $\vn{A}$ in $\vn{M}$ which is defined as:
\begin{equation*}
\nor[\vn{M}]{\vn{A}}=\{u\in\vn{M}~|~u\text{ unitary, }u\vn{A}u^{\ast}=\vn{A}\}
\end{equation*}
We will denote by $\noralg[\vn{M}]{\vn{A}}$ the von Neumann algebra generated by $\nor[\vn{M}]{\vn{A}}$.
\end{definition}

\begin{definition}[normalising Groupoid]\label{groupoidenormalisant}
Let $\vn{M}$ be a von Neumann algebra and $\vn{A}$ be a von Neumann sub-algebra of $\vn{M}$. We will denote by $\gn[\vn{M}]{\vn{A}}$ the \emph{normalising groupoid} of $\vn{A}$ in $\vn{M}$ which is defined as:
\begin{equation*}
\gn[\vn{M}]{\vn{A}}=\{u\in\vn{M}~|~uu^{\ast}u=u,uu^{\ast}\in\vn{A}, u^{\ast}u\in\vn{A}, u\vn{A}u^{\ast}\subset\vn{A}\}
\end{equation*}
We will denote by $\gnalg[\vn{M}]{\vn{A}}$ the von Neumann algebra generated by $\gn[\vn{M}]{\vn{A}}$.
\end{definition}

\begin{definition}[Dixmier Classification]\label{classification}
Let $\vn{M}$ be a factor, and $\vn{P}$ a MASA in $\vn{M}$. We distinguish three cases:
\begin{enumerate}
\item if $\noralg[\vn{M}]{\vn{P}}=\vn{M}$,  we say that $\vn{P}$ is \emph{regular} (or \emph{Cartan});
\item if $\noralg[\vn{M}]{\vn{P}}=\vn{K}$, where $\vn{K}$ is a factor distinct from $\vn{M}$, we say that $\vn{P}$ is \emph{semi-regular};
\item if $\noralg[\vn{M}]{\vn{P}}=\vn{P}$, we say that $\vn{P}$ is \emph{singular}.
\end{enumerate}
\end{definition}

The following four theorems can be found in the literature. The first two theorems can be found along with their proofs in Sinclair and Smith book \cite{FiniteVNAandMasas} about MASAs in finite factors. The third is a quite recent generalisation \cite{Chifan} of a result which was previously known to hold for singular MASAs.

\begin{theorem}[Dye, \cite{dye63}]\label{dyestheorem}
Let $\vn{M}$ be a von Neumann algebra with a faithful normal trace, and $\vn{A}$ a MASA in $\vn{M}$. Then the set $\gn[\vn{M}]{\vn{A}}$ is contained in the sub-vector space of $\vn{M}$ generated by $\nor[\vn{M}]{\vn{A}}$.
\end{theorem}

\begin{corollary}
Under the hypotheses of the preceding theorem, the von Neumann algebras $\noralg[\vn{M}]{\vn{A}}$ and $\gnalg[\vn{M}]{\vn{A}}$ are equal.
\end{corollary} 

\begin{theorem}[Jones and Popa, \cite{popajones}]\label{popajones}
Let $\vn{M}$ be a type $\text{II}_{1}$ factor, and $\vn{A}$ a MASA in $\vn{M}$. Let $\proj{p,q}\in\vn{A}$ be projections of equal trace. Then, if $\noralg[\vn{M}]{\vn{A}}$ is a factor, there exists a partial isometry $v_{0}\in \gn[\vn{M}]{\vn{A}}$ such that $\proj{p}=v_{0}v_{0}^{\ast}$ and $\proj{q}=v_{0}^{\ast}v_{0}$.
\end{theorem}

\begin{theorem}[Chifan, \cite{Chifan}]\label{chifan}
Let $\vn{M}_{1}$ and $\vn{M}_{2}$ be type $\text{II}_{1}$ factors. For $i=1,2$, let $\vn{A}_{i}$ be a MASA in $\vn{M}_{i}$. Then: 
\begin{equation*}
\noralg[\vn{M}_{1}\vntimes\vn{M}_{2}]{\vn{A}_{1}\vntimes\vn{A}_{2}}=\noralg[\vn{M}_{1}]{\vn{A}_{1}}\vntimes\noralg[\vn{M}_{2}]{\vn{A}_{2}}
\end{equation*}
\end{theorem}

\begin{theorem}[Connes, Feldman and Weiss \cite{connes}]\label{connes}
Let $\vn{A,B}$ be two regular MASAs of the hyperfinite factor $\finhyp$ of type $\text{II}_{1}$. Then $\vn{A}$ and $\vn{B}$ are unitarily equivalent.
\end{theorem}

\subsubsection{Pukansky's Invariant.} Pukansky \cite{pukansky} defined a numerical invariant for MASAs $\vn{A}$ of a type $\text{II}_{1}$ factor $\vn{N}$. Consider that $\vn{N}$ is endowed with a faithful normal trace $\tau$, and let $J$ be the anti linear isometry $Jx = x^{\ast}$ onto $L^{2}(\vn{N})$. Pukansky's invariant is based on the type $\text{I}$ decomposition of $(\vn{A}\cup J\vn{A}J)'$. Indeed this algebra, as the commutant of an abelian algebra, is of type {I} and therefore can be decomposed as a sum of factors of type ${I}_{n}$ (where $n$ might be equal to $\infty$). The Pukansky invariant is then essentially the set of all values of $n$ that appear in this decomposition.

  The following lemma justifies the definition of Pukansky's invariant. We define $e_{\vn{A}}$ as the projection of $L^{2}(\vn{N})$ onto $L^{2}(\vn{A})$ and we will write $\mathfrak{B}_{\vn{A}}$ the commutative algebra generated by $\vn{A}\cup J\vn{A} J$.

\begin{lemma}[Sinclair and Smith \cite{FiniteVNAandMasas}, Chapter 7] Let $\vn{N}$ be a type $\text{II}_{1}$ factor represented onto $L^{2}(\vn{N})$ and $\vn{A}$ a MASA in $\vn{N}$. Then $e_{\vn{A}} \in \mathfrak{B}_{\vn{A}}$ and $e_{\vn{A}}$ is a central projection -- i.e. a projection onto the center of the algebra -- in $\mathfrak{B}_{\vn{A}}'$.
\end{lemma}

\begin{definition}
Let $\vn{A}$ be a MASA in a factor $\vn{N}$ of type $\text{II}_{1}$. We define the \emph{Pukansky invariant} $\puk{\vn{A},\vn{N}}$ of $\vn{A}$ in $\vn{N}$ -- usually denoted by $\puk{\vn{A}}$ when the context is clear -- as the set of all natural numbers $n\in\naturalN\cup\{\infty\}$ such that $(1-e_{\vn{A}})\mathfrak{B}_{\vn{A}}'$ has a non-zero type $\text{I}_{n}$ part.
\end{definition}

By removing the projection $e_{\vn{A}}$ from $\mathfrak{B}_{\vn{A}}'$, we are erasing the part $\mathfrak{B}_{\vn{A}}'e_{\vn{A}}=\mathfrak{B}_{\vn{A}}e_{\vn{A}}$ which is abelian for all MASA $\vn{A}$. This allows for a better invariant since its inclusion would add the integer $1$ to all Pukansky invariants, rendering impossible the distinction between MASAs of invariant $\{2\}$ and those of invariant $\{1,2\}$.

The Pukansky invariant satisfies that if $\vn{A}$ and $\vn{B}$ are two unitarily equivalent MASAs in a factor $\vn{M}$ of type $\text{II}_{1}$, then $\puk{\vn{A}}=\puk{\vn{B}}$. However, the reciprocal statement is not true. One can even find four MASAs $\vn{A,B,C,D}$ in the type $\text{II}_{1}$ hyperfinite factor with equal invariants (all equal to $\{1\}$) where $\vn{A}$ is regular, $\vn{B}$ is semi-regular, $\vn{C}$ is singular, and $\vn{D}$ lies outside of Dixmier's classification. The Pukansky invariant is nonetheless very useful and some results about it will be used in this paper.

The four following theorems can be found in the book by Sinclair and Smith \cite{FiniteVNAandMasas}.

\begin{proposition}\label{pukreg}
Let $\vn{N}$ be a type $\text{II}_{1}$ factor and $\vn{A}$ be a MASA in $\vn{N}$. If  $\vn{A}$ is regular, then $\puk{\vn{A}}=\{1\}$.
\end{proposition}

\begin{proposition}\label{puk1}
Let $\vn{N}$ be a type $\text{II}_{1}$ factor and $\vn{A}$ be a MASA in $\vn{N}$. The following statements are equivalent:
\begin{itemize}
\item $\vn{A}$ is a MASA in $\B{L^{2}(\vn{N})}$;
\item $\puk{\vn{A}}=\{1\}$.
\end{itemize}
\end{proposition}

\begin{proposition}\label{puksemireg}
Let $\vn{N}$ be a type $\text{II}_{1}$ factor and $\vn{A}$ be a MASA in $\vn{N}$.
\begin{itemize}
\item If $\puk{\vn{A}}\subset\{2,3,4,\dots,\infty\}$, then $\vn{A}$ is singular.
\item If $\noralg{\vn{A}}\neq\vn{A}$, then $1\in\puk{\vn{A}}$.
\end{itemize}
\end{proposition}

\begin{proposition}\label{tenseurpuk}
Let $\vn{A}$ (resp. $\vn{B}$) be a MASA in a factor $\vn{M}$ (resp. $\vn{N}$) of type $\text{II}_{1}$. Then:
\begin{equation*}
\puk{\vn{A\otimes B}}=\puk{\vn{A}}\cup\puk{\vn{B}}\cup\puk{\vn{A}}\puk{\vn{B}}
\end{equation*}
where $\puk{\vn{A}}\puk{\vn{B}}=\{a\times b~|~a\in\puk{\vn{A}},b\in\puk{\vn{B}}\}$.
\end{proposition}

We have stated above that one can find four MASAs $\vn{A,B,C,D}$ of the hyperfinite factor $\finhyp$ of type $\text{II}_{1}$ that all have the same Pukansky invariant and such that $\vn{A}$ is regular, $\vn{B}$ is semi-regular, $\vn{C}$ is singular, and $\vn{D}$ lies outside of Dixmier classification (we will say that $\vn{D}$ is \emph{non-Dixmier-classifiable}). The regular MASAs are necessarily of Pukansky invariant $\{1\}$, and one can show that such MASAs exist in the type $\text{II}_{1}$ hyperfinite factor $\finhyp$ (for instance by considering a construction of the hyperfinite factor as a crossed product (\autoref{defcrossedprod}), as explained in Sinclair and Smith book \cite{FiniteVNAandMasas}). This gives the existence of a regular MASA $\vn{A}$ in $\finhyp$ of Pukansky invariant $\{1\}$. \label{refmasaspuk1-1}

On the other hand, Stuart White \cite{tauermasasorig} showed that the so-called \emph{Tauer MASAs} all have as Pukansky invariant the singleton $\{1\}$. And it is known that there exists singular Tauer MASAs \cite{tauermasas} and semi-regular Tauer MASAs \cite{tauermasasorig} in the hyperfinite factor $\finhyp$. This gives the existence of a semi-regular MASA $\vn{B}$  and a singular MASA $\vn{C}$ such that $\puk{\vn{B}}=\puk{\vn{C}}=\{1\}$.\label{refmasaspuk1-2}

Lastly, let us show that the existence of singular MASAs with Pukansky invariant equal to $\{1\}$ implies the existence of non-Dixmier-classifiable MASAs whose Pukansky invariant is equal to the singleton $\{1\}$. Indeed, if $\vn{A}$ is a MASA with $\puk{\vn{A}}=\{1\}$, we can consider $\vn{A}\otimes \vn{Q}$ where $\vn{Q}$ is a regular MASA (thus $\puk{\vn{Q}}=\{1\}$) of $\finhyp\otimes\finhyp$. We then have that $\puk{\vn{A}\otimes\vn{Q}}=\{1\}$ by \autoref{tenseurpuk}, and moreover, by \autoref{chifan}, we have:
$$\noralg[\finhyp\otimes\finhyp]{\vn{A}\otimes\vn{Q}}=\noralg[\finhyp]{\vn{A}}\otimes\noralg[\finhyp]{\vn{Q}}=\vn{A}\otimes\finhyp$$
But the center of $\vn{A}\otimes\finhyp$ is equal to $\vn{A}\otimes\complexN$ since $\vn{A}$ is commutative and the commutant of a tensor product is equal to the tensor product of the commutants (a result due to Tomita \cite{tomita}). Thus $\vn{A}\otimes\finhyp$ is not a factor, which implies that $\vn{A}\otimes\vn{Q}$ is neither regular nor semi-regular. Since $\vn{A}\otimes\vn{Q}$ is obviously not equal to $\vn{A}\otimes\finhyp$, we know that $\vn{A}\otimes\vn{Q}$ is not singular: it is therefore non-Dixmier-classifiable.
Eventually, as $\finhyp\otimes\finhyp$ is isomorphic to  $\finhyp$, it is enough to choose such an isomorphism $\phi$ to define $\vn{D}=\phi(\vn{A\otimes Q})$ a MASA in $\finhyp$  which is non-Dixmier-classifiable and such that $\puk{\vn{D}}=\{1\}$.\label{refmasaspuk1-3}

\section{Geometry of Interaction}

In this section, we review Girard's GoI models. This section has two distinct aims. The first is to offer a presentation of those constructions where the role of MASAs is shown explicitly. Indeed, MASAs played a role in all of Girard's GoI models, even though they were implicitly used though the choice of a specific basis of the Hilbert space in consideration. The second is to review Girard's GoI model in the hyperfinite factor \cite{goi5} since it is the starting point of our study.

\subsection{First Constructions: Nilpotency}

\subsubsection{Multiplicative Connectives and Exponentials.} The ancestor of GoI models \cite{multiplicatives} was an interpretation of multiplicative linear logic proofs as permutations. The first construction of a GoI model \cite{goi1} already used operator algebras as the notion of partial isometries provides a natural generalisation of permutations. Indeed, it is necessary to deal with infinite objects in order to represent exponential connectives, and finite permutations can naturally be replaced by permutations of a basis of an separable infinite-dimensional Hilbert space. This actually consists in working with partial isometries in the normalising groupoid of a fixed MASA, although this point of view was unknown to -- or at least never mentioned by -- Girard at the time. We will therefore present Girard's first GoI model under this novel perspective.

Let us start by choosing a separable infinite-dimensional Hilbert space $\hil{H}$, and a MASA $\vn{A}$ in $\B{\hil{H}}$. We will suppose that $\hil{H}=\ell^{2}(\naturalN)$ and that $\vn{A}$ is the MASA of diagonal operators in the basis $(\delta_{i,n})_{n\in\naturalN}$. We can then define operators\footnote{We recall that $\gn[\B{\hil{H}}]{\vn{A}}$ is the normalising groupoid of $\vn{A}$ (\autoref{groupoidenormalisant}).} $\ttl,\ttr \in\gn[\B{\hil{H}}]{\vn{A}}$ such that $\ttr\ttr^{\ast}+\ttl\ttl^{\ast}=1$ and $\ttr^{\ast}\ttr=\ttl^{\ast}\ttl=1$. We will chose here $\ttr((x_{n})_{n\in\naturalN})=(x_{2n})_{n\in\naturalN}$ and $\ttl((x_{n})_{n\in\naturalN})=(x_{2n+1})_{n\in\naturalN}$. If $\pi$ is a projection in $\vn{A}$ it is immediate that $\ttr\pi \ttr^{\ast}$ (respectively $\ttl\pi \ttl^{\ast}$) is a projection in $\vn{A}$, thus $\ttr\vn{A}\ttr^{\ast}\subset\vn{A}$ (respectively $\ttl\vn{A}\ttl^{\ast}\subset\vn{A}$) since $\vn{A}$ is generated by its projections. Moreover, $\ttr$ and $\ttl$ are partial isometries and the projections $\ttr\ttr^{\ast}$ and $\ttl\ttl^{\ast}$ are in $\vn{A}$. We have thus checked that $\ttr$ and $\ttl$ are indeed elements of $\gn[\B{\hil{H}}]{\vn{A}}$.

If $u\in\B{\ell^{2}(\naturalN)}$, we will write $\ttr(u)$ (resp. $\ttl(u)$) the operator $\ttr u\ttr^{\ast}$ (resp. $\ttl u\ttl^{\ast}$).

We will restrict in the following to elements in $\gn{\vn{A}}$. We now define a notion of orthogonality based on nilpotency.

\begin{definition}
Two operators $u,v$ in $\gn{\vn{A}}$ are orthogonal -- denoted by $u\perp v$ -- when $uv$ is nilpotent, i.e. when there exists an integer $n$ such that $(uv)^{n}=0$.
\end{definition}

This notion of orthogonality allows one to define types as bi-orthogonally closed sets.

\begin{definition}[Types]
A \emph{type} is a set of elements in $\gn{\vn{A}}$ which is bi-orthogonally closed, i.e. a set $T\subset \gn{\vn{A}}$ such that $T^{\bot\bot}=T$.
\end{definition}

The construction of the tensor product is performed using $\ttr$ and $\ttl$ which \emph{internalise} the direct sum of Hilbert spaces. Indeed, the Hilbert space $\hil{H}$ satisfies $\hil{H}\cong\hil{H}\oplus\hil{H}$.
\begin{definition}
If $A,B$ are types, the tensor product of $A$ and $B$ is defined as:
\begin{equation*}
A\otimes B=\{\ttr(u)+\ttl(v)~|~u\in A, v\in B\}^{\bot\bot}
\end{equation*}
We will write $u\odot v=\ttr(u)+\ttl(v)$. Using \autoref{lemmaproductpartialiso} and \autoref{lemmasumpartialiso} one can show that if $u,v$ are in $\gn[\B{\hil{H}}]{\vn{A}}$ then $u\odot v\in \gn[\B{\hil{H}}]{\vn{A}}$.
\end{definition}

We also define the linear implication. It is defined by the so-called \emph{execution formula}.
\begin{definition}
Let $u,v$ be operators in $\gn{\vn{A}}$ such that $u\perp \ttr(v)$. The execution of $u$ and $\ttr(v)$, denoted by $\Ex(u, \ttr(v))$, is defined as:
\begin{equation*}
\Ex(u,\ttr(v))=(1-\ttr\ttr^{\ast})(1-u\ttr(v))^{-1}(1-\ttr\ttr^{\ast})
\end{equation*}
Since $u\perp \ttr(v)$, the inverse of $1-u\ttr(v)$ always exists and can be computed as the series $\sum_{i=0}^{\infty} (u\ttr(v))^{i}$. One can show that\footnote{We will not present a proof here, but it would be similar to the proof of \autoref{composi50}.} if $u$ and $\ttr(v)$ are in $\gn[\B{\hil{H}}]{\vn{A}}$, then $\Ex(u,\ttr(v)) \in \gn[\B{\hil{H}}]{\vn{A}}$.
\end{definition}

One then shows that the following property, called the \emph{adjunction}, holds; it ensures that one can interpret soundly the connectives of linear logic\footnote{One can find detailed explanation in the author's previous work on interaction graphs \cite{seiller-goim,seiller-goia}.}.

\begin{proposition}[Adjunction]
If $u,v,w$ are elements of $\gn{\vn{A}}$, then:
\begin{equation*}
f\perp u\odot v\Leftrightarrow (f\perp \ttr(u))\wedge (\Ex(f,\ttr(u))\perp \ttl(v))
\end{equation*}
\end{proposition}

\begin{theorem}
If $A,B$ are types, we define the set $A\multimap B$ as
\begin{equation*}
A\multimap B=\{f\in\gn{\vn{A}}~|~\forall u\in A, \exists v\in B, f\bot \ttr(u)~\wedge~\Ex(f,\ttr(u))=\ttl(v)\}
\end{equation*}
This set is a type and satisfies the following:
\begin{equation*}
A\multimap B=(A\otimes B^{\bot})^{\bot}
\end{equation*}
\end{theorem}

\begin{proof}
Let $f\in A\multimap B$. Then for all $u\in A$ and $v'\in B^{\bot}$, $f\bot \ttr(u)$ and $\Ex(f,\ttr(u))=\ttl(v)$ where $v$ is an element of $B$. Therefore $\Ex(f,\ttr(u))\perp \ttl(v')$. Since $f\bot \ttr(u)$ and $\Ex(f,\ttr(u))\perp \ttl(v')$, we deduce from the adjunction that $f\perp u\odot v'$. As a consequence, we have that $f$ is an element of $(A\otimes B^{\bot})^{\bot}$.

Conversely, if $f\in (A\otimes B^{\bot})^{\bot}$, then $f\perp u\odot v'$ for all $u\in A, v'\in B^{\bot}$. From the adjunction, we get that $f\bot \ttr(u)$ and $\Ex(f,\ttr(u))$ is defined and orthogonal to $\ttl(v')$. Thus $\Ex(f,\ttr(u))$ is equal to $\ttl(v)$ for $v$ an element of $B$. This shows that $f\in A\multimap B$.

We conclude that $A\multimap B=(A\otimes B^{\bot})^{\bot}$,  which implies that $A\multimap B$ is a type.
\end{proof}

In order to define exponential connectives, one uses an internalisation of the tensor product of Hilbert spaces. Indeed, since $\hil{H}$ is separable and infinite-dimensional, it satisfies $\hil{H}\cong\hil{H}\otimes\hil{H}$. We thus choose such an isomorphism $\beta$ and an internalisation of the associativity: an operator called $t$. For instance, one can use the bijection $\beta:\naturalN\times\naturalN\rightarrow \naturalN$ defined by $(n,m)\mapsto 2^{n}(2m+1)-1$. This bijection $\beta$ induces a unitary $u_{\beta}: \hil{H}\otimes\hil{H}\rightarrow \hil{H}$ by defining $(\delta_{i,n},\delta_{i,m})_{n,m\in\naturalN}\mapsto(\delta_{i,\beta(m,n)})$ on basis elements. We can then define an internalisation of the tensor product: if $u,v$ are elements of $\gn{\vn{A}}$, we define $u\mathbin{\bar{\otimes}}v=u_{\beta}(u\otimes v)u_{\beta}^{\ast}$. Naturally, $(u\mathbin{\bar{\otimes}}v)\mathbin{\bar{\otimes}}w=u_{\beta}((u\mathbin{\bar{\otimes}}v)\otimes w)u_{\beta}^{\ast}=u_{\beta}((u_{\beta}(u\otimes v)u_{\beta}^{\ast})\otimes w)u_{\beta}^{\ast}$ is not equal to $u\mathbin{\bar{\otimes}}(v\mathbin{\bar{\otimes}}w)=u_{\beta}(u\otimes (u_{\beta}(v\otimes w)u_{\beta}^{\ast}))u_{\beta}^{\ast}$. There exists however a unitary $t$  which internalises the associativity, namely the operator induced by the map:
\begin{equation*}
\gamma:\begin{array}{rcl}\naturalN&\rightarrow&\naturalN\\
 		\beta(\beta(p,q),r)&\mapsto& \beta(p,\beta(q,r))
		\end{array}
\end{equation*}

\begin{definition}
Let $u\in\gn{A}$. We define $\oc u$ as the internalisation of $1\otimes u$, i.e. as $1\mathbin{\bar{\otimes}} u$.
\end{definition}

The definition of the exponential by $1\mathbin{\bar{\otimes}} u$ boils down to replacing $u$ by an infinite (countable) number of copies of itself. Indeed, $1\otimes u\in\B{\hil{H}\oplus\hil{H}}$ is equal to $\sum_{i\in\naturalN} e_{i}\otimes u$, where $(e_{i})$ is a basis of $\hil{H}$.

\subsubsection{Interpretation of Proofs.}\label{gdi1interprete} In his paper \cite{goi1}, Girard defined the interpretation of proofs as matrices in $\mathcal{M}_{n}(\B{\ell^{2}(\naturalN)})$, where $n$ is the number of formulas in the conclusion (taking into account the cut formulas that appear in the conclusion in the particular sequent calculus he considers). We will here present only the interpretation of the multiplicative fragment (MLL). \autoref{calculsequentgdinilp} shows the derivation rules of the system considered. Formulas are those of multiplicative linear logic, and sequents are of the form $\vdash [\Delta],\Gamma$ where $\Delta=A_{1},A_{1}^{\bot},A_{2},A_{2}^{\bot},\dots,A_{k},A_{k}^{\bot}$ is the multiset of cut formulas. Girard then defines the interpretation of a proof as a pair $(\pi^{\bullet},\sigma_{\pi})$, where $\pi^{\bullet}$ is a partial isometry in $\mathcal{M}_{n}(\B{\ell^{2}(\naturalN)})$ (more precisely in $\mathcal{M}_{n}(\gn[\B{\ell^{2}(\naturalN)}]{\vn{A}})$) which represents the proof $\pi$, and $\sigma_{\pi}$ is a partial symmetry $\mathcal{M}_{n}(\B{\ell^{2}(\naturalN)})$ (more precisely in $\mathcal{M}_{n}(\gn[\B{\ell^{2}(\naturalN)}]{\vn{A}})$) which represents the set of cut rules in $\pi$.

\begin{center}
\begin{figure}
\centering
\begin{tabular}{cc}
\begin{minipage}{4cm}
\begin{prooftree}
\AxiomC{}
\RightLabel{\scriptsize{Ax}}
\UnaryInfC{$\vdash [], A,A^{\bot}$}
\end{prooftree}
\end{minipage}
&
\begin{minipage}{6cm}
\begin{prooftree}
\AxiomC{$\vdash [\Theta_{1}],\Delta, A$}
\AxiomC{$\vdash [\Theta_{2}],\Gamma, A^{\bot}$}
\RightLabel{\scriptsize{Cut}}
\BinaryInfC{$\vdash [\Theta_{1},\Theta_{2},A,A^{\bot}],\Gamma,\Delta$}
\end{prooftree}
\end{minipage}
\\~\\
\begin{minipage}{4cm}
\begin{prooftree}
\AxiomC{$\vdash [\Theta],\Gamma, A, B$}
\RightLabel{\scriptsize{$\parr$}}
\UnaryInfC{$\vdash [\Theta],\Gamma, A\parr B$}
\end{prooftree}
\end{minipage}
&
\begin{minipage}{6cm}
\begin{prooftree}
\AxiomC{$\vdash [\Theta_{1}],\Gamma, A$}
\AxiomC{$\vdash [\Theta_{2}],\Delta, B$}
\RightLabel{\scriptsize{$\otimes$}}
\BinaryInfC{$\vdash [\Theta_{1},\Theta_{2}],\Gamma, \Delta, A\otimes B$}
\end{prooftree}
\end{minipage}
\end{tabular}
\caption{Sequent Calculus with Explicit Cuts}\label{calculsequentgdinilp}
\end{figure}
\end{center}

As opposed to Girard, we will define directly the interpretation of proofs as elements of $\B{\ell^{2}(\naturalN)}$ by internalising the algebra of matrices, i.e. by working modulo the isomorphism between $\vn{M}_{n}(\B{\ell^{2}(\naturalN)})$ and $\B{\ell^{2}(\naturalN)}$. We will represent a sequent by the $\parr$ of the formulas it is composed of. The two projections $\ttr\ttr^{\ast}$ and $\ttl\ttl^{\ast}$ are equivalent in the sense of Murray and von Neumann: the partial isometry $\tta=\ttl\ttr^{\ast}$ satisfies $\tta\tta^{\ast}=(\ttl\ttr^{\ast})(\ttl\ttr^{\ast})^{\ast}=\ttl\ttr^{\ast}\ttr\ttl^{\ast}=\ttl\ttl^{\ast}$ and $\tta^{\ast}\tta=\ttr\ttr^{\ast}$. It will be used to represent axioms.

Let $\vdash [\Delta],\Gamma$ be a sequent. Each formula $A$ in $\Delta\cup\Gamma$ can be assigned an \emph{address}, i.e. a sequence of $\ttr$ and $\ttl$ describing the projection onto the subspace corresponding to $A$. If $A$ and $A^{\bot}$ are two formulas with the addresses $p_{1},p_{2}$ respectively, we can define a partial isometry $p_{2}p_{1}^{\ast}$ between those (constructed from the partial isometries $\ttr$ and $\ttl$) which we will denote by $\tau(p_{1},p_{2})$. Notice that $\tau(\ttr,\ttl)=\tta$.

\begin{definition}[Representation of Proofs]
We define the representation $(\pi^{\bullet},\sigma_{\pi})$ of a proof $\pi$ inductively:
\begin{itemize}
\item if $\pi$ is an axiom rule, we define $\pi^{\bullet}=\tta+\tta^{\ast}$ and $\sigma_{\pi}=0$;
\item if $\pi$ is obtained from $\pi_{1}$ and $\pi_{2}$  by applying a $\otimes$ rule, we define $\pi^{\bullet}=\pi_{1}^{\bullet}\odot\pi_{2}^{\bullet}$ and $\sigma_{\pi}=\ttr(\sigma_{1})+\ttl(\sigma_{2})$;
\item if $\pi$ is obtained from $\pi_{1}$ by applying a $\parr$ rule, then $\pi^{\bullet}=\pi_{1}^{\bullet}$ and $\sigma_{\pi}=\sigma_{\pi_{1}}$;
\item if $\pi$ is obtained from $\pi_{1}$ and $\pi_{2}$ by applying a $\text{Cut}$ rule between formulas at the addresses $p_{1},p_{2}$, we define $\pi^{\bullet}=\pi_{1}^{\bullet}\odot\pi_{2}^{\bullet}$ and $\sigma_{\pi}=\ttr(\sigma_{1})+\ttl(\sigma_{2})+\tau(p_{1},p_{2})+\tau(p_{2},p_{1})$.
\end{itemize}
\end{definition}

One can then show that if $\pi$ is a proof with cuts and $\pi'$ is the cut-free proof obtained from applying the cut-elimination procedure on $\pi$, the operators $\pi^{\bullet}$ and $\sigma_{\pi}$ are orthogonal and the result of the execution formula $\Ex(\pi^{\bullet},\sigma_{\pi})=\sum_{i=0}^{\infty} (\pi^{\bullet}\sigma_{\pi})^{i}$ is equal to $(\pi')^{\bullet}$.

\subsubsection{Weak Nilpotency.} The GoI model we partially exposed allows one to interpret system $F$. In order to extend the model to the full pure lambda-calculus, Girard replaced the notion of nilpotency by a weaker notion, namely \emph{weak nilpotency}, i.e. point wise nilpotency.

\begin{definition}
An operator $u$ is \emph{weakly nilpotent} if $u^{n}$ weakly converges to $0$.
\end{definition}

The main difficulty in this work consists in showing that the execution formula $\Ex(u,\sigma)$ is still well-defined when $u\sigma$ is only \emph{weakly nilpotent} but not necessarily nilpotent. In the previous construction, the nilpotency of $u\sigma$ ensured the convergence of the series $\sum_{i=0}^{\infty} (u\sigma)^{i}$, and therefore the fact that the execution formula was well-defined. The case when $u\sigma$ is only weakly nilpotent is more delicate. Indeed, if $u\sigma$ is weakly nilpotent, the operator $1-u\sigma$ need not be invertible. Girard showed \cite{goi2} that it however admits an unbounded inverse $\rho$ defined on a dense subspace of $\hil{H}$. Moreover, since the operators considered are all partial isometries in the normalising groupoid of a given MASA $\vn{A}$, the operators $(u\sigma)^{k}$ are partial isometries of disjoint domains and codomains and their sum is again a partial isometry in the normalising groupoid\footnote{This is justified by \autoref{lemmasumpartialiso} and \autoref{lemmaproductpartialiso}.} of $\vn{A}$. From this, one can show that the restriction of $\rho$ to the subspace $(1-\sigma^{2})\hil{H}$ is a bounded operator, which yields the following proposition.

\begin{proposition}[Girard \cite{goi2}]
If $u\sigma$ is weakly nilpotent, the execution $\Ex(u,\sigma)$ is well-defined.
\end{proposition}

\subsubsection{Additive Connectives.} The definition of additive connectives appeared in the next GoI model \cite{goi3}, using the notion of \emph{dialect}. This notion allows one to encode \emph{private information} in the operators, i.e. information which has no consequence on the interaction with other operators. To get a good intuition, one may compare this to the usual notion of \emph{(control) state} for abstract machines: while the current state of a machine $M$ has an impact on how this machine will transition at the next step, it will not modify the behaviour of another machine that may compute the input or compute on the output of $M$. This is translated when composing machines by taking a product of the sets of states; similarly here, composition will be dealt with by taking a tensor product -- an operation which amounts to a product of the bases.

We therefore replace operators acting on a Hilbert space $\hil{H}$ by operators acting on the Hilbert space $\hil{H}\otimes\hil{H}$, where the first copy of $\hil{H}$ is \emph{public}, while the second (the dialect) is \emph{private}. This \enquote{privacy} will mainly show when considering the composition of two such operators, when we will consider their dialects to be disjoint: an operator will necessarily act as the identity on the dialect of the other. To enforce operators to act as identities on the dialect of another, one uses the tensor product $\hil{H\otimes H}$: if $u,v$ are operators on $\hil{H\otimes H}$, we can extend them as operators $u^{\dagger},v^{\ddagger}$ acting on $\hil{H}\otimes(\hil{H\otimes H})$. They are then considered as operators with two dialects (the dialect of $u$ and the dialect of $v$) but which act non-trivially only on one of them (they are extended by the identity on the second dialect, defining for instance $u^{\dagger}=u\otimes 1$). Through an internalisation of the tensor product, one can then consider the pair of two dialects as a single dialect: the operator $u^{\dagger}v^{\ddagger}$ can be seen as acting on $\hil{H\otimes H}$.

Girard's article \emph{Geometry of Interaction $\text{III}$: Accommodating the Additives} \cite{goi3} interprets proofs as operators in a C$^{\ast}$-algebra, as was the case in the previous models described above. This algebra is however described by Girard as an algebra of clauses. For homogeneity reasons, and because the presentation as an algebra of clauses hides once more the dependency of the construction on the choice of a MASA, we will present it here by using operators. This presentation is a small variation on the presentation one can find in Duchesne's PhD thesis \cite{theseetienne}, although the latter once again hides the role of MASAs behind a fixed choice of basis.

Let us choose once again the Hilbert space $\hil{H}=\ell^{2}(\naturalN)$, and the MASA $\vn{A}$ of $\B{\ell^{2}(\naturalN)}$ defined as the algebra of diagonal operators in the basis $(\delta_{i,n})_{n\in\naturalN}$. We will consider (disjoint) sums of operators of the form $u\otimes p$ -- where $p$ is a projection and $u$ a partial isometry -- in the normalising groupoid of $\vn{A}\otimes\vn{A}$, a MASA in $\B{\hil{H}\otimes\hil{H}}$.

\begin{definition}
Let $u$ be an operator. We say that $u$ is a \emph{GoI operator} when $u=\sum_{i\in I} u_{i}\otimes p_{i}$, where for all $i\in I$, $u_{i}$ is a partial isometry in $\gn{\vn{A}}$ and $p_{i}$ is a projection in $\vn{A}$. We will moreover impose that $\sum_{i\in I}p_{i}\sim 1$ and $p_{i}\sim p_{j}$ (for all $i,j$) where $\sim$ represents the Murray and von Neumann equivalence.
\end{definition}

\begin{remark}
If $u$ is a GoI operator, then $u\in\gn[\B{\hil{H\otimes H}}]{\vn{A}\otimes\vn{A}}$.
\end{remark}

\begin{definition}
Let $u=\sum_{i\in I} u_{i}\otimes p_{i}$ and $v=\sum_{j\in J} v_{j}\otimes q_{j}$ be two GoI operators. We define $u^{\dagger}$ as the operator $\sum_{i\in I} u_{i}\otimes (p_{i}\mathbin{\bar{\otimes}}1)$. We define similarly $v^{\ddagger}$ as the operator $\sum_{j\in J} v_{j}\otimes (1\mathbin{\bar{\otimes}} q_{j})$. 
\end{definition}

\begin{definition}
If $u$ and $v$ are two GoI operators, we will say they are orthogonal when $u^{\dagger}v^{\ddagger}$ is nilpotent. We will say that $u$ and $v$ are weakly orthogonal when the product $u^{\dagger}v^{\ddagger}$ is weakly nilpotent.
\end{definition}

We will now use once again the partial isometries $\ttr,\ttl$ introduced earlier in this section. However, $\ttr$ and $\ttl$ are operators in $\B{\hil{H}}$ while GoI operators are elements in $\B{\hil{H\otimes H}}$. We thus extend these operators to operators in $\B{\hil{H\otimes H}}$ in order to take dialects into account. We will write $\httr$ (respectively $\httl$) the operator $\ttr\otimes 1$ (respectively $\ttl\otimes 1$) and we will write $\bttr(u)$ (respectively $\bttl(u)$) the operator $\httr u\httr^{\ast}$ (respectively $\httl u\httl^{\ast}$).

\begin{proposition}
Let $u,v$ be GoI operators. Then:
\begin{equation*}
\begin{array}{rcccl}
\bttr(u^{\dagger})&=&\sum_{i\in I} \ttr(u_{i})\otimes (p_{i}\mathbin{\bar{\otimes}} 1)&=&\sum_{i\in I}\sum_{j\in J} \ttr(u_{i})\otimes (p_{i}\mathbin{\bar{\otimes}}q_{j})\\
\bttl(v^{\ddagger})&=&\sum_{j\in J} \ttl(v_{j})\otimes (1\mathbin{\bar{\otimes}} q_{j})&=&\sum_{j\in J}\sum_{i\in I} \ttl(v_{j})\otimes (p_{i}\mathbin{\bar{\otimes}}q_{j})
\end{array}
\end{equation*}
The operator $u\odot v=\bttr(u^{\dagger})+\bttl(v^{\ddagger})$ is therefore a GoI operator equal to:
 $$\sum_{i\in I}\sum_{j\in J} (\ttr(u_{i})+ \ttl(v_{j}))\otimes (p_{i}\mathbin{\bar{\otimes}} q_{j})$$
\end{proposition}

Once again, one can define a (weak) type as a (weakly) biorthogonally closed set of GoI operators.

\begin{definition}
Let $A,B$ be two (weak) types, we define their tensor product as:
\begin{equation*}
A\otimes B=\{u\odot v~|~ u\in A, v\in B\}^{\bot\bot}
\end{equation*}
\end{definition}

\begin{definition}
Let $u,v$ be orthogonal GoI operators. We define the execution of $u$ and $\bttr(v)$ as the GoI operator $\Ex(u, \bttr(v))=(1-\httr\httr^{\ast})\sum_{i\geqslant 1} (u^{\dagger}\bttr(v^{\ddagger}))^{i}u(1-\httr\httr^{\ast})$.
\end{definition}

The fact that this is a GoI operator comes from the following computation:
\begin{eqnarray*}
\lefteqn{\Ex(u,\bttr(v))}\\
&=&(1-\httr\httr^{\ast})\sum_{i\geqslant 1} (u^{\dagger}\bttr(v^{\ddagger}))^{i}u^{\dagger}(1-\httr\httr^{\ast})\\
&=&(1-\httr\httr^{\ast})\sum_{i\geqslant 1}\left(\left(\sum_{k=0}^{n}\sum_{l=0}^{m} u_{k}\otimes (p_{k}\mathbin{\bar{\otimes}}q_{l})\right)\left(\sum_{k'=0}^{n}\sum_{l'=0}^{m} \ttr(v_{l'})\otimes (p_{k'}\mathbin{\bar{\otimes}}q_{l'})\right)\right)^{i}u^{\dagger}(1-\httr\httr^{\ast})\\
&=&(1-\httr\httr^{\ast})\sum_{i\geqslant 1}\left(\sum_{k=0}^{n}\sum_{l=0}^{m} (u_{k}v_{l})\otimes (p_{k}\mathbin{\bar{\otimes}}q_{l})\right)^{i}\left(\sum_{k=0}^{n}\sum_{l=0}^{m} u_{k}\otimes (p_{k}\mathbin{\bar{\otimes}}q_{l})\right)(1-\httr\httr^{\ast})\\
&=&\sum_{k=0}^{n}\sum_{l=0}^{m} \left((1-\httr\httr^{\ast})\sum_{i\geqslant 1}(u_{k}\ttr(v_{l}))^{i}u_{k}(1-pp^{\ast})\right)\otimes (p_{k}\mathbin{\bar{\otimes}}q_{l})
\end{eqnarray*}

\begin{proposition}
Let $u,v$ be orthogonal GoI operators. Then:
\begin{equation*}
u\perp (v\odot w)\Leftrightarrow (u\perp \bttr(v^{\dagger}))\wedge(\Ex(u,\bttr(v^{\dagger})))\perp \bttl(w))
\end{equation*}
\end{proposition}

\begin{theorem}
Let $A,B$ be (weak) types. Then $$A\multimap B=\{f~|~\forall u\in A, \exists v\in B, \Ex(f,\bttr(u))=\bttl(v)\}$$ is a (weak) type and $A\multimap B=(A\otimes B^{\bot})^{\bot}$.
\end{theorem}

\begin{definition}
Let $u,v$ be GoI operators. We define $$u\with v=(1\otimes p)u(1\otimes p)^{\ast}+(1\otimes q)v(1\otimes q)^{\ast}$$ If $A,B$ are (weak) types, we define $$A\with B=\{u\with v~|~u\in A, v\in B\}^{\bot\bot}$$
\end{definition}

\begin{definition}
Let $u,v$ be GoI operators. We say that $u$ is a variant of $v$ -- written $u\sim v$ -- when there exists a partial isometry $w\in\gn{\vn{A}}$ such that $u=(1\otimes w)v(1\otimes w)^{\ast}$.
\end{definition}

\begin{proposition}
Let $u,v,w$ be GoI operators such that $u\sim v$. Then $u\perp w$ if and only if $v\perp w$. Moreover, $\Ex(u,w)\sim \Ex(v,w)$.
\end{proposition}

\subsection{Hyperfinite GoI}\label{goi5}


\subsubsection{Locativity.} Between the GoI models explained above and the hyperfinite GoI model \cite{goi5}, Girard introduced \emph{ludics} \cite{locussolum}. If the constructions of ludics may appear at first sight quite different from the constructions of GoI models, both constructions are, in a sense, exactly the same\footnote{In particular, we have shown that our construction of Interaction Graphs based on graphings unifies Girard's GoI models \cite{seiller-goig}. We believe that the general framework of graphings will yield a a special case Girard's ludics.}. Indeed, the constructions differ only from their starting point: when GoI models are built upon an abstraction of proof nets (or rather proof structures) \cite{seiller-axioms}, ludics is built upon an abstraction of (focalised) MALL sequent calculus derivations (with a modified axiom rule \cite{seiller-axioms,llludintro2}). 

One can show that a formula $A$ is provable in MALL if and only if a specific formula $A^{\sharp}$ is provable in a system MALL$_{\text{foc}}$. This formula $A^{\sharp}$ is a normal form of $A$ obtained by using distributivity isomorphisms. The system MALL$_{\text{foc}}$ uses the fact that all provable sequent has a \emph{focalised proof}, i.e. a proof alternating between a sequence of reversible (negative) rules of maximal length, and a sequence of non reversible (positive) rules introducing the positive connectives of a single formula (thus the choice of terminology) of maximal length. This sequent calculus possesses an axiom rule and two schemes of rules: a negative scheme and a positive scheme -- each representing the possible sequences of reversible or non-reversible rules.

Ludics is then an abstraction of this sequent calculus: we first replace the axiom rule by a rule $\Dai$. This rule $\Dai$ (called \emph{daimon}) in ludics introduces only positive sequents and therefore can never introduce a sequent of the form $\vdash A,A^{\pol}$; it consequently never corresponds to the application of an axiom rule. This is counterbalanced by the consideration of infinite derivation trees: a correct sequent -- such as $A\vdash A$ -- will then be introduced by a sort of infinite $\eta$-expansion named the \emph{fax}. The second abstraction consists in replacing formulas by addresses -- finite sequences of integers. We already considered a notion of address in our presentation of  the interpretation of proofs in Girard's first GoI model.

We will not detail the constructions of ludics here, but we will stress this \emph{locative} aspect. If Girard's first GoI models were already locative -- an address was then a sequence of symbols $\ttr$ and $\ttl$ -- it only became explicit after the introduction of ludics. In particular, in the first GoI models, the tensor product was always defined because its was defined through adequate \emph{delocations}: conjugating the left-hand element by $\ttl$ and the right-hand by $\ttl$. In particular, this allows one to consider operators that always act on the same space: the Hilbert space $\ell^{2}(\naturalN)$. In Girard's hyperfinite GoI model \cite{goi5}, the operators considered are elements of an algebra $\infhyp$ of type $\text{II}_{\infty}$, but act only on a \emph{finite} subspace (finite from the point of view of the algebra, i.e. the projection onto the subspace is finite in the algebra). Then, the objects under study are given together with a projection $p\in\infhyp$ -- the \emph{location} -- and an operator $u$ such that $pup=u$. A consequence of the locative approach is that some operations are only partially defined -- as the tensor product. It is however possible to retrieve total constructions by working \emph{modulo delocations}; for instance, this is how one builds categories from locative GoI models\footnote{Let us stress that the cited works are not the first describing how denotational semantics arise from GoI models, although the second is the first dealing with additives. First, the main theorem of Girard's first GoI construction \cite[Theorem 1, (ii)]{goi1} shows that one can (partially) define a denotational semantics from the GoI interpretation of proofs. Second, Haghverdi \cite{haghverdi} showed that GoI situations \cite{goisituation1,goisituation2} -- a categorical axiomatisation of earlier GoI models -- can be used to define denotational semantics.} \cite{seiller-goim,seiller-goia}.

\subsubsection{The Feedback Equation.} The feedback equation is the operator-theoretic counterpart of the cut-elimination procedure. Hence a solution to the feedback equation is the equivalent of the normal form a proof net that may contain cuts. This equation is stated as follows: if $u,v$ are operators acting on the Hilbert spaces $\hil{H}\oplus\hil{H}'$ and $\hil{H}'\oplus\hil{H''}$ respectively, a \emph{solution to the feedback equation} is an operator $w$ acting on $\hil{H}\oplus\hil{H}''$ and such that $w(x\oplus z)=x'\oplus z'$ as long as there exist $y,y'\in\hil{H}'$ satisfying:
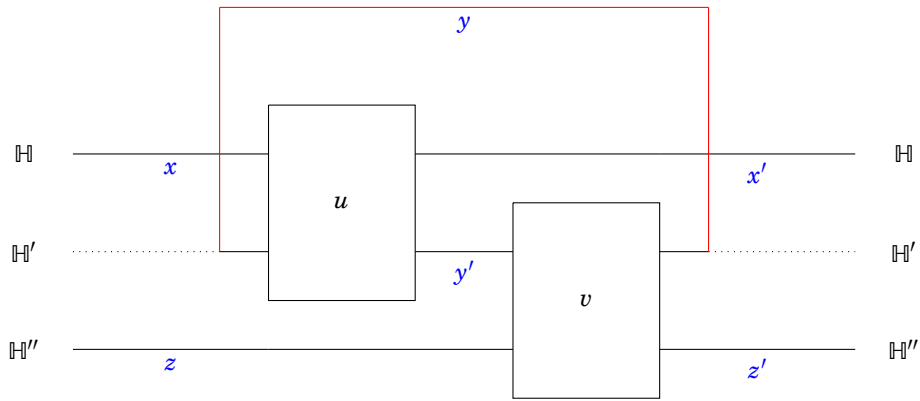
\begin{eqnarray*}
u(x\oplus y)&=&x'\oplus y'\\
v(y'\oplus z)&=&y\oplus z'
\end{eqnarray*}

\begin{center}
\begin{tikzpicture}[x=0.65cm,y=0.65cm]
	\draw (-2,0) -- (2,0) node [midway,below,blue] {$x$};
		\node (Hi) at (-3,0) {$\hil{H}$};
	\draw (1,-2) -- (2,-2) {};
	\draw[dotted] (-2,-2) -- (1,-2) {};
		\node (HHi) at (-3,-2) {$\hil{H'}$};
	\draw (-2,-4) -- (2,-4) node [midway,below,blue] {$z$};
	\draw (2,-4) -- (7,-4) {};
		\node (HHHi) at (-3,-4) {$\hil{H''}$};
	\draw (2,1) -- (2,-3) -- (5,-3) -- (5,1) -- (2,1);
		\node (U) at (3.5,-1) {$u$};
	\draw (7,-1) -- (7,-5) -- (10,-5) -- (10,-1) -- (7,-1);
		\node (V) at (8.5,-3) {$v$};
	\draw (5,0) -- (10,0) {};
	\draw (10,0) -- (14,0) node [midway,below,blue] {$x'$};
		\node (Ho) at (15,0) {$\hil{H}$};
	\draw (5,-2) -- (7,-2) node [midway,below,blue] {$y'$};
	\draw (10,-2) -- (11,-2) {};
	\draw[dotted] (11,-2) -- (14,-2) {};
		\node (HHo) at (15,-2) {$\hil{H'}$};
	\draw (10,-4) -- (14,-4) node [midway,below,blue] {$z'$};
		\node (HHHo) at (15,-4) {$\hil{H''}$};
	\draw[red] (11,-2) -- (11,3) {};
	\draw[red] (11,3) -- (1,3) node [midway,below,blue] {$y$};	
	\draw[red] (1,3) -- (1,-2) {};
\end{tikzpicture}
\end{center}

Let us write $p,p',p''$ the projections onto the subspaces $\hil{H,H',H''}$ respectively. The execution formula $\Ex(u, v)=(p+p''v)\sum_{i\geqslant 0} (uv)^{i}(up+p'')$, when it is defined, yields a solution to the feedback equation involving $u$ and $v$. More generally, the formula $(p+p''v)(1-uv)^{-1}(up+p'')$, when defined, describes a solution to the feedback equation.

Girard studied, in the paper \emph{Geometry of Interaction $\text{IV}$: the Feedback Equation} \cite{feedback}, an extension of this solution. Indeed, he showed that as long as $u,v$ are hermitians of norm at most $1$, the solution $(p+p''v)(1-uv)^{-1}(up+p'')$ defines a partial functional application which can be extended to be defined for all pairs of hermitian operators in the unit ball $1$. Moreover, this extension is the unique such extension that preserves some properties. This unique extension will be denoted by $\plug$ in this section.

\subsubsection{The Determinant.} The hyperfinite GoI model no longer uses the orthogonality defined by nilpotency but considers a more involved notion defined through the determinant of operators. In order to motivate this change, we consider $G,H,F$ three square matrices of respective dimensions $n\times n$, $m\times m$ and $(n+m)\times(n+m)$. We can write $F$ as a block matrix as follows:
$$F=\left(\begin{array}{cc} F_{11} & F_{12}\\ F_{21} & F_{22}\end{array}\right)$$
where $F_{11}$ (respectively $F_{22}$) is a square matrix of dimension $n\times n$ (respectively $m\times m$). We will write $G\oplus H$ the square matrix of dimension $(n+m)\times(n+m)$ defined as:
$$G\oplus H=\left(\begin{array}{cc} G & 0\\ 0 & H\end{array}\right)$$
One can then notice that when $1-F_{11}G$ is invertible ($1$ is here the identity matrix of dimension $n\times n$), the computation of the determinant of $1-F(G\oplus H)$ involves the execution formula $\Ex(F,G)$:
\begin{eqnarray*}
\lefteqn{\det(1-F(G\oplus H))}\\
&=&\left|\begin{array}{cc}1-F_{11}G & -F_{12}H\\ -F_{21}G & 1-F_{22}H\end{array}\right|\\
&=&\left|\begin{array}{cc}1-F_{11}G & -F_{12}H+(1-F_{11}G)C\\ -F_{21}G & 1-F_{22}H-F_{21}GC\end{array}\right|~~~~~(C=(1-F_{11}G)^{-1}F_{12}H)\\
&=&\left|\begin{array}{cc}1-F_{11}G &0\\ -F_{21}G & 1-\Ex(F,G).H\end{array}\right|\\
&=&\det(1-F_{11}G)\det(1-\Ex(F,G).H)\nonumber
\end{eqnarray*}

Keeping in mind that $G\oplus H$ interprets linear logic's tensor product of $G$ and $H$, this equality is reminiscent of the adjunction upon which the interpretation of multiplicative connectives in earlier GoI was based. In fact, as shown by the author \cite{seiller-goia}, this equality and the adjunction are strongly related -- even more than that, they come from a unique identity between cycles.

However, in order to interpret exponential connectives one has to consider operators acting on infinite-dimensional spaces. This is why the hyperfinite GoI model takes place in a von Neumann algebra of type $\text{II}$. Indeed, the existence of a trace in factors of type $\text{II}_{1}$ allows one to define a generalisation of the determinant. This new GoI model thus considers operators in a particular algebra: the type $\text{II}_{\infty}$ hyperfinite factor. In fact, as already mentioned, the operators considered will belong to a sub-algebra $p\infhyp p$, where $p$ is a finite projection. This amounts to saying that we are working with operators in the type $\text{II}_{1}$ hyperfinite factor embedded in the type $\text{II}_{\infty}$ hyperfinite factor; the latter being used only to ensure that one do not run out of (disjoint) locations.

In a type $\text{II}_{1}$ factor, as explained in the previous section, there exists a trace. It is therefore possible to define a generalisation of the determinant of matrices by using the identity $\det(\exp(A))=\exp(\tr(A))$ which is, in finite dimensions, satisfied for all matrix $A$. Indeed, if $A$ is a matrix with complex coefficients,we can suppose it is in upper triangular form. The determinant of $\exp(A)$ is then the product of the exponentials of eigenvalues of $A$: $\det(\exp(A))=\prod_{i} \exp(\lambda_{i})$ and therefore $\det(\exp(A))=\exp(\sum_{i} \lambda_{i})$. This shows that $\det(\exp(A))=\exp(\tr(A))$.

In the case of a factor of type $\text{II}_{1}$, with its normalised trace $\tr$, we can define for all invertible operator $A$:
\begin{equation*}
\det(A)=e^{\tr(\log(\abs{A}))}
\end{equation*}
This generalisation of the determinant was introduced by Fuglede and Kadison \cite{FKdet} who showed that it can be extended to all operators, though not in a unique way. They also show a number of properties satisfied by the determinant, among which the following that will be useful in this paper:
\begin{itemize}
\item $\det$ is multiplicative: $\det(AB)=\det(A)\det(B)$;
\item for all $A$ $\det(A)<\specrad{A}$, where $\specrad{A}$ is the spectral radius of $A$.
\end{itemize}

In the construction of the GoI model, Girard actually uses a generalisation of the notion of trace: the so-called pseudo-traces. A \emph{pseudo-trace} is an hermitian ($\alpha(u)=\overline{\alpha(u^{\ast})}$), tracial ($\alpha(uv)=\alpha(vu)$), faithful and  normal ($\sigma$-weakly continuous) linear form.

\begin{definition}
If $\alpha$ is a pseudo-trace on $\vn{A}$, and $\tr$ a trace on $\finhyp$, we can define for all invertible $A\in\finhyp\otimes\vn{A}$:
\begin{equation*}
\det{}_{\tr\otimes\alpha}(A)=e^{\tr\otimes\alpha(\log(\abs{A}))}
\end{equation*}
\end{definition}

\begin{proposition}
If $\phi$  is a  $\ast$-isomorphism from $\vn{A}$ onto $\vn{B}$, then for all invertible $A\in\finhyp\otimes\vn{A}$:
\begin{equation*}
\det{}_{\tr\otimes(\alpha\circ\phi^{-1})}((\textnormal{Id}\otimes\phi)(A))=\det{}_{\tr\otimes\alpha}(A)
\end{equation*}
\end{proposition}

\begin{proof}
Let $A$ be  invertible in $\infhyp\otimes\vn{A}$. Then, using the definition of the determinant and the fact that $\text{Id}\otimes\phi$ commutes with functional calculus, we have:
\begin{eqnarray*}
\det{}_{\tr\otimes(\alpha\circ\phi^{-1})}((\text{Id}\otimes\phi)(A))&=&\text{exp}(\tr\otimes(\alpha\circ\phi^{-1})(\log(\abs{(\text{Id}\otimes\phi)(A)})))\\
&=&\text{exp}(\tr\otimes(\alpha\circ\phi^{-1})((\text{Id}\otimes\phi)(\log(\abs{A}))))\\
&=&\text{exp}(\tr\otimes\alpha(\log(\abs{A})))
\end{eqnarray*}
Thus $\det{}_{\tr\otimes(\alpha\circ\phi^{-1})}((\text{Id}\otimes\phi)(A))=\det{}_{\tr\otimes\alpha}(A)$.
\end{proof}

\begin{lemma}
If $A$ is nilpotent, then $\specrad{A}=0$. 
\end{lemma}

\begin{proof}
We have $\specrad{A}=\lim_{n\rightarrow \infty} \norm{A^{n}}^{\frac{1}{n}}$. Since $A$ is nilpotent, of degree $k$ for instance, we have that $\norm{A^{n}}=0$ for all $n\geqslant k$. Thus $\specrad{A}=0$.
\end{proof}

\begin{lemma}\label{detnil}
If $A$ is nilpotent, then $P(A)=\sum_{i=1}^{k} \alpha_{k}A^{k}$ is nilpotent.
\end{lemma}

\begin{proof}
The minimal degree of $A$ in $P(A)^{i}$ is equal to $i$. Thus $P(A)^{i}=0$ for $i\geqslant k$.
\end{proof}

\begin{proposition}
If $A$ is nilpotent, then $\det(1+A)=1$.
\end{proposition}

\begin{proof}
We will denote by $k$ the degree of nilpotency of $A$. Since $A$ is nilpotent, $\specrad{A}=0$. Pick $\lambda\in\spec{1+A}$. By definition, $\lambda.1-1-A$ is non-invertible, which means that $(\lambda-1).1-A$ is non-invertible, i.e. $\lambda-1\in\spec{A}$. This implies that $\lambda=1$ since the spectrum of $A$ is reduced to $\{0\}$. This implies that $\specrad{1+A}\leqslant 1$ and therefore that $\det(1+A)\leqslant 1$.

Moreover, $(1+A)^{-1}=\sum_{i=0}^{k-1} (-A)^{i}=1+\sum_{i=1}^{k-1}(-A)^{i}$. By the preceding lemma we know that $B=\sum_{i=1}^{k-1} (-A)^{i}$ is nilpotent, and therefore $\det(1+B)\leqslant 1$ using the same argument as before. Since $\det(1+B)=\det((1+A)^{-1})=(\det(1+A))^{-1}$, we conclude that $\det(1+A)=1$.
\end{proof}

\subsection{The Hyperfinite GoI}

We will here present a modified version of Girard's hyperfinite GoI model. The modification concerns mainly the additive connectives for which we will follow the constructions detailed in our work on interaction graphs \cite{seiller-goia}; as it was explained in the cited work, the construction of additives thus obtained is much more satisfactory than the construction detailed in Girard's paper \cite{goi5}. The results shown do not depend on this choice, but it will allow us to consider a sequent calculus for elementary linear logic we already studied \cite{seiller-goie}. We will call this construction \emph{Girard's hyperfinite GoI model}. We will also present a slight alteration of this construction (namely, with a different orthogonality) later on; this second construction will be used to obtain the main result of the paper, after we clarified its relation to Girard's hyperfinite GoI model.

The constructions are based on two essential properties, as explained in the author's previous work \cite{seiller-goia,seiller-goig}: 
\begin{itemize}
\item the associativity of the execution $\plug$ \cite{feedback}; 
\item the \enquote{adjunction} (the next theorem) which relates the execution and the \emph{measurement} between operators 
$$\meashyp{u,v}=-\log(\det(1-uv))$$
\end{itemize}

\begin{theorem}[Girard \cite{goi5}]\label{adjonctionFK}
Let $u,v,w$ be three  hermitian operators in the unit ball of the hyperfinite factor of type $\text{II}_{1}$, and $u\plug v$ the solution to the feedback equation involving $u$ and $v$. Then:
\begin{equation*}
\meashyp{u,v+w}=\meashyp{u,v}+\meashyp{u \plug v,w}
\end{equation*}
\end{theorem}

\subsubsection{Multiplicatives.}

\begin{definition}\label{defprojetgdi5}
A \emph{hyperfinite project} is a tuple $\de{a}=(p,a,\vn{A},\alpha,A)$, where:
\begin{itemize}
\item $p$ is a finite projection in $\infhyp$, the \emph{carrier} of $\de{a}$ ;
\item $a\in \realN\cup\{\infty\}$ is called the \emph{wager} of $\de{a}$;
\item $\vn{A}$ is a finite von Neumann algebra of type $\text{I}$, the \emph{dialect} of $\de{a}$;
\item $\alpha$ is a pseudo-trace on $\vn{A}$;
\item $A\in p\infhyp p\otimes \vn{A}$ is a  hermitian operator of norm at most $1$.
\end{itemize}
Using Girard's notation, we will write $\de{a}=a\cdot +\cdot \alpha +A$. When the dialect is equal to $\complexN$, we will denote by $1_{\complexN}$ the \enquote{trace} $x\mapsto x$.
\end{definition}

If $A\in \infhyp\otimes\vn{A}$ and $B\in\infhyp\otimes\vn{B}$, we will write $A^{\dagger_{\vn{B}}}$ and $B^{\ddagger_{\vn{A}}}$ (usually simplified as $A^{\dagger}$ and $B^{\ddagger}$) the operators in $\infhyp\otimes\vn{A}\otimes\vn{B}$ defined as:
\begin{eqnarray*}
A^{\dagger_{\vn{B}}}&=&A\otimes 1_{\vn{B}}\\
B^{\ddagger_{\vn{A}}}&=&(\text{Id}_{\infhyp}\otimes\tau)(B\otimes 1_{\vn{A}})
\end{eqnarray*}
where $\tau$  is the isomorphism $\vn{B}\otimes \vn{A}\rightarrow \vn{A}\otimes \vn{B}$.

\begin{definition}
Let $\de{a}=a\cdot +\cdot \alpha +A$ and $\de{b}=b\cdot+\cdot \beta+B$ be two hyperfinite projects. Then $\de{a}\simperp\de{b}$ when:
\begin{equation*}
\sca[\textnormal{hyp}]{a}{b}=a\beta(1_{\vn{B}})+\alpha(1_{\vn{A}})b-\log(\det{}_{\tr\otimes\alpha\otimes\beta}(1-A^{\dagger}B^{\ddagger}))\neq 0,\infty
\end{equation*}
If $A$ is a set of hyperfinite projects, we will write $A^{\pol}=\{\de{b}~|~\forall \de{a}\in A, \de{a}\poll\de{b}\}$ and $A^{\pol\pol}=(A^{\pol})^{\pol}$.
\end{definition}

\begin{definition}
Let $p$ be a finite projection in $\infhyp$. A \emph{conduct} of carrier $p$ is a set $\cond{A}$ of hyperfinite projects of carrier $p$ such that $\cond{A}=\cond{A}^{\pol\pol}$.
\end{definition}

\begin{definition}
If $\de{a},\de{b}$ are hyperfinite projects of disjoint carrier, the tensor product of $\de{a}$ and $\de{b}$ is defined as the hyperfinite project of carrier $p_{\de{a}}+p_{\de{b}}$ defined as 
$$\de{a\otimes b}=a\beta(1_{\vn{B}})+\alpha(1_{\vn{A}})b\cdot+\cdot \alpha\otimes\beta + A^{\dagger}+B^{\ddagger}$$
\end{definition}

\begin{definition}
If $\cond{A,B}$ are conducts of disjoint carrier, their tensor product is defined as the conduct:
\begin{equation*}
\cond{A\otimes B}=\{\de{a\otimes b}~|~\de{a}\in\cond{A}, \de{b}\in\cond{B}\}^{\pol\pol}
\end{equation*}
\end{definition}


\begin{definition}
If $\cond{A,B}$ are conducts of disjoint carrier, we define:
\begin{equation*}
\cond{A\multimap B}=\{\de{f}~|~\forall\de{a}\in\cond{A},\de{f\plug a}\in\cond{B}\}
\end{equation*}
\end{definition}

These definitions are coherent interpretations of multiplicative connectives of linear logic, as shown by the following property -- a direct consequence of the \emph{adjunction}.

\begin{theorem}\label{duality}
$$\cond{A\multimap B}=\cond{(A\otimes B^{\pol})^{\pol}}$$
\end{theorem}

\subsubsection{Additives.}

\begin{lemma}[Variants]
Let $\de{a}$ be a hyperfinite project in a conduct $\cond{A}$, and $\phi:\vn{A}\rightarrow \vn{B}$ a $\ast$-isomorphism. Then the hyperfinite project $\de{a}^{\phi}=a\cdot+\cdot\alpha\circ\phi^{-1}+\text{Id}\otimes\phi(A)$ is an element of $\cond{A}$. We will say that $\de{a^{\phi}}$ is a  \emph{variant} of $\de{a}$.
\end{lemma}

\begin{proof}
Let $\de{c}$ be a hyperfinite project whose carrier is equal to the carrier of $\de{a}$. Then:
\begin{eqnarray*}
\sca[\textnormal{hyp}]{a^{\phi}}{c}&=&a\gamma(1_{\vn{C}})+\alpha\circ\phi^{-1}(1_{\vn{B}})c-\log(\det(1-(\text{Id}\otimes\phi(A))^{\dagger_{\vn{C}}}C^{\ddagger_{\vn{B}}}))\\
&=&a\gamma(1_{\vn{C}})+\alpha(1_{\vn{A}})c-\log(\det(1-\text{Id}\otimes\phi\otimes\text{Id}(A^{\dagger_{\vn{C}}})C^{\ddagger_{\phi(\vn{A})}})\\
&=&a\gamma(1_{\vn{C}})+\alpha(1_{\vn{A}})c-\log(\det(1-\text{Id}\otimes\phi\otimes\text{Id}(A^{\dagger_{\vn{C}}}C^{\ddagger_{\vn{A}}}))
\end{eqnarray*}
Finaly, since $\det(1-A)=\det(1-\psi(A))$ for all isomorphism $\psi$, we obtain $\sca[\textnormal{hyp}]{a^{\phi}}{c}=\sca[\textnormal{hyp}]{a}{c}$. We deduce that for all $\de{c}\in\cond{A}^{\pol}$, $\de{a^{\phi}}\poll\de{c}$, and therefore $\de{a^{\phi}}\in\cond{A}$.
\end{proof}

\begin{definition}
Let $\de{a,b}$ be hyperfinite projects of equal carrier $p$, and $\lambda\in\realN$. We define $\de{a+\lambda b}$ as the hyperfinite project $a+\lambda b\cdot +\cdot \alpha\oplus\lambda\beta+A\oplus B$, of dialect $\vn{A\oplus B}$ and carrier $p$.
\end{definition}

\begin{definition}
A conduct has the \emph{inflation property} when for all $\de{a}\in\cond{A}$, and all $\lambda\in\realN$, the hyperfinite project $\de{a+\lambda 0}$ belongs to $\cond{A}$, where $\de{0}$  is the project  $0\cdot+\cdot 1_{\complexN} +0$ whose carrier is equal to the carrier of $\de{a}$.
\end{definition}

The following proposition shows that this combinatorial definition -- considered in the author's work on interaction graphs \cite{seiller-goia} -- is equivalent to Girard's definition \cite{goi5}.

\begin{proposition}
If $\cond{A}$ has the inflation property, then for all element $\de{a}=(p,a,\vn{A},\alpha,A)$ in $\cond{A}$, for all finite von Neumann algebra $\vn{B}$ and all injective $\ast$-morphism $\phi:\vn{A}\rightarrow \vn{B}$, the hyperfinite project $\de{a}^{\phi}=(p,a,\vn{B},\beta,(\text{Id}\otimes\phi)(A))$ such that $\beta\circ\phi=\alpha$ is an element of $\cond{A}$.
\end{proposition}

\begin{proof}
Let $p$ be the projection which is the image of the identity through $\phi$. Then $(\text{Id}\otimes\phi)(A)=p(\text{Id}\otimes\phi)(A)p$. Moreover,
\begin{equation*}
\beta(1_{\vn{B}})=\beta(p+(1_{\vn{B}}-p))=\beta(p)+\beta(1_{\vn{B}}-p)=\beta(\phi(1_{\vn{A}}))+\beta(1_{\vn{B}}-p)=\alpha(1_{\vn{A}})+\beta(1_{\vn{B}}-p)
\end{equation*}
Let $\de{c}\in\cond{A}^{\pol}$. We notice that:
\begin{eqnarray*}
\det(1-\text{Id}\otimes\phi\otimes\text{Id}(A^{\dagger})C^{\ddagger_{\vn{B}}})&=&\det(1-(1\otimes p\otimes 1)\text{Id}\otimes\phi\otimes\text{Id}(A^{\dagger})(1\otimes p\otimes 1)C^{\ddagger_{\vn{B}}})\\
&=&\det(1-\text{Id}\otimes\phi\otimes\text{Id}(A^{\dagger})(1\otimes p\otimes 1)C^{\ddagger_{\vn{B}}}(1\otimes p\otimes 1))\\
&=&\det(1-\text{Id}\otimes\phi\otimes\text{Id}(A^{\dagger})C^{\ddagger_{p\vn{B}p}})\\
&=&\det(1-\text{Id}\otimes\phi\otimes\text{Id}(A^{\dagger})C^{\ddagger_{\phi(\vn{A})}})\\
&=&\det(1-\text{Id}\otimes\phi\otimes\text{Id}(A^{\dagger}C^{\ddagger_{\vn{A}}}))\\
&=&\det(1-A^{\dagger}C^{\ddagger_{\vn{A}}})
\end{eqnarray*}
We can now compute $\sca[\textnormal{hyp}]{c}{a^{\phi}}$:
\begin{eqnarray*}
\sca[\textnormal{hyp}]{c}{a^{\phi}}&=&c\beta(1_{\vn{B}})+a\gamma(1_{\vn{C}})-\log(\det(1-(\text{Id}\otimes\phi(A))^{\dagger}C^{\ddagger_{\vn{B}}}))\\
&=&c(\alpha(1_{\vn{A}})+\lambda)+a\gamma(1_{\vn{C}})-\log(\det(1-(\text{Id}\otimes\phi\otimes\text{Id}(A^{\dagger}))C^{\ddagger_{\vn{B}}}))\\
&=&c(\alpha(1_{\vn{A}})+\lambda)+a\gamma(1_{\vn{C}})-\log(\det(1-A^{\dagger}C^{\ddagger_{\vn{A}}}))\\
&=&c(\alpha(1_{\vn{A}})+\lambda)+a\gamma(1_{\vn{C}})-\log(\det(1-(A\oplus 0)^{\dagger}C^{\ddagger_{\vn{A}\oplus\complexN}}))\\
&=&\sca[\textnormal{hyp}]{c}{a+\lambda 0}
\end{eqnarray*}
We have shown that $\sca[\textnormal{hyp}]{c}{a^{\phi}}=\sca[\textnormal{hyp}]{c}{a+\lambda 0}$. Since $\cond{A}$ has the inflation property and $\de{a}\in\cond{A}$, we have that $\sca[\textnormal{hyp}]{c}{a+\lambda 0}\neq 0,\infty$ for all $\de{c}\in\cond{A}^{\pol}$. Thus $\de{a^{\phi}}\poll\de{c}$ for all $\de{c}\in\cond{A}^{\pol}$, which implies that $\de{a^{\phi}}\in\cond{A}$.
\end{proof}

\begin{definition}
A \emph{dichology} is a conduct $\cond{A}$ such that both $\cond{A}$ and $\cond{A^{\pol}}$ have the inflation property. A dichology $\cond{A}$ is \emph{proper} when both $\cond{A}$ and $\cond{A}^{\pol}$ are non-emtpy.
\end{definition}

\begin{definition}
Let $\de{a}$ be a hyperfinite projet of carrier $p$, and $q$ a projection such that $pq=0$. We define $\de{a}_{p+q}$ as the hyperfinite project $a\cdot+\cdot \alpha+(A+0)$ of carrier $p+q$.

If $\cond{A}$ is a conduct of carrier $p$, we define $\cond{A}_{p+q}=\{\de{a}_{p+q}~|~\de{a}\in\cond{A}\}^{\pol\pol}$.
\end{definition}

\begin{definition}\label{def_additives}
Let $\cond{A,B}$ be two conducts of respective disjoint carriers $p,q$. We define:
\begin{eqnarray*}
\cond{A\with B}&=&((\cond{A}^{\pol})_{p+q})^{\pol}\cap ((\cond{B}^{\pol})_{p+q})^{\pol}\\
\cond{A\oplus B}&=&((\cond{A}_{p+q})^{\pol\pol}\cup (\cond{B}_{p+q})^{\pol\pol})^{\pol\pol}
\end{eqnarray*}
\end{definition}

\begin{proposition}
If $\cond{A,B}$ are dichologies of disjoint carriers, the conducts $\cond{A\otimes B}, \cond{A\with B}, \cond{A\oplus B}$ and $\cond{A\multimap B}$ are dichologies.
\end{proposition}

\begin{proposition}[Distributivity]\label{distributivity}
For any dichologies $\cond{A,B,C}$, and delocations $\phi,\psi,\theta,\rho$ of $\cond{A},\cond{A},\cond{B},\cond{C}$ respectively, there is a project $\de{distr}$ in the dichology 
$$\cond{((\phi(A)\!\multimap\! \theta(B))\!\with\! (\psi(A)\!\multimap\! \rho(C)))\!\multimap\! (A\!\multimap\! (B\!\with\! C))}$$
\end{proposition}

\begin{definition}
Given two hyperfinite projects $\de{a}=a\cdot+\cdot\alpha+A$ and $\de{b}=b\cdot+\cdot\beta+B$, we define the hyperfinite project $\de{a+b}$:
$$\de{a+b}=a+b\cdot+\cdot\alpha\oplus\beta+A\oplus 0+0\oplus B$$
\end{definition}

\begin{lemma}
If $\cond{A,B}$ are proper dichologies, then $\cond{A+B}=\{\de{a}_{p+q}+\de{b}_{p+q}~|~\de{a}\in\cond{A},\de{b}\in\cond{B}\}$ is such that $\cond{A+B}\subset\cond{A\with B}$.
\end{lemma}

\subsubsection{Exponentials.}

Exponential connectives are defined through the notion of \emph{perennialization}. We will not justify this definition nor explain why it indeed yields exponential connectives, the interested reader can find those in our paper on exponentials in interaction graphs \cite{seiller-goie}. We will here only briefly describe the particular perennialization used by Girard \cite{goi5}.

\begin{definition}
A perennialization is an isomorphism $\Phi:\infhyp\otimes\finhyp\rightarrow\infhyp$. 
\end{definition}

\begin{definition}
A hyperfinite project $\de{a}=a\cdot+\cdot \alpha+A$ is \emph{balanced} when $a=0$, $\vn{A}$ is a finite factor of type $\text{I}$, and $\alpha$ is the normalised trace on $\vn{A}$. If $\de{a}$ is balanced with dialect $\vn{M}_{k}(\complexN)$, and $\theta:\vn{M}_{k}(\complexN)\rightarrow \finhyp$ is a trace-preserving $\ast$-isomorphism, we will abusively write $\de{a}^{\theta}$ as the \enquote{project\footnote{It is not exactly a hyperfinite project since its dialect is not an algebra of type $\text{I}$.}} $a\cdot +\cdot \tr+\text{Id}\otimes\theta(A)$, where $\tr$ is the normalised trace on $\finhyp$.
\end{definition}

\begin{definition}
If $\cond{A}$ is a dichology and $\Phi$ a perennialization, we define $\cond{\sharp_{\Phi} A}$ as the set:
\begin{equation*}
\sharp_{\Phi}\cond{A}=\{\de{\oc_{\Phi}a^{\theta}}=0\cdot+\cdot 1_{\complexN}+\Phi(\text{Id}\otimes\theta(A))~|~\de{a}\in\cond{A}\text{ balanced}, \theta: \vn{A}\rightarrow\finhyp\text{ trace-preserving $\ast$-iso}\}
\end{equation*}
We can then define the conducts $\oc_{\Phi}\cond{A}=(\cond{\sharp A})^{\pol\pol}$ and $\wn_{\Phi}\cond{A}=(\cond{\sharp(A^{\pol})})^{\pol}$.
\end{definition}

The morphism used in Girard's hyperfinite GoI model is defined from a group action. The group is chosen so as to possess a number of properties: it is an infinite conjugacy class (I.C.C.) and amenable group which contains the free monoid on two elements. Groups that are I.C.C. and amenable are of particular interest in the theory of von Neumann algebra as the group von Neumann algebra $\vn{G}(G)$ of a non-trivial I.C.C. group $G$ is a type {II}$_{1}$ factor \cite[Proposition {V}.7.9, page 367]{takesaki1}, while amenability of $G$ implies the hyperfiniteness of $\vn{G}(G)$ \cite[Theorem {XIII}.4.10, page 71]{takesaki3}. At a first glance, the existence of such a group is not clear, as the typical example of a non-amenable group is the free group on two generators.

Let us denote by $\integerN^{|\integerN|}$ the group of almost-everywhere null functions $\integerN\rightarrow\integerN$ with point wise sum. We can then define an action of the group $\integerN$ on $\integerN^{|\integerN|}$ by translation: we define $\alpha:\integerN\rightarrow \mathcal{A}ut(\integerN^{|\integerN|})$ by $\alpha(p):(x_{n})_{n\in\integerN}\mapsto(x_{n+p})_{n\in\integerN}$. We now consider the group $\mathfrak{G}$ defined as the semi-direct product, or crossed product, of  $\integerN^{|\integerN|}$ by $\integerN$ along the action  $\alpha$. Elements of $\mathfrak{G}$ are pairs $((x_{n})_{n\in\integerN},p)$ where the first element is in $\integerN^{|\integerN|}$ and the second in $\integerN$, and the composition is defined as:
\begin{equation*}
((x_{n})_{n\in\integerN},p).((y_{n})_{n\in\integerN},q)=((x_{n})_{n\in\integerN}+(\alpha(p)((y_{n})))_{n\in\integerN},p+q)=((x_{n}+y_{n+p})_{n\in\integerN}, p+q)
\end{equation*}
As a semi-direct product of amenable groups, $\mathfrak{G}$ is an amenable group. It is moreover I.C.C. since, if $x=((x_{n}),p)$ is different from $((0),0)$, the conjugacy class of  $x$ contains the elements  $((\delta_{n,0}),k)^{-1}x((\delta_{n,0}),k)$ for all $k\in\naturalN$. But $((\delta_{n,0}),k)^{-1}=((-\delta_{n,-k}),-k)$, and therefore $((\delta_{n,0}),k)^{-1}x((\delta_{n,0}),k)=((x_{n-k}+\delta_{n,p}-\delta_{n,-k}))_{n\in\integerN},p)$. Thus the conjugacy class of $x$ is infinite since those elements are pairwise distinct.

Lastly, one can find a copy of the free monoid on two elements in $\mathfrak{G}$. Let us first define $\mathbf{a}=((\delta_{n,0})_{n},0)$ and $\mathbf{b}=((0)_{n},1)$. We can then compute 
$$((a_{k})_{k},p)\mathbf{b}=((a_{k})_{k},p+1)~~~~~~~~((a_{k})_{k},p)\mathbf{a}=((a_{k}+\delta_{k,0})_{k},p)$$ 
We can use these equalities to show:
\begin{equation*}
\mathbf{a}^{p_{k}}\mathbf{b}^{q_{k}}\mathbf{a}^{p_{k-1}}\mathbf{b}^{q_{k-1}}\dots \mathbf{a}^{p_{1}}\mathbf{b}^{q_{1}}=((\bar{p}_{n})_{n},\sum_{i=1}^{k}q_{i})
\end{equation*}
where $\bar{p}_{n}=p_{i}$ when $n=\sum_{j=1}^{i} q_{j}$, and $\bar{p}_{n}=0$ otherwise. This shows that the submonoid generated by $\mathbf{a}$ and $\mathbf{b}$ in $\mathfrak{G}$ is free.

For instance, the word $\mathbf{a}^{2}\mathbf{b}^{1}\mathbf{a}^{48}\mathbf{b}^{2}$ is equal to $((x_{n}),3)$ where $(x_{n})$ is the sequence defined by $x_{2}=48$, $x_{3}=2$ and $x_{n}=0$ for all $n\neq 2,3$.


This shows that:
\begin{proposition}[Girard \cite{goi5}]
The group $\integerN^{|\integerN|}\rtimes \integerN$ is amenable, I.C.C. and contains the free monoid on two generators.
\end{proposition}

The definition of the perennialization used by Girard \cite{goi5} is built on the \emph{crossed product algebra} which generalises the semi-direct product of groups.

\begin{definition}[Crossed product]\label{defcrossedprod}
Let $\vn{M}\subset\B{\hil{H}}$ be a von Neumann algebra, $G$ a locally compact group, and $\alpha$ an action of $G$ on $\vn{M}$. Let $\hil{K}=L^{2}(G,\hil{H})$ be the Hilbert space of square-summable $\hil{H}$-valued functions on $G$. We define representations $(\hil{K},\pi_{\alpha})$ and $(\hil{K},\lambda)$ of $\vn{M}$ and $G$ respectively.
\begin{eqnarray*}
(\pi_{\alpha}(x).\xi)(g)&=&\alpha(g)^{-1}(x)\xi(g)\\
(\lambda(g).\xi)(h)&=&\xi(g^{-1}h)
\end{eqnarray*}
Then the von Neumann algebra on $\hil{K}$ generated by $\pi_{\alpha}(\vn{M})$ and $\lambda(G)$ is called the crossed product of $\vn{M}$ by $\alpha$ and is denoted by $\vn{M}\rtimes_{\alpha} G$.
\end{definition}

Now, if $A$ is an operator in $\infhyp\otimes\finhyp$ and ${\tt M}$ denotes the free monoid generated by $\mathbf{a}$ and $\mathbf{b}$, we use the fact that there exists an isomorphism between $\finhyp$ and $\otimes_{\omega\in\mathbf{a}{\tt M}} \finhyp$ to obtain an operator $\bar{A}$ in $\infhyp\otimes \finhyp^{{\tt M}}$. This operator embeds as an element $\pi_{\alpha}(\bar{A})$ of the crossed product algebra $\infhyp\otimes \finhyp^{\mathfrak{G}}\rtimes\mathfrak{G}$. Since $\mathfrak{G}$ is I.C.C. and amenable, the crossed product $\finhyp^{\mathfrak{G}}\rtimes\mathfrak{G}$ is isomorphic to $\finhyp$. Moreover, $\infhyp\otimes\finhyp$ is isomorphic to $\infhyp$, and we can thus find an isomorphism $\Psi$ from $\infhyp\otimes \finhyp^{\mathfrak{G}}\rtimes\mathfrak{G}$ into $\infhyp$. Defining $\Omega(A)=\Psi(\pi_{\alpha}(\bar{A}))$, we easily check that $\Omega$ defines an injective morphism from $\infhyp\otimes\finhyp$ to $\infhyp$.

\section{Subjective Truth and Matricial GoI}

\subsection{Success and Bases} 

In geometry of Interaction, as in the theory of proof structures \cite{ll}, in game semantics \cite{hylandong} or in classical realisability \cite{krivine1,krivine2}, one needs to characterise those elements which correspond to proofs: proof nets (i.e. satisfying the correctness criterion), winning strategies, or proof-like terms. In GoI models, these \enquote{proof-like terms}, or \emph{winning strategies} are called \emph{successful projects}. In previous GoI models a successful project was defined as a partial symmetry. This definition was quite satisfying, but some of its important properties relied on the fact that the model depended on a chosen MASA $\vn{A}$, i.e. it relied on the fact that the constructions were basis-dependent (i.e. operators are chosen in the normalising groupoid of $\vn{A}$ only).

In Girard's hyperfinite model, constructions are no longer basis-dependent: the operators considered are no longer restricted to those elements that are in the normalising groupoid of a MASA, but can be any hermitian operator of norm at most $1$. By going to this more general setting, defining successful projects as partial symmetries is no longer satisfying. The reason for this is quite easy to understand. Indeed, a satisfying notion of success should verify two essential properties. The first of these is that it should \enquote{compose}, i.e. the execution of two successful projects should be a successful project. The second is that it should be \enquote{coherent}, i.e. two orthogonal projects cannot be simultaneously successful.

Since we are no longer restricted to operators in a chosen normalising groupoid, the definition of successful projects as partial symmetries now lacks these two essential properties. This can be illustrated by easy examples on matrices (to obtain examples in the hyperfinite factor, use your favourite embedding). For instance, let us consider the following matrices:

\begin{equation*}
u= \left(\begin{array}{ccc}
	0 & 1 & 0\\
	1 & 0 & 0\\
	0 & 0 & 0
	\end{array}\right)~~~~~~
v= \left(\begin{array}{ccc}
	0 & \sqrt{\frac{1}{2}} & -\sqrt{\frac{1}{2}}\\
	\sqrt{\frac{1}{2}} & 0 & 0\\
	-\sqrt{\frac{1}{2}} & 0 & 0
	\end{array}\right)
\end{equation*}
One can check that $u,v$ are partial symmetries: it is obvious for $u$, and the following computation shows it for $v$.
\begin{equation*}
vv^{\ast}=v^{2}=\left(\begin{array}{ccc}
	1 & 0 & 0\\
	0 & \frac{1}{2} & -\frac{1}{2}\\
	0 & -\frac{1}{2} & \frac{1}{2}
	\end{array}\right)
	=\left(\begin{array}{ccc}
	1 & 0 & 0\\
	0 & \frac{1}{2} & -\frac{1}{2}\\
	0 & -\frac{1}{2} & \frac{1}{2}
	\end{array}\right)^{2}
\end{equation*}
However, their product is not a partial isometry (hence not a partial symmetry), which shows that the notion do not compose.
\begin{equation*}
uv=\left(\begin{array}{ccc}
	0 & 1 & 0\\
	1 & 0 & 0\\
	0 & 0 & 0
	\end{array}\right)\left(\begin{array}{ccc}
	0 & \sqrt{\frac{1}{2}} & -\sqrt{\frac{1}{2}}\\
	\sqrt{\frac{1}{2}} & 0 & 0\\
	-\sqrt{\frac{1}{2}} & 0 & 0
	\end{array}\right)
	=
	\left(\begin{array}{ccc}
	\sqrt{\frac{1}{2}} & 0 & 0\\
	0 & \sqrt{\frac{1}{2}} & -\sqrt{\frac{1}{2}}\\
	0 & 0 & 0
	\end{array}\right)
\end{equation*}
Moreover, the computation of the determinant of $1-uv$ shows that the notion is not coherent, since one can define from them two orthogonal projects.
\begin{equation*}
\left|\begin{array}{ccc}
	1-\sqrt{\frac{1}{2}} & 0 & 0\\
	0 & 1-\sqrt{\frac{1}{2}} & \sqrt{\frac{1}{2}}\\
	0 & 0 & 1
	\end{array}\right|
	=
	(1-\sqrt{\frac{1}{2}})^{2}\neq 0,\infty
\end{equation*}

In order to obtain a good notion of successful project, we will have to restrict ourselves to a class of partial symmetries which is closed under sum and composition. As shown in \autoref{lemmasumpartialiso} and \autoref{lemmaproductpartialiso}, sums and products of partial isometries in the normalising groupoid of a MASA $\vn{A}$ is again a partial isometry\footnote{In the case of the sum, one has to impose a condition on domains and codomains.} in the normalising groupoid of $\vn{A}$. This will be enough to show that if $u$ and $v$ are partial symmetries in $\gn{\vn{A}}$, then $u\plug v$ is a partial symmetry in $\gn{\vn{A}}$ (\autoref{composi50}).

In the finite-dimensional case, this amounts to choosing a basis. Indeed, the complete classification of MASAs in $\B{\hil{H}}$ (\autoref{classificationmasasbh}) shows that when $\hil{H}$ is of finite dimension the MASAs of $\B{\hil{H}}$ are exactly the diagonal MASAs: the set of diagonal matrices in a fixed basis. One can therefore define a \emph{subjective} notion of successful projects, i.e. a notion of success that depends on the choice of a basis. An operator is then \emph{successful w.r.t. $\mathcal{B}$} when it is a partial symmetry in the normalising groupoid of the algebra $\vn{D}_{\mathcal{B}}$ of diagonal operators in the basis $\mathcal{B}$. The composition of such partial symmetries can be shown to be itself a partial symmetry in the normalising groupoid of $\vn{D}_{\mathcal{B}}$ and the definition of success is therefore consistent with the execution. However, we are still unable to show the coherence of this definition: given two partial symmetries $u,v$ in $\gn{\vn{D}_{\mathcal{B}}}$, the logarithm of the determinant of $1-uv$ is not necessarily equal to $0$ or $\infty$. Once again, it is enough to consider matrices to illustrate this fact, and we will give an example with $2\times 2$ matrices. Let $u$ and $v$ be the following matrices:
\begin{equation*}
u=\left(\begin{array}{cc}
	0 & -1\\
	-1 & 0
	\end{array}\right)
	~~~~~~
v=\left(\begin{array}{cc}
	0 & 1\\
	1 & 0
	\end{array}\right)
\end{equation*}
Then $\det(1-uv)=4$, i.e. $-\log(\det(1-uv))\neq0,\infty$. 

The issue here arises from the fact that one cannot distinguish between the identity and its opposite, i.e. the definition does not exclude negative coefficients. The solution proposed by Girard \cite{goi5} is to consider a notion of success that depends on a representation of the algebra: a successful operator will then have its operators $u$ defined as the operator induced from a measure-preserving transformation on a measured space. We expose the precise definition in \autoref{verite52}, but we will first study another way to bypass the problem just exposed. It corresponds to an old version of Girard's hyperfinite GoI model, which we will refer to as the \emph{matricial} GoI model, in which orthogonality is slightly modified. In this GoI model, it is possible to keep the notion of successful projects as partial symmetries in the normalising groupoid of a MASA $\vn{A}$ since the change of orthogonality bypasses the issue with coherence. This GoI model will be related to the matricial GoI model later on.

\subsection{Matricial GoI}

The matricial GoI model is based on the same notion of projects as the hyperfinite GoI model. The two constructions essentially differ on the measurement $\sca[]{\cdot}{\cdot}$ which is used to defines the orthogonality relation. Notice that all constructions on hyperfinite projects are the same in both models.

\begin{definition}
A \emph{dialectal operator} of carrier $p^{\ast}=p^{2}=p\in\infhyp$ and dialect $\vn{A}$ a finite von Neumann algebra of type $\text{I}$ is a pair $(A,\alpha)$ where:
\begin{enumerate}
\item $A^{\ast}=A\in p\infhyp p\otimes\vn{A}$ is an hermitian operator such that $\norm{A}\leqslant 1$;
\item $\alpha$ is a pseudo-trace on $\vn{A}$.
\end{enumerate}
\end{definition}

For all von Neumann algebras $\vn{A},\vn{B}$ we consider the isomorphims:
\begin{eqnarray*}
(\cdot)^{\dagger_{\vn{B}}}&:& \infhyp\otimes\vn{A}\rightarrow \infhyp\otimes\vn{A\otimes B}\\
(\cdot)^{\ddagger_{\vn{A}}}&:& \infhyp\otimes\vn{B}\rightarrow \infhyp\otimes\vn{A\otimes B}
\end{eqnarray*}
defined on simple tensors as follows:
\begin{eqnarray*}
(x\otimes a)^{\dagger_{\vn{B}}}&=& x\otimes a\otimes 1_{\vn{B}}\\
(x\otimes b)^{\ddagger_{\vn{A}}}&=& x\otimes 1_{\vn{A}}\otimes b
\end{eqnarray*}

\begin{definition}[ldet]
Let $A\in \infhyp\otimes\vn{A}$ be a dialectal operator of norm strictly less than $1$, let $\tr$ be a trace on $\infhyp$ and $\alpha$ be a pseudo-trace on the dialect $\vn{A}$. We define:
\begin{equation*}
\ldet_{\tr\otimes\alpha}(1-A)=\sum_{i=1}^{\infty} \frac{\tr\otimes\alpha(A^{k})}{k}
\end{equation*}
In most cases, the context will be clear and we will forget the subscripts.
\end{definition}

\begin{definition}
We define the measurement between two dialectal operators $A,B$ of respective carriers $p,q$ and dialects $\vn{A},\vn{B}$ as:
\begin{equation*}
\measmat{A,B}=\left\{\begin{array}{ll}
	\ldet(1-A^{\dagger_{\vn{B}}}B^{\ddagger_{\vn{A}}})&\text{ when $\specrad{A^{\dagger_{\vn{B}}}B^{\ddagger_{\vn{A}}}}<1$}\\
	\infty &\text{ otherwise}
	\end{array}\right.
\end{equation*}
\end{definition}

\begin{lemma}\label{sumofldet}\label{adjonctionspectrale}
Let $u,v\in\infhyp\otimes\vn{A}$ and $\alpha$ be a pseudo-trace on $\vn{A}$. Then, supposing the series converge: 
\begin{equation*}
\ldet(1-(u+v-uv))=\ldet(1-u)+\ldet(1-v)
\end{equation*}
\end{lemma}

\begin{proof}
Supposing the series converge: 
\begin{eqnarray*}
\ldet(1-(u+v))&=&-\tr(\log((1-u)(1-v)))\\
&=&-\tr(\log(1-u)+\log(1-v))\\
&=&-\tr(\log(1-u))-\tr(\log(1-v))
\end{eqnarray*}
Thus $\ldet(1-(u+v-uv))=\ldet(1-u)+\ldet(1-v)$.
\end{proof}

\begin{lemma}\label{dialectoutldet}
Choose $u\in\infhyp\otimes\vn{A}$, a trace $\tr$ on $\infhyp$ and a pseudo-trace $\alpha$ on an dialect $\vn{A}$. For all dialect $\vn{B}$ and pseudo-trace $\beta$ on $\vn{B}$, we have:
\begin{equation*}
\ldet(1-u\otimes 1_{\vn{B}})=\beta(1_{\vn{B}})\ldet(1-u)
\end{equation*}
\end{lemma}

\begin{proof}
By definition:
\begin{eqnarray*}
\ldet(1-u\otimes 1_{\vn{B}})&=&\sum_{k=1}^{\infty} \frac{\tr\otimes\alpha\otimes\beta((u\otimes1_{\vn{B}})^{k})}{k}\\
&=&\sum_{k=1}^{\infty} \frac{\tr\otimes\alpha\otimes\beta(u^{k}\otimes 1_{\vn{B}})}{k}\\
&=&\sum_{k=1}^{\infty} \frac{\tr\otimes\alpha(u^{k})\beta(1_{\vn{B}})}{k}\\
&=&\beta(1_{\vn{B}})\sum_{k=1}^{\infty} \frac{\tr\otimes\alpha((u)^{k})}{k}
\end{eqnarray*}
\end{proof}

\begin{definition}
Two dialectal operators $(A,\alpha),(A',\alpha')$ are universally equivalent if for all dialectal operator $(B,\beta)$, we have $\measmat{(A,\alpha),(B,\beta)}=\measmat{(A',\alpha'),(B,\beta)}$.
\end{definition}

\begin{lemma}[Variants]\label{variantgdi50}
Let $(A,\alpha)$ be a dialectal operator of dialect $\vn{A}$, and $\varphi : \vn{A}\rightarrow\vn{C}$ a unital isomorphism of von Neumann algebras. Then $((\text{Id}_{\infhyp}\otimes\varphi)(A),\alpha\circ\varphi^{-1})$ is universally equivalent to $(A,\alpha)$.
\end{lemma}

\begin{proof}
Let $(B,\beta)$  be a dialectal operator. Since $B^{\ddagger_{\vn{C}}}=(\text{Id}_{\infhyp}\otimes\varphi)(B^{\ddagger_{\vn{A}}})$, since $\text{Id}_{\infhyp}\otimes\varphi$ is a unital isomorphism of von Neumann algebras and since isomorphisms between C$^{\ast}$-algebras are isometric, we obtain:
\begin{eqnarray*}
\norm{[(\text{Id}_{\infhyp}\otimes\varphi)(A)]^{\dagger_{\vn{B}}}B^{\ddagger_{\vn{C}}}}&=&\norm{(\text{Id}_{\infhyp}\otimes\varphi)(A^{\dagger_{\vn{B}}}B^{\ddagger_{\vn{A}}})}\\
&=&\norm{A^{\dagger_{\vn{B}}}B^{\ddagger_{\vn{A}}}}
\end{eqnarray*}
Similarly, we show that for all $k$:
\begin{equation*}
\norm{([(\text{Id}_{\infhyp}\otimes\varphi)(A)]^{\dagger_{\vn{B}}}B^{\ddagger_{\vn{C}}})^{k}}=\norm{(A^{\dagger_{\vn{B}}}B^{\ddagger_{\vn{A}}})^{k}}
\end{equation*}
Consequently, since $\specrad{u}=\lim_{k\rightarrow\infty} \norm{u^{k}}^{\frac{1}{k}}$ we obtain that:
\begin{equation*}
\specrad{[(\text{Id}_{\infhyp}\otimes\varphi)(A)]^{\dagger_{\vn{B}}}B^{\ddagger_{\vn{C}}}}=\specrad{A^{\dagger_{\vn{B}}}B^{\ddagger_{\vn{A}}}}
\end{equation*}
Moreover, if $\specrad{A^{\dagger_{\vn{B}}}B^{\ddagger_{\vn{A}}}}<1$:
\begin{eqnarray*}
\lefteqn{\ldet(1-[(\text{Id}_{\infhyp}\otimes\varphi)(A)]^{\dagger_{\vn{B}}}B^{\ddagger_{\vn{C}}})}\\
&=& \ldet(1-(\text{Id}_{\infhyp}\otimes \varphi\otimes 1_{\vn{B}})(A^{\dagger_{\vn{B}}}B^{\ddagger_{\vn{A}}}))\\
&= &\sum_{k=1}^{\infty}\frac{\tr([(\text{Id}_{\infhyp}\otimes \varphi\otimes 1_{\vn{B}})(A^{\dagger_{\vn{B}}}B^{\ddagger_{\vn{A}}}))]^{k})}{k}\\
&=&\sum_{k=1}^{\infty}\tr((\text{Id}_{\infhyp}\otimes \varphi\otimes 1_{\vn{B}})(\frac{[A^{\dagger_{\vn{B}}}B^{\ddagger_{\vn{A}}}]^{k}}{k}))\\
&=&\sum_{k=1}^{\infty}(\tr_{\infhyp}\otimes(\alpha\circ\varphi^{-1})\otimes \beta)((\text{Id}_{\infhyp}\otimes \varphi\otimes 1_{\vn{B}})(\frac{[A^{\dagger_{\vn{B}}}B^{\ddagger_{\vn{A}}}]^{k}}{k}))\\
&=&\sum_{k=1}^{\infty}(\tr_{\infhyp}\otimes\alpha\otimes\beta)(\frac{[A^{\dagger_{\vn{B}}}B^{\ddagger_{\vn{A}}}]^{k}}{k})\\
&=&\ldet(1-A^{\dagger_{\vn{B}}}B^{\ddagger_{\vn{A}}})
\end{eqnarray*}
Finaly, we have just shown that $\measmat{((\text{Id}_{\infhyp}\otimes\varphi)(A),\alpha\circ\varphi^{-1}),(B,\beta)}=\measmat{(A,\alpha),(B,\beta)}$.
\end{proof}

\begin{definition}
Let $(A,\alpha),(B,\beta)$ be two dialectal operators of carrier $p+r,r+q$ and dialects $\vn{A,B}$ such that $\specrad{AB}<1$. The execution of $A$ and $B$ is defined as the dialectal operator $(A\plug B,\alpha\otimes\beta)$ of carrier $p+q$ and dialect $\vn{A}\otimes\vn{B}$ where $A\plug B$ is defined as:
\begin{equation*}
A\plug B=(pA^{\dagger_{\vn{B}}}+q)(1-B^{\ddagger_{\vn{A}}}A^{\dagger_{\vn{B}}})^{-1}(p+B^{\ddagger_{\vn{A}}}q)
\end{equation*}
When $r=0$, we have $B^{\ddagger_{\vn{A}}}A^{\dagger_{\vn{B}}}=0$, and we will write 
\begin{equation*}
A\cup B=A\plug B=pA^{\dagger_{\vn{B}}}p+qB^{\ddagger_{\vn{A}}}q=A^{\dagger_{\vn{B}}}+B^{\ddagger_{\vn{B}}}
\end{equation*}
\end{definition}

\begin{proposition}[Adjunction, Girard \cite{pointaveugle2}]
Let $(F,\phi),(G,\gamma),(H,\rho)$ be dialectal operators of respective carriers $p,q,r$ such that $qr=0$. Then:
\begin{equation*}
\measmat{F,G\cup H}=\rho(1_{\vn{H}})\measmat{F,G}+\measmat{H,F\plug G}
\end{equation*}
\end{proposition}

\begin{definition}
If $(A,\alpha)$ and $(B,\beta)$ are dialectal operators of equal carrier $p$ and respective dialects $\vn{A,B}$, and if $\lambda\in\realN$ is a real number, one can define $\lambda (A,\alpha) + (B,\beta)$ as the dialectal operator $(A+B,\lambda\alpha\oplus\beta)$ of carrier $p$ and dialect $\vn{A\oplus B}$.
\end{definition}

\begin{definition}[Project]
A hyperfinite project is a pair $\de{a}=(a,(A,\alpha))$ where $(A,\alpha)$ is a dialectal operator and $a\in\realN\cup\{\infty\}$ is the wager. Following Girard's notation, we will sometimes write $\de{a}$ as $a\cdot +\cdot\alpha +A$.
\end{definition}

\begin{definition}[Variants]
Let $(A,\alpha)$, $(B,\beta)$ be dialectal operators. If there exists $\phi:\vn{A}\rightarrow\vn{B}$ an isomorphism such that $(B,\beta)=(\text{Id}_{\infhyp}\otimes\phi(A),\alpha\circ\phi^{-1})$, we will say that $(B,\beta)$ is a variant of $(A,\alpha)$.
\end{definition}

We can now define the measurement $\sca[\textnormal{mat}]{a}{b}$ between hyperfinite projects as:
$$\sca[\textnormal{mat}]{a}{b}=\alpha(1_{\vn{A}})b+\beta(1_{\vn{B}})a+\measmat{A,B}$$

We then follow the constructions of the hyperfinite GoI model described above to define multiplicative connectives, dichologies, additive connectives, and exponential connectives. Indeed, the key properties used in the constructions of connectives of linear logic are the associativity of execution and the adjunction, both of which hold in both constructions. These two slightly different GoI models actually only differ on their measurement between operators: $\meashyp{A,B}=-\log(\det(1-AB))$ in the hyperfinite GoI model and $\measmat{A,B}$ in the matricial GoI model.

\subsection{A Few Technical Lemmas}

The first two lemma will be used to prove compositionality of our notion of success (\autoref{composi50}). As explained above, it is essential that the normalising groupoid of a MASA is closed under products of partial isometries and sums of \enquote{disjoint} partial isometries.

\begin{lemma}[Products of Partial Isometries]\label{lemmaproductpartialiso}
Let $u,v$ be partial isometries, both in the normalising groupoid $\gn[\vn{N}]{\vn{P}}$ of a MASA $\vn{P}$ of a von Neumann algebra $\vn{N}$. Then $uv$ is a partial isometry, and $uv$ is in $\gn[\vn{N}]{\vn{P}}$.
\end{lemma}

\begin{proof}
Since $u$ and $v$ are in the normalising groupoid of $\vn{P}$, we have that $\proj{p}=u^{\ast}u\in\vn{P}$ and $\proj{q}=vv^{\ast}\in\vn{P}$. Moreover, since $\vn{P}$ is commutative, $\proj{p}\proj{q}=\proj{qp}$. Thus:
\begin{equation*}
(uv)(uv)^{\ast}(uv)=uvv^{\ast}u^{\ast}uv=u\proj{qp}v=u\proj{pq}v=uu^{\ast}uvv^{\ast}v=uv
\end{equation*}
We just showed that $(uv)(uv)^{\ast}(uv)=uv$, and therefore $uv$ is a partial isometry. Finally, since the projections $uv(uv)^{\ast}$ and $(uv)^{\ast}uv$ are elements of $\vn{P}$ and $uv\vn{P}(uv)^{\ast}\subset\vn{P}$, $uv$ is in the normalising groupoid of $\vn{A}$.
\end{proof}

\begin{lemma}[Sums of Partial Isometries]\label{lemmasumpartialiso}
Let $u,v$ be partial isometries in the normalising groupoid $\gn[\infhyp]{\vn{P}}$ of a MASA $\vn{P}$ in a von Neumann algebra $\vn{N}$. If $uv^{\ast}=u^{\ast}v=0$, then the sum $u+v$ is a partial isometry and belongs to $\gn[\infhyp]{\vn{P}}$.
\end{lemma}

\begin{proof}
We supposed that $(uv^{\ast})^{\ast}=vu^{\ast}=0$. This implies that $(u^{\ast}v)^{\ast}=v^{\ast}u=0$. We can then compute:
\begin{equation*}
(u+v)(u+v)^{\ast}(u+v)=(u+v)(u^{\ast}+v^{\ast})(u+v)=uu^{\ast}u+vv^{\ast}v=u+v
\end{equation*}
This shows that $u+v$ is a partial isometry. We now have to show that it is in the normalising groupoid of $\vn{P}$. The projections $uu^{\ast},u^{\ast}u,vv^{\ast},v^{\ast}v$ are all elements of $\vn{P}$, and commute one with another. If $a$ is an element of $\vn{P}$:
\begin{equation*}
(u+v)a(u+v)^{\ast}=uau^{\ast}+vav^{\ast}
\end{equation*}
Since $u,v$ are in the normalising groupoid of $\vn{P}$, we conclude from this that $(u+v)a(u+v)^{\ast}$ is the sum of two elements in $\vn{P}$, which is again an element in $\vn{P}$. This implies that $(u+v)\vn{P}(u+v)^{\ast}\subseteq\vn{P}$.\\
Finally, the projections $(u+v)(u+v)^{\ast}=uu^{\ast}+vv^{\ast}$ and $(u+v)^{\ast}(u+v)=u^{\ast}u+v^{\ast}v$ are also in $\vn{P}$ as sums of elements in $\vn{P}$.
\end{proof}

Now, we show a succession of lemmas whose final goal is a proof of the fact that nilpotent dialectal operators $(A,\alpha)$ satisfy $\ldet(1-A)=0$. This fact will be of use in both the proof of coherence (\autoref{coherence50}) and the proof of compositionality (\autoref{composi50}).

\begin{lemma}\label{nilpotenttraceis0}
Let $A$ be an operator in a factor $\vn{M}$, and $\tr$ be a trace on $\vn{M}$. If $A$ is nilpotent, $\tr(A)=0$.
\end{lemma}

\begin{proof}
Let $N$ be the degree of nilpotency of $A$. We define by induction a sequence of projections $(p_{i})_{i=1}^{N}$ and a sequence of operators $(A_{i})_{i=1}^{N}$ as follows:
\begin{itemize}
\item $A_{1}=A$ and $p_{1}$ is the projection onto the closure of the range of $A_{1}$;
\item $A_{i+1}=A_{i}p_{i}$ and $p_{i+1}$ is the projection onto the closure of the range of $A_{i+1}$.
\end{itemize}
Notice that for all $i$, the projections satisfy $p_{i+1}p_{i}=p_{i+1}$ since $p_{i}\geqslant p_{i+1}$. A simple induction shows that $A_{i+1}=A_{i}p_{i}$ is nilpotent of degree $N-i$:
\begin{itemize}
\item First, notice that $A_{2}^{N-1}=(A_{1}p_{1})^{N-1}=A^{N-1}p_{1}$ since $p_{i}A_{i}=A_{i}$ by definition of $p_{i}$. Since $p_{1}$ is the projection onto the closure of the range of $A$, each element of the form $p_{1}x$ is the limit of a sequence $(Ay_{i})_{i=0}^{\infty}$, i.e. $A^{N-1}p_{1}x$ is the limit of the sequence $(A^{N}y_{i})_{i=0}^{\infty}$ by continuity of $A$. As $A$ is nilpotent of degree $N$, the sequence is equal to $0$ everywhere, whence $A^{N-1}p_{1}=0$, and $A_{2}=(A_{1}p_{1})$ is nilpotent of degree $N-1$.
\item Then we know by induction that $A_{i}$ is nilpotent of degree $N-i+1$ and that $p_{i}$ is defined as the projection onto the closure of the range of $A_{i}$. We use the same argument and show that $A_{i+1}=A_{i}p_{i}$ is nilpotent of degree $N-i$ since $(A_{i}p_{i})^{N-i}=A_{i}^{N-i}p_{i}=0$.
\end{itemize}
As a consequence, $A_{N}$ is nilpotent of degree $1$, i.e. $A_{N}=0$.

We now use the \enquote{traciality} of the trace, i.e. that $\tr(AB)=\tr(BA)$ and the fact that $A_{i}=p_{i}A_{i}$ to show that $\tr(A)=\tr(A_{N})$:
\begin{equation*}
\tr(A)=\tr(A_{1})=\tr(p_{1}A_{1})=\tr(A_{1}p_{1})=\tr(A_{2})=\tr(p_{2}A_{2})=\tr(A_{2}p_{2})=\tr(A_{3})=\dots=\tr(A_{N})
\end{equation*}
Since $A_{N}=0$, we conclude that $\tr(A)=\tr(A_{N})=0$.
\end{proof}

\begin{lemma}\label{lemmapseudotracefactor}
Let $\vn{A}=\vn{M}_{k}(\complexN)$ be a matrix algebra, and $\alpha$ a pseudo-trace on $\vn{A}$. There exists a real number $\lambda$ such that $\alpha=\lambda\tr$ where $\tr$ is the normalised trace (i.e. $\tr(1)=1$) on $\vn{A}$.
\end{lemma}

\begin{proof}
Let us fix $\pi_{1},\dots,\pi_{k}$ a set of minimal projections of $\vn{A}=\vn{M}_{k}(\complexN)$ such that $\sum_{j=1}^{k}\pi_{j}=1$. Since the projections $\pi_{j}$ are equivalent in the sense of Murray and von Neumann, one can find for all $1\leqslant j \leqslant k$ a partial isometry $u_{j}$ such that $u_{j}u_{j}^{\ast}=\pi_{j}$ and $u_{j}^{\ast}u_{j}=\pi_{1}$. Using the \enquote{traciality} of $\alpha$ we obtain:
$$\alpha(\pi_{j})=\alpha(u_{j}u_{j}^{\ast})=\alpha(u_{j}^{\ast}u_{j})=\alpha(\pi_{1})$$ 
We now write $\lambda=k\times\alpha(\pi_{1})$. 

We now show that $\alpha(p)=\lambda\tr(p)$ for all projection $p$ in $\vn{M}_{k}(\complexN)$. If $p$ is such a projection, there exists a partial isometry $w$ between $p$ and a sum $\sum_{j=1}^{m}\pi_{j}$ where $j$ is an integer between $1$ and $k$. The \enquote{traciality} of $\alpha$ and $\tr$ then respectively imply that $\alpha(p)=\alpha(\sum_{j=1}^{m}\pi_{j})$ and $\tr(p)=\tr(\sum_{j=1}^{m}\pi_{j})$. We thus obtain:
$$\alpha(p)=\alpha(\sum_{j=1}^{m}\pi_{j})=\sum_{j=1}^{m}\alpha(\pi_{j})=\sum_{j=1}^{m}\lambda\tr(\pi_{j})=\lambda\tr(\sum_{j=1}^{m}\pi_{j})=\lambda\tr(p)$$
The two linear forms $\alpha$ and $\lambda\tr$ are therefore equal on the set of all projections and continuous: they are therefore equal on the whole algebra since the latter is generated by its projections. We are left with showing that $\lambda$ is a real number. This is a straightforward consequence of the equality $\overline{\alpha(a)}=\alpha(a^{\ast})$: if $a$ is self-adjoint, then $\alpha(a)\in\realN$; in particular, a projection $\pi_{1}$ is self-adjoint and $\tr(\pi_{1})=1/k$ is in $\realN$. Finally, $\lambda=\alpha(\pi_{1})/\tr(\pi_{1})$ is a real number.
\end{proof}

\begin{lemma}\label{normalformpseudotrace}
Let $\vn{A}=\bigoplus_{i=1}^{l} \vn{M}_{k_{i}}(\complexN)$ be a direct sum of matrix algebras, and $\alpha$ a pseudo-trace on $\vn{A}$. There exists a family $(\lambda_{i})_{i=1}^{l}$ of real numbers such that $\alpha=\bigoplus_{i=1}^{l}\lambda_{i}\tr_{\vn{M}_{k_{i}}(\complexN)}$.
\end{lemma}

\begin{proof}
We write $p_{i}$ ($i=1,\dots,l$) the projection of $\vn{A}$ onto the algebra $\vn{M}_{k_{i}}(\complexN)$. If $a\in \vn{A}$, we have $a=\bigoplus_{i=1}^{l}a_{i}$, from which we obtain that $\alpha(a)=\sum_{i=1}^{l}\alpha(a_{i})$. We therefore write $\alpha_{i}$ the restriction of $\alpha$ to the algebra $\vn{M}_{k_{i}}(\complexN)$, i.e. $\alpha_{i}=\alpha(p_{i}\iota(x)p_{i})$ where $\iota:\vn{M}_{k_{i}}(\complexN)\rightarrow \vn{A}$ is the canonical injection. Since $\alpha$ is a pseudo-trace, $\alpha_{i}$ is a pseudo-trace on $\vn{M}_{k_{i}}(\complexN)$ and, by \autoref{lemmapseudotracefactor} there exists a real number $\lambda_{i}$ such that $\alpha_{i}=\lambda_{i}\tr_{\vn{M}_{k_{i}}(\complexN)}$. Finally, for all $a\in\vn{A}$, we have:
$$\alpha(a)=\sum_{i=1}^{l}\alpha(a_{i})=\sum_{i=1}^{l}\alpha_{i}(a_{i})=\sum_{i=1}^{l}\lambda_{i}\tr_{\vn{M}_{k_{i}}(\complexN)}(a_{i})=(\bigoplus_{i=1}^{l}\lambda_{i}\tr_{\vn{M}_{k_{i}}(\complexN)})(a)$$
This concludes the proof.
\end{proof}

\begin{lemma}\label{nilpimpliesdet0}
If a dialectal operator $(A,\alpha)$ is nilpotent, then $\ldet(1-A)=0$.
\end{lemma}

\begin{proof}
By the equality $\specrad{A}=\lim\inf_{n\rightarrow\infty}\norm{A^{n}}^{\frac{1}{n}}$ we have that $\specrad{A}=0$ since $\norm{A^{k}}=0$ for all big enough $k$ (i.e. $k$ greater than the nilpotency degree $N$ of $A$). Thus $\ldet(1-A)$ is computed as the series $\sum_{i=1}^{\infty} \frac{\tr\otimes\alpha(A^{n})}{n}$, which we can immediately simplify as $\sum_{i=1}^{N}\frac{\tr\otimes\alpha(A^{n})}{n}$. Let us pick an integer $k$ in $\{1,2,\dots,N\}$, and suppose that $\tr\otimes\alpha(A^{k})\neq 0$. By \autoref{normalformpseudotrace} the pseudo-trace $\alpha$ can be written as $\alpha=\oplus_{i=1}^{a} \alpha_{i}\tr_{d_{i}}$ where $\alpha_{i}$ are real numbers and the traces $\tr_{d_{i}}$ are normalised traces on type {I}$_{n}$ factors. Moreover, $A$ can be written as a direct sum $\bigoplus_{i=1}^{a} A_{i}$. Since $A$ is nilpotent, each $A_{i}$ is nilpotent. Therefore, if we prove that $\tr\otimes\tr_{d_{i}}(A_{i})=0$, we will be able to conclude. But $\tr\otimes\tr_{d_{i}}$ is a trace (not merely a pseudo-trace), and therefore $\tr\otimes\tr_{d_{i}}(A_{i})=0$ by Lemma \ref{nilpotenttraceis0}.
\end{proof}

\subsection{Subjective Truth in the Matricial GoI Model}

We are now ready to define a notion of success for the matricial GoI model. To avoid an overlap of terminology, we call here \emph{outlooks} the equivalent of Girard's \emph{viewpoints}, and \emph{promising projects} the equivalent of Girard's \emph{successful projects}. We will first show how this notion of success satisfies coherence and compositionality, and we will then explain its relation to Girard's notion of success.

\begin{definition}[Outlook]
An \emph{outlook} is a MASA in $\infhyp$.
\end{definition}

\begin{definition}[Promising Project]\label{projetsprometteurs}
An hyperfinite project $\de{a}=a\cdot+\cdot\alpha+A$ is \emph{promising w.r.t the outlook $\vn{P}$} if:
\begin{itemize}
\item \textbf{Dialect.} The dialect $\vn{A}$ is a finite factor, i.e. a matrix algebra;
\item \textbf{Pseudo-Trace.} $\alpha$ is the normalised trace on $\vn{A}$;
\item \textbf{Wager.} $\de{a}$ is wager-free: $a=0$;
\item \textbf{Symmetry.} $A$ is a partial symmetry in the normalising groupoid $\vn{P}\otimes \vn{Q}$, where $\vn{Q}$ is a MASA in $\vn{A}$;
\item \textbf{Traces.} for all projection $\pi\in\vn{P}\otimes\normalisateur[\vn{A}]{Q}$, $\tr(\pi A)=0$.
\end{itemize}
\end{definition}

We remark here that the last \enquote{trace condition} is an addition to Girard's condition of successful projects (see \autoref{remarkadditionalcondition}). This is however a quite natural condition as a successful or promising project should be understood as the representation of a set of axiom links in a proof net, i.e. it should be a symmetry without any fixpoints. Even though the alternative definition of success without this additional condition would be satisfactory (it would satisfy coherence and compositionality), it does not convey the intuitions coming from proof nets.

\begin{remark}
We showed earlier that if $\phi$ is an injective morphism from $\vn{A}$ into $\vn{B}$, then $\de{a}^{\phi}$ is universally equivalent to $\de{a}$ in the sense that $\sca[]{a}{b}=\sca[]{a^{\phi}}{b}$ for all hyperfinite project $\de{b}$. Thus, $\de{a}\in\cond{A}$ if and only if $\de{a}^{\phi}\in\cond{A}$. Moreover, if $\vn{A}$ is not a factor, it is a direct sum of matrix algebras $\vn{A}=\bigoplus_{i=1}^{k}\vn{M}_{p_{i}}(\complexN)$, and thus embeds naturally as a sub-algebra of $\vn{B}=\vn{M}_{\sum_{i}p_{i}}(\complexN)$, a factor. As a consequence, if $\de{a}$ is a promising project, we can always suppose that $\vn{A}$ is a factor. 
\end{remark}

\begin{definition}
A conduct $\cond{A}$ is \emph{correct w.r.t. the outlook $\vn{A}$} if there exists a hyperfinite project $\de{a}\in\cond{A}$ which is promising w.r.t. $\vn{A}$.
\end{definition}

We now check that the notion of promising project satisfies the essential properties: compositionality and coherence. 

\begin{proposition}[Coherence]\label{coherence50}
Let $\vn{P}$ be an outlook. The two conducts $\cond{A}$ and $\cond{A}^{\pol}$ cannot both contain a promising project w.r.t. $\vn{P}$.
\end{proposition}

\begin{proof}
Suppose that $\de{f,g}$ are promising hyperfinite projects in $\cond{A}$ and $\cond{A}^{\pol}$ respectively. 

We will show that $\ldet(1-F^{\dagger_{\vn{G}}}G^{\ddagger_{\vn{F}}})=0$. Since $\de{f}\poll\de{g}$, we know that $\rho(F^{\dagger_{\vn{G}}}G^{\ddagger_{\vn{F}}})<1$. In other terms:
\begin{equation}
\textnormal{lim inf }\norm{(F^{\dagger_{\vn{G}}}G^{\ddagger_{\vn{F}}})^{k}}^{\frac{1}{k}}<1\label{liminf2}
\end{equation}
Let $\vn{P}_{\vn{F}}$ be a MASA in $\vn{F}$ and$\vn{P}_{\vn{G}}$ a MASA in $\vn{G}$ such that $F$ and $G$ are in the normalising groupoids of $\vn{P}\vntimes\vn{P}_{\vn{F}}$ and $\vn{P}\vntimes\vn{P}_{\vn{G}}$ respectively.

Since $F$ and $G$ are partial symmetries in the normalising groupoid of $\vn{P}\vntimes\vn{P}_{\vn{F}}$ and $\vn{P}\vntimes\vn{P}_{\vn{G}}$ respectively, it is clear that $F^{\dagger_{\vn{G}}}$ and $G^{\ddagger_{\vn{F}}}$ are again partial symmetries, and we can show\footnote{Using the fact that the unit of a von Neumann algebra $\vn{N}$ is contained in all MASA of $\vn{N}$.} that they are both in the normalising groupoid of $\vn{P}\vntimes\vn{P}_{\vn{F}}\vntimes\vn{P}_{\vn{G}}$, a MASA of $\infhyp\otimes\vn{F}\otimes\vn{G}$. Moreover, they are \emph{a fortiori} partial isometries, and \autoref{lemmaproductpartialiso} ensures that $F^{\dagger_{\vn{G}}}G^{\ddagger_{\vn{F}}}$ is a partial isometry in the normalising groupoid of $\vn{P}\vntimes\vn{P}_{\vn{F}}\vntimes\vn{P}_{\vn{G}}$. One easily shows in this way that for all $k\in\naturalN$, $(F^{\dagger_{\vn{G}}}G^{\ddagger_{\vn{F}}})^{k}$ is a partial isometry. Since the norm of a partial isometry is necessarily equal to $0$ or $1$, we deduce that for all $k\in\naturalN$, $\norm{(F^{\dagger_{\vn{G}}}G^{\ddagger_{\vn{F}}})^{k}}=0$ or $\norm{(F^{\dagger_{\vn{G}}}G^{\ddagger_{\vn{F}}})^{k}}=1$. Using \autoref{liminf2}, we conclude that there exists $N\in\naturalN$ such that $\norm{(F^{\dagger_{\vn{G}}}G^{\ddagger_{\vn{F}}})^{N}}=0$, i.e. $(F^{\dagger_{\vn{G}}}G^{\ddagger_{\vn{F}}})^{N}=0$.
Finally, this shows that $F^{\dagger_{\vn{G}}}G^{\ddagger_{\vn{F}}}$ is nilpotent, which implies that $\ldet(1-F^{\dagger_{\vn{G}}}G^{\ddagger_{\vn{F}}})=0$ by \autoref{nilpimpliesdet0}.
\end{proof}

The following proposition shows that the notion of promising project is compositional. The statement contains however a small condition on the types, as we need to use elements in the orthogonal types in the proof. These conditions are however non-restrictive as they exclude a case of composition of proofs that do not arise in the interpretation of sequent calculus.

\begin{proposition}[Compositionality]\label{composi50}

Let $\cond{A,B,C}$ be conducts such that $\cond{A}$ and $\cond{C^{\pol}}$ are non-empty. If $\de{f}\in\cond{A}\multimap\cond{B}$ and $\de{g}\in\cond{B}\multimap\cond{C}$ are promising hyperfinite projects w.r.t. the outlook $\vn{P}$, then $\de{f}\plug\de{g}$ is a promising hyperfinite project w.r.t. $\vn{P}$ in the conduct $\cond{A}\multimap\cond{C}$.
\end{proposition}

\begin{proof}
Let $\de{f}=0\cdot+\cdot \phi + F$ and $\de{g}=0\cdot+\cdot \psi+G$ be the promising projects w.r.t. $\vn{A}$, respectively in $\cond{A\multimap B}$ of carrier $p+q$ and in $\cond{B\multimap C}$ of carrier $q+r$. Let $\de{h}=\de{f}\plug\de{g}=\ldet(1-F^{\dagger_{\vn{G}}}G^{\ddagger_{\vn{F}}})\cdot+\cdot\phi\otimes\psi+F\plug G$ be the hyperfinite project obtained as the execution of $\de{f}$ and $\de{g}$.
\begin{itemize}
\item \textbf{Dialect.} It is clear that the dialect $\vn{F}\vntimes\vn{G}$ is a finite factor, since it is the tensor product of two finite factors $\vn{F}$ and $\vn{G}$.
\item \textbf{Pseudo-Trace.} Since $\phi$ and $\psi$ are the normalised traces on $\vn{F}$ and $\vn{G}$ respectively, the \enquote{pseudo-trace} $\phi\otimes\psi$ is the normalised trace on $\vn{F}\otimes\vn{G}$. 
\item \textbf{Wager.} Suppose that $F^{\dagger_{\vn{G}}}G^{\ddagger_{\vn{F}}}$ is not nilpotent. By \autoref{lemmaproductpartialiso}, it is a partial isometry, which implies that its spectral radius is either equal to $1$ or $0$. However, all the iterates $(F^{\dagger_{\vn{G}}}G^{\ddagger_{\vn{F}}})^{n}$ are non-zero partial isometries, thus of norm equal to $1$. Therefore, the spectral radius of $F^{\dagger_{\vn{G}}}G^{\ddagger_{\vn{F}}}$, which is equal to $\lim\inf_{n\rightarrow\infty}\norm{(F^{\dagger_{\vn{G}}}G^{\ddagger_{\vn{F}}})^{n}}^{\frac{1}{n}}$, is necessarily equal to $1$. Moreover, the norm of $F^{\dagger_{\vn{G}}}G^{\ddagger_{\vn{F}}}$ is also equal to $1$ since it is a non-zero partial isometry. Since $\specrad{F^{\dagger_{\vn{G}}}G^{\ddagger_{\vn{F}}}}=1$, the measurement $\measmat{F,G}$ is equal to $\infty$. 

Since $\cond{A,C^{\pol}}$ are non-empty, we can chose $\de{a}\in\cond{A}$ and $\de{c}\in\cond{C^{\pol}}$ and consider $\de{f\plug a}$ and $\de{g\plug c}$ which are respectively in $\cond{B}$ and $\cond{B^{\pol}}$. Since they are orthogonal, we have that the measurement $\sca[\textnormal{mat}]{f\plug a}{g\plug c}$ is different from $0$ and $\infty$. But $\sca[\textnormal{mat}]{f\plug a}{g\plug c}=\sca[\textnormal{mat}]{f}{(g\plug c)\otimes a}=\sca[\textnormal{mat}]{f}{(g\otimes a)\plug c}=\sca[\textnormal{mat}]{f\otimes c}{g\otimes a}$. Thus $\sca[\textnormal{mat}]{f\otimes c}{g\otimes a}\neq \infty$. This implies however that $F^{\dagger_{\vn{G}}}G^{\ddagger_{\vn{F}}}$ has spectral radius strictly less than $1$ since if $\lambda$ is in the spectrum of $F^{\dagger_{\vn{G}}}G^{\ddagger_{\vn{F}}}$, it is also in the spectrum of $(F^{\dagger_{\vn{A}}}+A^{\ddagger_{\vn{F}}})^{\dagger_{\vn{G\otimes C}}}(G^{\dagger_{\vn{C}}}+C^{\ddagger_{\vn{G}}})^{\ddagger_{\vn{F\otimes A}}}$. This is a contradiction, and we can conclude that $F^{\dagger_{\vn{G}}}G^{\ddagger_{\vn{F}}}$ is nilpotent.

We can now conclude that $\ldet(1-F^{\dagger_{\vn{G}}}G^{\ddagger_{\vn{F}}})=0$ by \autoref{nilpimpliesdet0}.

\item \textbf{Symmetry.} We will abusively denote by $\proj{p},\proj{r}$ the tensor products $\proj{p}\otimes 1_{\vn{F}}\otimes 1_{\vn{G}}$ and $\proj{r}\otimes 1_{\vn{F}}\otimes 1_{\vn{G}}$ in order to simplify the expression and make the computations more readable. Since $F^{\dagger_{\vn{G}}}G^{\ddagger_{\vn{F}}}$ is nilpotent, we have: 
\begin{eqnarray*}
F\plug G&=&(\proj{p}F^{\dagger_{\vn{G}}}+\proj{r})(1-G^{\ddagger_{\vn{F}}}F^{\dagger_{\vn{G}}})^{-1}(\proj{p}+G^{\ddagger_{\vn{F}}}\proj{r})\\
&=&(\proj{p}F^{\dagger_{\vn{G}}}+\proj{r})(\sum_{k=0}^{N-1}(G^{\ddagger_{\vn{F}}}F^{\dagger_{\vn{G}}})^{k})(\proj{p}+G^{\ddagger_{\vn{F}}}\proj{r})\\
&=&\sum_{k=0}^{N-1}\proj{p}F^{\dagger_{\vn{G}}}(G^{\ddagger_{\vn{F}}}F^{\dagger_{\vn{G}}})^{k}\proj{p}+\sum_{k=1}^{N-1}\proj{r}(G^{\ddagger_{\vn{F}}}F^{\dagger_{\vn{G}}})^{k}\proj{p}\\&&+\sum_{k=0}^{N-2}\proj{p}F^{\dagger_{\vn{G}}}(G^{\ddagger_{\vn{F}}}F^{\dagger_{\vn{G}}})^{k}G^{\ddagger_{\vn{F}}}\proj{r}+\sum_{k=0}^{N-1}\proj{r}(G^{\ddagger_{\vn{F}}}F^{\dagger_{\vn{G}}})^{k}G^{\ddagger_{\vn{F}}}\proj{r}\\
&=&\sum_{k=0}^{N-1}\proj{p}F^{\dagger_{\vn{G}}}(G^{\ddagger_{\vn{F}}}F^{\dagger_{\vn{G}}})^{k}\proj{p}+\sum_{k=0}^{N-1}\proj{r}(G^{\ddagger_{\vn{F}}}F^{\dagger_{\vn{G}}})^{k}G^{\ddagger_{\vn{F}}}\proj{r}\\&&+\sum_{k=1}^{N-1}(\proj{r}(G^{\ddagger_{\vn{F}}}F^{\dagger_{\vn{G}}})^{k}\proj{p}+\proj{p}(F^{\dagger_{\vn{G}}}G^{\ddagger_{\vn{F}}})^{k}\proj{r})
\end{eqnarray*}
We define: 
\begin{eqnarray*}
t_{k}&=&F^{\dagger_{\vn{G}}}(G^{\ddagger_{\vn{F}}}F^{\dagger_{\vn{G}}})^{k}\\
t'_{k}&=&(G^{\ddagger_{\vn{F}}}F^{\dagger_{\vn{G}}})^{k}G^{\ddagger_{\vn{F}}}\\
s_{k}&=&(G^{\ddagger_{\vn{F}}}F^{\dagger_{\vn{G}}})^{k}\\
s'_{k}&=&(F^{\dagger_{\vn{G}}}G^{\ddagger_{\vn{F}}})^{k}
\end{eqnarray*}
Using the fact that $F,G$ are partial symmetries in the normalising groupoid of $\vn{P}\vntimes\vn{P}_{\vn{F}}$ and $\vn{P}\vntimes\vn{P}_{\vn{G}}$ respectively and that $\proj{p},\proj{r}\in\vn{P}\vntimes\vn{P}_{\vn{F}}\vntimes\vn{P}_{\vn{G}}$, we show that the three terms $t_{k}$, $t'_{k}$ and $s_{k}+s'_{k}$ are partial isometries in the normalising groupoid of $\vn{P}\vntimes\vn{P}_{\vn{F}}\vntimes\vn{P}_{\vn{G}}$ (using \autoref{lemmaproductpartialiso}) and are hermitian (since $F$ and $G$ are). Therefore, they are partial symmetries in the normalising groupoid of $\vn{P}\vntimes\vn{P}_{\vn{F}}\vntimes\vn{P}_{\vn{G}}$.
\begin{itemize}
\item We have that $(\proj{p}t_{k}\proj{p})^{3}=\proj{p}t_{k}\proj{p}$ since the projections $t_{k}^{2}$ are subprojections of $p$, and this implies that $(\proj{p}t_{k}\proj{p}t_{k})^{2}=\proj{p}t_{k}\proj{p}t_{k}$. Since $\proj{p}t_{k}\proj{p}t_{k}=(\proj{p}t_{k}\proj{p}t_{k})^{\ast}$, we obtain, composing by $\proj{p}$ on the right:
\begin{eqnarray}
\proj{p}t_{k}\proj{p}t_{k}&=&t_{k}\proj{p}t_{k}\proj{p}\\
\proj{p}t_{k}\proj{p}t_{k}&=&\proj{p}t_{k}\proj{p}t_{k}\proj{p}\label{linetk}
\end{eqnarray}
\item We show similarly that: 
\begin{equation}
\proj{r}t'_{k}\proj{r}t_{k}=\proj{r}t'_{k}\proj{r}t'_{k}\proj{r}\label{linetprimek}
\end{equation}
\item One can also compute the following:
\begin{equation*}
(\proj{p}s'_{k}\proj{r}+\proj{r}s_{k}\proj{p})^{3}=\proj{p}s'_{k}\proj{r}s_{k}\proj{p}s'_{k}\proj{r}+\proj{r}s_{k}\proj{p}s'_{k}\proj{r}s_{k}
\end{equation*}
Since $(\proj{p}s'_{k}\proj{r}+\proj{r}s_{k}\proj{p})^{3}=\proj{p}s'_{k}\proj{r}+\proj{r}s_{k}\proj{p}$, composing by $\proj{p}$ on the left (resp. on the right) and $\proj{r}$ on the right (resp. on the left), we obtain $\proj{p}s'_{k}\proj{r}=\proj{p}s'_{k}\proj{r}s_{k}\proj{p}s'_{k}\proj{r}$ (resp. $\proj{r}s_{k}\proj{p}=\proj{r}s_{k}\proj{p}s'_{k}\proj{r}s_{k}$). This implies:
\begin{eqnarray*}
(\proj{r}s_{k}\proj{p}s'_{k})^{2}&=&\proj{r}s_{k}\proj{p}s'_{k}\\
(\proj{p}s'_{k}\proj{r}s_{k})^{2}&=&\proj{p}s'_{k}\proj{r}s_{k}
\end{eqnarray*}
We then show that:
\begin{eqnarray*}
\proj{r}s_{k}\proj{p}s'_{k}+\proj{p}s'_{k}\proj{r}s_{k}&=&(\proj{r}s_{k}\proj{p}s'_{k}+\proj{p}s'_{k}\proj{r}s_{k})^{\ast}\\
&=&(\proj{r}s_{k}\proj{p}s'_{k})^{\ast}+(\proj{p}s'_{k}\proj{r}s_{k})^{\ast}\\
&=&s_{k}\proj{p}s'_{k}\proj{r}+s'_{k}\proj{r}s_{k}\proj{p}
\end{eqnarray*}
Composing by $\proj{p}$ (resp.$\proj{r}$) on the right, and using the fact that $G\proj{p}=0$ (resp. $F\proj{r}=0$), we obtain:
\begin{eqnarray}
\proj{r}s_{k}\proj{p}s'_{k}\proj{r}&=&\proj{r}s_{k}\proj{p}s'_{k}\label{linesk}\\
\proj{p}s'_{k}\proj{r}s_{k}\proj{p}&=&\proj{p}s'_{k}\proj{r}s_{k}\label{linesprimek}
\end{eqnarray}
\end{itemize}
Using \autoref{linetk}, \autoref{linetprimek}, \autoref{linesk} and \autoref{linesprimek}, we show that the product of two distinct terms is always equal to zero. Since $\proj{p}\proj{r}=\proj{r}\proj{p}=0$, we have five cases to consider:
\begin{itemize}
\item $(\proj{p}t_{i}\proj{p})(\proj{p}t_{j}\proj{p})$, with $i\not= j$.\\ 
We can suppose that $i<j$, since the case $j<i$ reduces to the case $i<j$ by considering the adjoints. We then have:
\begin{eqnarray*}
(\proj{p}t_{i}\proj{p})(\proj{p}t_{j}\proj{p})&=&\proj{p}t_{i}\proj{p}t_{j}\proj{p}\\
&=&(\proj{p}t_{i}\proj{p}F^{\dagger_{\vn{G}}}(G^{\ddagger_{\vn{F}}}F^{\dagger_{\vn{G}}})^{j}\proj{p}\\
&=&(\proj{p}t_{i}\proj{p}t_{i})(G^{\ddagger_{\vn{F}}}F^{\dagger_{\vn{G}}})^{j-i}\proj{p}\\
&=&(\proj{p}t_{i}\proj{p}t_{i}\proj{p})s_{j-i}\proj{p}
\end{eqnarray*}
Since $\proj{p}G^{\ddagger_{\vn{F}}}=0$, we obtain $\proj{p}s_{j-i}=0$, and finally $(\proj{p}t_{i}\proj{p})(\proj{p}t_{j}\proj{p})=0$.
\item $(\proj{r}t'_{i}\proj{r})(\proj{r}t'_{j}\proj{r})$, with $i\not= j$.\\
This case is similar to the previous one. We treat the case $j<i$:
\begin{eqnarray*}
((\proj{r}t'_{i}\proj{r})(\proj{r}t'_{j}\proj{r}))^{\ast}&=&(\proj{r}t'_{j}\proj{r})(\proj{r}t'_{i}\proj{r})\\
&=&\proj{r}t'_{j}\proj{r}t'_{i}\proj{r}\\
&=&\proj{r}t'_{j}\proj{r}(G^{\ddagger_{\vn{F}}}F^{\dagger_{\vn{G}}})^{i}G^{\ddagger_{\vn{F}}}\proj{r}\\
&=&(\proj{r}t'_{j}\proj{r}t'_{j})(F^{\dagger_{\vn{G}}}G^{\ddagger_{\vn{F}}}))^{i-j}\proj{r}\\
&=&(\proj{r}t'_{j}\proj{r}t'_{j}\proj{r})s'_{i-j}\proj{r}
\end{eqnarray*}
Since $\proj{r}F^{\dagger_{\vn{G}}}=0$, we have $\proj{r}s'_{i-j}=0$, and thus $(\proj{r}t'_{i}\proj{r})(\proj{r}t'_{j}\proj{r})=0$.
\item $(\proj{p}t_{i}\proj{p})(\proj{p}s'_{j}\proj{r})$.\\
If $j\leqslant i$, we have:
\begin{eqnarray*}
(\proj{p}t_{i}\proj{p})(\proj{p}s'_{j}\proj{r})&=&\proj{p}t_{i}\proj{p}s'_{j}\proj{r}\\
&=&\proj{p}t_{i-j}(s_{j}\proj{p}s'_{j}\proj{r})\\
&=&\proj{p}t_{i-j}(\proj{r}s_{j}\proj{p}s'_{j}\proj{r})
\end{eqnarray*}
If $i< j$, we have:
\begin{eqnarray*}
(\proj{p}t_{i}\proj{p})(\proj{p}s'_{j}\proj{r})&=&\proj{p}t_{i}\proj{p}s'_{j}\proj{r}\\
&=&\proj{p}F^{\dagger_{\vn{G}}}(G^{\ddagger_{\vn{F}}}F^{\dagger_{\vn{G}}})^{i}\proj{p}(F^{\dagger_{\vn{G}}}G^{\ddagger_{\vn{F}}})^{j}\proj{r}\\
&=&\proj{p}F^{\dagger_{\vn{G}}}(G^{\ddagger_{\vn{F}}}F^{\dagger_{\vn{G}}})^{i}\proj{p}(F^{\dagger_{\vn{G}}}G^{\ddagger_{\vn{F}}})^{i}F^{\dagger_{\vn{G}}}G^{\ddagger_{\vn{F}}}(F^{\dagger_{\vn{G}}}G^{\ddagger_{\vn{F}}})^{j-i-1}\proj{r}
\end{eqnarray*}
Since $F^{\dagger_{\vn{G}}}$ and $G^{\ddagger_{\vn{F}}}$ are hermitian and in the normalising groupoid of $\vn{P}\vntimes\vn{P}_{\vn{F}}\vntimes\vn{P}_{\vn{G}}$, and since $\proj{r}$ is an element of $\vn{P}\vntimes\vn{P}_{\vn{F}}\vntimes\vn{P}_{\vn{G}}$, we show that the operator $$F^{\dagger_{\vn{G}}}(G^{\ddagger_{\vn{F}}}F^{\dagger_{\vn{G}}})^{i}\proj{p}(F^{\dagger_{\vn{G}}}G^{\ddagger_{\vn{F}}})^{i}F^{\dagger_{\vn{G}}}$$ is an element of $\vn{P}\vntimes\vn{P}_{\vn{F}}\vntimes\vn{P}_{\vn{G}}$. It then commutes with $\proj{p}$ which is itself an element of $\vn{P}\vntimes\vn{P}_{\vn{F}}\vntimes\vn{P}_{\vn{G}}$, an abelian algebra.
Finally:
\begin{eqnarray*}
(\proj{p}t_{i}\proj{p})(\proj{p}s'_{j}\proj{r})&=&\proj{p}(F^{\dagger_{\vn{G}}}(G^{\ddagger_{\vn{F}}}F^{\dagger_{\vn{G}}})^{i}\proj{p}(F^{\dagger_{\vn{G}}}G^{\ddagger_{\vn{F}}})^{i}F^{\dagger_{\vn{G}}})G^{\ddagger_{\vn{F}}}(F^{\dagger_{\vn{G}}}G^{\ddagger_{\vn{F}}})^{j-i-1}\proj{r}\\
&=&(F^{\dagger_{\vn{G}}}(G^{\ddagger_{\vn{F}}}F^{\dagger_{\vn{G}}})^{i}\proj{p}(F^{\dagger_{\vn{G}}}G^{\ddagger_{\vn{F}}})^{i}F^{\dagger_{\vn{G}}})\proj{p}G^{\ddagger_{\vn{F}}}(F^{\dagger_{\vn{G}}}G^{\ddagger_{\vn{F}}})^{j-i-1}\proj{r}
\end{eqnarray*}
Since $\proj{p}G^{\ddagger_{\vn{F}}}=0$, we have that $(\proj{p}t_{i}\proj{p})(\proj{p}s'_{j}\proj{r})=0$.
\item $(\proj{r}t'_{i}\proj{r})(\proj{r}s_{j}\proj{p})$.\\
Similarly, in the case $j\leqslant i$, we have:
\begin{eqnarray*}
(\proj{r}t'_{i}\proj{r})(\proj{r}s_{j}\proj{p})&=&\proj{r}t'_{i}\proj{r}s_{j}\proj{p}\\
&=&\proj{r}G^{\ddagger_{\vn{F}}}s'_{i}\proj{r}s_{j}\proj{p}\\
&=&\proj{r}G^{\ddagger_{\vn{F}}}s'_{i-j}(s'_{j}\proj{r}s_{j}\proj{p})\\
&=&\proj{r}G^{\ddagger_{\vn{F}}}s'_{i-j}(\proj{p}s'_{j}\proj{r}s_{j}\proj{p})
\end{eqnarray*}
If $i<j$, we obtain:
\begin{eqnarray*}
(\proj{r}t'_{i}\proj{r})(\proj{r}s_{j}\proj{p})&=&\proj{r}t'_{i}\proj{r}s_{j}\proj{p}\\
&=&\proj{r}(G^{\ddagger_{\vn{F}}}F^{\dagger_{\vn{G}}})^{i}G^{\ddagger_{\vn{F}}}\proj{r}(G^{\ddagger_{\vn{F}}}F^{\dagger_{\vn{G}}})^{j}\proj{p}\\
&=&\proj{r}G^{\ddagger_{\vn{F}}}(F^{\dagger_{\vn{G}}}G^{\ddagger_{\vn{F}}})^{i}\proj{r}(G^{\ddagger_{\vn{F}}}F^{\dagger_{\vn{G}}})^{i}G^{\ddagger_{\vn{F}}}F^{\dagger_{\vn{G}}}(G^{\ddagger_{\vn{F}}}F^{\dagger_{\vn{G}}})^{j-i-1}\proj{p}
\end{eqnarray*}
Once again, we can show that $G^{\ddagger_{\vn{F}}}(F^{\dagger_{\vn{G}}}G^{\ddagger_{\vn{F}}})^{i}\proj{r}(G^{\ddagger_{\vn{F}}}F^{\dagger_{\vn{G}}})^{i}G^{\ddagger_{\vn{F}}}$ is an element of $\vn{P}\vntimes\vn{P}_{\vn{F}}\vntimes\vn{P}_{\vn{G}}$ and therefore commutes with $\proj{r}$. Thus:
\begin{eqnarray*}
(\proj{r}t'_{i}\proj{r})(\proj{r}s_{j}\proj{p})&=&\proj{r}(G^{\ddagger_{\vn{F}}}(F^{\dagger_{\vn{G}}}G^{\ddagger_{\vn{F}}})^{i}\proj{r}(G^{\ddagger_{\vn{F}}}F^{\dagger_{\vn{G}}})^{i}G^{\ddagger_{\vn{F}}})F^{\dagger_{\vn{G}}}(G^{\ddagger_{\vn{F}}}F^{\dagger_{\vn{G}}})^{j-i-1}\proj{p}\\
&=&(G^{\ddagger_{\vn{F}}}(F^{\dagger_{\vn{G}}}G^{\ddagger_{\vn{F}}})^{i}\proj{r}(G^{\ddagger_{\vn{F}}}F^{\dagger_{\vn{G}}})^{i}G^{\ddagger_{\vn{F}}})\proj{r}F^{\dagger_{\vn{G}}}(G^{\ddagger_{\vn{F}}}F^{\dagger_{\vn{G}}})^{j-i-1}\proj{p}
\end{eqnarray*}
Since $\proj{r}F^{\dagger_{\vn{G}}}=0$, we have $(\proj{r}t'_{i}\proj{r})(\proj{r}s_{j}\proj{p})=0$.
\item $(\proj{r}s_{i}\proj{p}+\proj{p}s'_{i}\proj{r})(\proj{r}s_{j}\proj{p}+\proj{p}s'_{j}\proj{r})$, with $i\not=j$.\\
We suppose, without loss of generality, that $i<j$. Then:
\begin{eqnarray*}
(\proj{r}s_{i}\proj{p}+\proj{p}s'_{i}\proj{r})(\proj{r}s_{j}\proj{p}+\proj{p}s'_{j}\proj{r})&=&\proj{r}s_{i}\proj{p}s'_{j}\proj{r}+\proj{p}s'_{i}\proj{r}s_{j}\proj{p}\\
&=&(\proj{r}s_{i}\proj{p}s'_{i})s'_{j-i}\proj{r}+(\proj{p}s'_{i}\proj{r}s_{i})s_{j-i}\proj{p}\\
&=&(\proj{r}s_{i}\proj{p}s'_{i}\proj{r})s'_{j-i}\proj{r}+(\proj{p}s'_{i}\proj{r}s_{i}\proj{p})s_{j-i}\proj{p}
\end{eqnarray*}
Since $\proj{r}F=0=\proj{p}G$, we have that $\proj{r}s'_{j-i}=\proj{p}s_{j-i}=0$, and therefore $(\proj{r}s_{i}\proj{p}+\proj{p}s'_{i}\proj{r})(\proj{r}s_{j}\proj{p}+\proj{p}s'_{j}\proj{r})=0$.
\end{itemize}
Finally, we have shown that $F\plug G$ is a sum of partial isometries that satisfy the hypotheses of \autoref{lemmasumpartialiso}. We then deduce that $F\plug G$ is a partial isometry in the normalising groupoid of $\vn{P}\vntimes\vn{P}_{\vn{F}}\vntimes\vn{P}_{\vn{F}}$.
\item \textbf{Traces.} Let\footnote{Notice that $\vn{P}\vntimes\vn{N}_{\finhyp\vntimes\finhyp}(\vn{P}_{\vn{F}}\vntimes\vn{P}_{\vn{G}})=\vn{P}\vntimes\vn{N}_{\finhyp}(\vn{P}_{\vn{F}})\vntimes\vn{N}_{\finhyp}(\vn{P}_{\vn{G}})$ by \autoref{chifan}.} $\pi\in\vn{P}\vntimes\vn{N}_{\finhyp}(\vn{P}_{\vn{F}})\vntimes\vn{N}_{\finhyp}(\vn{P}_{\vn{G}})$ be a projection. We can compute $\tr(\pi A)$ :
\begin{eqnarray*}
\tr(\pi F\plug G)&=&\tr(\pi [(\proj{p}F^{\dagger_{\vn{G}}}+\proj{r})(1-G^{\ddagger_{\vn{F}}}F^{\dagger_{\vn{G}}})^{-1}(\proj{p}+G^{\ddagger_{\vn{F}}}\proj{r})])\\
&=&\tr(\pi[\sum_{k=0}^{N-1}\proj{p}F^{\dagger_{\vn{G}}}(G^{\ddagger_{\vn{F}}}F^{\dagger_{\vn{G}}})^{k}\proj{p}])+\tr(\pi[\sum_{k=0}^{N-1}\proj{r}(G^{\ddagger_{\vn{F}}}F^{\dagger_{\vn{G}}})^{k}G^{\ddagger_{\vn{F}}}\proj{r}])\\&&+\tr(\pi[\sum_{k=1}^{N-1}\proj{r}(G^{\ddagger_{\vn{F}}}F^{\dagger_{\vn{G}}})^{k}\proj{p}])+\tr(\pi[\sum_{k=0}^{N-1}\proj{p}(F^{\dagger_{\vn{G}}}G^{\ddagger_{\vn{F}}})^{k}\proj{r}])\\
&=&\sum_{k=0}^{N-1}\tr(\pi[\proj{p}F^{\dagger_{\vn{G}}}(G^{\ddagger_{\vn{F}}}F^{\dagger_{\vn{G}}})^{k}\proj{p}])+\sum_{k=0}^{N-1}\tr(\pi[\proj{r}(G^{\ddagger_{\vn{F}}}F^{\dagger_{\vn{G}}})^{k}G^{\ddagger_{\vn{F}}}\proj{r}])\\&&+\sum_{k=1}^{N-1}\tr(\pi[\proj{r}(G^{\ddagger_{\vn{F}}}F^{\dagger_{\vn{G}}})^{k}\proj{p}])+\sum_{k=1}^{N-1}\tr(\pi[\proj{p}(F^{\dagger_{\vn{G}}}G^{\ddagger_{\vn{F}}})^{k}\proj{r}])
\end{eqnarray*}
It is enough to show that $\tr(\pi\proj{p} t_{k}\proj{p})=\tr(\pi \proj{r}t'_{k}\proj{r})=0$ and $\tr(\pi \proj{r}s_{k}\proj{p})=\tr(\pi \proj{p}s'_{k}\proj{r})=0$ for all $1\leqslant k\leqslant N-1$. 
\begin{itemize}
\item $\tr(\pi\proj{p}t_{k}\proj{p})=0$ and $\tr(\pi\proj{r}t'_{k}\proj{r})=0$.~\\
We suppose, without loss of generality, that $\pi\proj{p}=\pi$, i.e. that $\pi$ is a subjection of $\proj{p}$. Then:
\begin{equation*}
\tr(\pi\proj{p} t_{k}\proj{p})=\left\{\begin{array}{ll}
\tr(F^{\dagger_{\vn{G}}}(G^{\ddagger_{\vn{F}}}F^{\dagger_{\vn{G}}})^{\frac{k}{2}}\pi(F^{\dagger_{\vn{G}}}G^{\ddagger_{\vn{F}}})^{\frac{k}{2}})&\text{ if $k\equiv 0 [2]$}\\
\tr(G^{\ddagger_{\vn{F}}}(F^{\dagger_{\vn{G}}}G^{\ddagger_{\vn{F}}})^{\frac{k-1}{2}}(F^{\dagger_{\vn{G}}}\pi F^{\dagger_{\vn{G}}})(G^{\ddagger_{\vn{F}}}F^{\dagger_{\vn{G}}})^{\frac{k-1}{2}})&\text{ if $k\equiv 1 [2]$}\end{array}\right.
\end{equation*}
Since $F^{\dagger_{\vn{G}}}$ and $G^{\ddagger_{\vn{F}}}$ are in the normalising groupoid of $\vn{P}\vntimes\vn{P}_{\vn{F}}\vntimes\vn{P}_{\vn{G}}$, we can show they are elements of $\vn{N}(\vn{P}_{\proj{p+q+r}}\vntimes\vn{N}_{\finhyp}(\vn{P}_{\vn{F}})\vntimes\vn{N}_{\finhyp}(\vn{P}_{\vn{G}}))$. Indeed:
\begin{eqnarray*}
\vn{N}(\vn{P}_{\proj{p+q+r+}}\vntimes\vn{P}_{\vn{F}}\vntimes\vn{P}_{\vn{G}})&=&\vn{N}(\vn{P}_{\proj{p+q+r}})\vntimes\vn{N}(\vn{P}_{\vn{F}})\vntimes\vn{N}(\vn{P}_{\vn{G}})\\
&\subset&\vn{N}(\vn{P}_{\proj{p+q+r}})\vntimes\vn{N}(\vn{N}(\vn{P}_{\vn{F}}))\vntimes\vn{N}(\vn{N}(\vn{P}_{\vn{G}}))\\
&\subset&\vn{N}(\vn{P}_{\proj{p+q+r}}\vntimes\vn{N}(\vn{P}_{\vn{F}})\vntimes\vn{N}(\vn{P}_{\vn{G}}))
\end{eqnarray*}
Thus, since $\pi$ is a projection in $\vn{P}_{\proj{p+q+r}}\vntimes\vn{N}_{\finhyp}(\vn{P}_{\vn{F}})\vntimes\vn{N}_{\finhyp}(\vn{P}_{\vn{G}})$, we deduce that the terms $(G^{\ddagger_{\vn{F}}}F^{\dagger_{\vn{G}}})^{\frac{k}{2}}\pi(F^{\dagger_{\vn{G}}}G^{\ddagger_{\vn{F}}})^{\frac{k}{2}}$ and $(F^{\dagger_{\vn{G}}}G^{\ddagger_{\vn{F}}})^{\frac{k-1}{2}}(F^{\dagger_{\vn{G}}}\pi F^{\dagger_{\vn{G}}})(G^{\ddagger_{\vn{F}}}F^{\dagger_{\vn{G}}})^{\frac{k-1}{2}}$ represent projections in $\vn{P}\vntimes\vn{N}_{\finhyp}(\vn{P}_{\vn{F}})\vntimes\vn{N}_{\finhyp}(\vn{P}_{\vn{G}})$.~\\
Since $\tr(F \nu)=0$ for all projection $\nu\in\vn{P}\vntimes\vn{N}_{\finhyp}(\vn{P}_{\vn{F}})$, we can show that $\tr(F^{\dagger_{\vn{G}}} \mu)=0$ for all projection $\mu\in\vn{P}\vntimes\vn{N}_{\finhyp}(\vn{P}_{\vn{F}})\vntimes\finhyp$ and therefore for all projection $\mu$ in the algebra $\vn{P}\vntimes\vn{N}_{\finhyp}(\vn{P}_{\vn{F}})\vntimes\vn{N}_{\finhyp}(\vn{P}_{\vn{G}})$. Similarly, $\tr(G^{\ddagger_{\vn{F}}} \mu)=0$ for such a projection $\mu$, and we can conclude that $\tr(\pi\proj{p}t_{k}\proj{p})=0$.

The case $\tr(\pi\proj{r}t'_{k}\proj{r})=0$ is treated in a similar fashion.
\item $\tr(\pi\proj{r}s_{k}\proj{p})=0$ and $\tr(\pi\proj{p}s'_{k}\proj{r})=0$.~\\
Since $\proj{p}=\tilde{\proj{p}}\otimes1\otimes 1$ and since $\vn{P}$ is commutative, $\proj{p}$ is an element of the center $\vn{Z}(\vn{P}\vntimes\finhyp\vntimes\finhyp)$ of $\vn{P}\vntimes\finhyp\vntimes\finhyp$ and therefore commutes with every elements in the algebra $\vn{P}\vntimes\vn{N}_{\finhyp}(\vn{P}_{\vn{F}})\vntimes\vn{N}_{\finhyp}(\vn{P}_{\vn{G}})$. Similarly, the projection $\proj{r}$ is an element of $\vn{Z}(\vn{P}\vntimes\finhyp\vntimes\finhyp)$ and therefore commutes with every element in $\vn{P}\vntimes\vn{N}_{\finhyp}(\vn{P}_{\vn{F}})\vntimes\vn{N}_{\finhyp}(\vn{P}_{\vn{G}})$. Then, since $\proj{pr}=\proj{rp}=0$:
\begin{eqnarray*}
\tr(\pi\proj{r}s_{k}\proj{p})=\tr(\proj{r}\pi s_{k}\proj{p})=\tr(\proj{pr}\pi s_{k})=0\\
\tr(\pi\proj{p}s'_{k}\proj{r})=\tr(\proj{p}\pi s'_{k}\proj{r}=\tr(\proj{rp}\pi s'_{k})=0
\end{eqnarray*}
\end{itemize}
\end{itemize}
Finally we have shown that the project $\de{f\plug g}$ is wager-free, normalised, that the operator $F\plug G$ is a partial symmetry in the normalising groupoid of $\vn{P}\vntimes\vn{P}_{\vn{F}}\vntimes\vn{P}_{\vn{G}}$,  and that for all projection $\pi\in\vn{P}\vntimes\vn{N}_{\finhyp}(\vn{P})$, we have $\tr(\pi (F\plug G))=0$. This shows that it is promising w.r.t. the outlook $\vn{P}$.
\end{proof}

\subsection{Hyperfinite GoI}\label{verite52}

\subsubsection{Promising and Successful Projects}

We first recall the notion of truth considered by Girard \cite{goi5} which is based on the notion of \emph{successful hyperfinite project}. We will then exhibit a correspondence between our notion of \emph{promising project} and the latter. This will allow us to deduce some results about Girard's GoI model.

\begin{definition}[Viewpoint]
A \emph{viewpoint} is a representation $\pi$ of the algebra $\infhyp$ onto $L^{2}(\realN,\lambda)$ where $\lambda$ is the Lebesgue measure, which satisfies the following conditions:
\begin{itemize}
\item $L^{\infty}(\realN,\lambda)\subset\pi(\infhyp)$;
\item $\forall A\subset\realN$, $\tr(\pi^{-1}(\chi_{A}))=\lambda(A)$, where $\chi_{A}$ is the characteristic function of $A$.
\end{itemize}
A viewpoint if \emph{faithful} when the representation $\pi$ is faithful (\autoref{representations}).
\end{definition}

If $T:\realN\rightarrow \realN$ is a measure-preserving transformation, one can define the isometry $[T]\in\B{L^{2}(\realN,\lambda)}$:
\begin{equation*}
[T]: f\in L^{2}(\realN,\lambda)\mapsto f\circ T\in L^{2}(\realN,\lambda)
\end{equation*}
That $[T]$ is an isometry comes from the fact that $T$ is measure-preserving:
\begin{eqnarray*}
\inner{[T]f}{[T]g}&=&\int_{\realN} ([T]f)(x)\overline{([T]g)(x)}d\lambda(x)\\
&=&\int_{\realN} f\circ T(x) \overline{g\circ T(x)}d\lambda(x)\\
&=&\int_{\realN} f\circ T(x) \overline{g\circ T(x)}d\lambda(x)\\
&=&\int_{T(\realN)} f(x)\overline{g(x)}d\lambda(x)\\
&=&\int_{\realN} f(x)\overline{g(x)}d\lambda(x)\\
&=&\inner{f}{g}
\end{eqnarray*}
Suppose now given $U: X\rightarrow Y$ a measure-preserving transformation, with $X,Y\subset \realN$ measurable subsets. We define, for all map $f\in L^{2}(\realN,\lambda)$, $[U]f(x)=f\circ U(x)$ if $x\in X$ and $[U]f(x)=0$ otherwise. The operator $[U]$ thus defined is a partial isometry. Indeed, if we write $p$ the projection in $\B{L^{2}(\realN,\lambda)}$ induced by the characteristic map of $Y$, then for all $f,g\in pL^{2}(\realN,\lambda)$,
\begin{eqnarray*}
\inner{[U]f}{[U]g}&=&\int_{\realN} ([U]f)(x)\overline{([U]g)(x)}d\lambda(x)\\
&=&\int_{X} ([U]f)(x)\overline{([U]g)(x)}d\lambda(x)\\
&=&\int_{X} f\circ U(x) \overline{g\circ U(x)} d\lambda(x)\\
&=&\int_{Y} f(x)\overline{g(x)}d\lambda(x)\\
&=&\int_{\realN} f(x)\overline{g(x)}d\lambda(x)
\end{eqnarray*}
Moreover, it is clear that for all $f,g\in (1-p)L^{2}(\realN,\lambda)$ one has $\inner{[U]f}{[U]g}=0$.

\begin{definition}
A hyperfinite project $\de{a}=0\cdot+\cdot \alpha+A$ of carrier $p$ is \emph{successful} w.r.t. a viewpoint $\pi$ when:
\begin{itemize}
\item $\pi(p)\in L^{\infty}(\realN)$;
\item $\alpha$ is the normalised trace on $\vn{A}$;
\item there exists a basis $e_{1},\dots,e_{n}$ of the dialect $\vn{A}$ such that $A=[f]$ where $f$ is a partial measure-preserving bijection of $\realN\times\{1,\dots,n\}$;
\item the set $\{x\in \realN\times\{1,\dots,n\}~|~ f(x)=x\}$ is of null measure.
\end{itemize}
\end{definition}

\begin{remark}\label{remarkadditionalcondition}
We added the last condition to the definition proposed by Girard \cite{goi5}. This condition corresponds to the trace condition in our definition of promising projects (\autoref{projetsprometteurs}).
\end{remark}

\begin{proposition}
Every faithful viewpoint defines an outlook.
\end{proposition}

\begin{proof}
Let $\pi$ be a faithful viewpoint, then $L^{\infty}(\realN,\lambda)\subset\pi(\infhyp)$. Since $L^{\infty}(\realN,\lambda)$ is a MASA in $\B{L^{2}(\realN,\lambda)}$, $L^{\infty}(\realN,\lambda)$ is equal to its commutant in $\B{L^{2}(\realN,\lambda)}$. We deduce that the commutant of $L^{\infty}(\realN,\lambda)$ in $\pi(\infhyp)$ is included in $L^{\infty}(\realN,\lambda)$. But, since the algebra is commutative, the converse inclusion is also satisfied. Thus $L^{\infty}(\realN,\lambda)$ is a MASA in $\pi(\infhyp)$. Let us now consider $\vn{B}=\pi^{-1}(L^{\infty}(\realN,\lambda))$. This sub-algebra of $\infhyp$ is clearly commutative. Moreover, since $x\in\infhyp$ commutes with the elements in $\vn{B}$, then $\pi(x)$ commutes with the elements in $L^{\infty}(\realN,\lambda)$. From the maximality of the latter, we deduce that $\pi(x)\in L^{\infty}(\realN,\lambda)$, and therefore $x\in \vn{B}$. Thus $\vn{B}$ is a MASA in $\infhyp$.
\end{proof}

\begin{proposition}
Let $\pi$ be a faithful viewpoint, and $\vn{B}=\pi^{-1}(L^{\infty}(\realN,\lambda))$ the outlook defined by $\pi$. If $\de{a}=0\cdot +\cdot tr + A$ is successful w.r.t. $\pi$, then it is promising w.r.t. $\vn{B}$.
\end{proposition}

\begin{proof}
One only needs to check that $A$ belongs to the normalising groupoid of $\vn{B}\vntimes\vn{Q}$ for a MASA $\vn{Q}$ in $\vn{A}$. Since $\vn{A}$ is a finite factor, we can suppose without loss of generality (modulo considering a variant of $\de{a}$) that it is equal to $\vn{M}_{n}(\complexN)$ for an integer $n\in\naturalN$. The basis $e_{1},\dots,e_{n}$ of the dialect such that $\pi(A)=[f]$ where $f$is a partial measure-preserving bijection on $\realN\times\{1,\dots,n\}$, defines a MASA $\vn{D}$ in $\vn{M}_{n}(\complexN)$, the algebra of diagonal operators in the basis $\{e_{1},\dots,e_{n}\}$. Let $p$ be a projection in $\vn{B}\vntimes\vn{D}$. We know that $\pi(p)=[\chi_{U}]$ where $U$ is a measurable subset of $\realN\times\{1,\dots,n\}$. For all $\xi\in L^{2}(\realN,\lambda)$, we obtain:
\begin{equation*}
\pi(A)^{\ast}\pi(p)\pi(A)\xi=\pi(A)\pi(p)\xi\circ f=\pi(A)\xi\circ f\circ \chi_{U}=\xi\circ f\circ\chi_{U}\circ f^{-1}
\end{equation*}
But $f\circ\chi_{U}\circ f^{-1}=\chi_{V}$ where $V$ is the measurable subset of $\realN\times\{1,\dots,n\}$ defined by $V=f(U)$. Thus $\pi(A)^{\ast}\pi(p)\pi(A)\xi=[\chi_{V}]\xi=\pi(q)\xi$, where $q$ is the projection $[\chi_{V}]$. Since $q$ is in $\vn{B\otimes D}$, we showed that $\pi(A)$ normalises every projection in $\vn{B\otimes D}$. We can conclude that $\pi(A)\in\gn{\vn{B\otimes D}}$ by using the fact that the projections in $\vn{B\otimes D}$ generate the whole algebra. As a consequence, $A$ is a partial symmetry in the normalising groupoid of $\vn{B\otimes D}$.

\begin{lemma}\label{mesureandtrace}
Let $\phi: X\rightarrow Y$ be a measure-preserving transformation such that $[\phi]\in\infhyp$. Then $\tr(\omega[\phi])=\lambda(\{x\in X~|~\phi(x)=x\})\times\omega$.
\end{lemma}

\begin{proof}
From the linearity of the trace, we have $\tr(\omega[\phi])=\omega\times \tr([\phi])$. We are thus left to show that $\tr([\phi])=\lambda(\{x\in X~|~\phi(x)=x\})$ for all partial measure-preserving transformation. This result is shown by Girard in the annex of his paper on hyperfinite geometry of interaction \cite{goi5}.
\end{proof}

We now are left with showing that the trace condition holds. For this, we pick a projection $p$ in $\vn{B\otimes M}_{n}(\complexN)$ and use \autoref{mesureandtrace} to show:
\begin{equation*}
\tr(pA)=\lambda(\{x\in U_{p}\times\{1,\dots,n\}~|~f(x)=x\})\leqslant\lambda(\{x\in \realN\times\{1,\dots,n\}~|~f(x)=x\})
\end{equation*}
Since $\de{a}$ is successful, we then obtain $0\leqslant \tr(pA)\leqslant 0$, from which we can conclude.
\end{proof}

\subsubsection{Adapting a Few Results}

We here state two results that give a better understanding of the relationship between Girard's notion of success and our own. In particular, a viewpoint gives rise to  outlooks satisfying a specific property with respect to the Pukansky invariant of its \emph{finite restrictions}.

\begin{theorem}
Let $\pi$ be a faithful viewpoint. The associated outlook $\vn{A}$ is such that for all finite projections\footnote{Notice that although $\vn{A}$ is a MASA in a type {II}$_{\infty}$ factor $\infhyp$, $p\vn{A}p$ is a MASA in $p\infhyp p$, which is a type {II}$_{1}$ factor, and its Pukansky invariant is therefore well defined.} $p\in\vn{A}$, $\puk{p\vn{A}p}=\{1\}$.
\end{theorem}

\begin{proof}
For all finite projections $p\in\vn{A}$, the algebra $p\vn{A}p$ is a MASA in $p\infhyp p$ by \autoref{singularprojection}. Suppose there exists a projection $p$ such that $\puk{p\vn{A}p}$ contains at least one integer greater than $2$, that we will denote by $k$. Fo all regular MASA $\vn{Q}$ in $\finhyp$, $\puk{p\vn{A}p \otimes \vn{Q}}$ contains the integer $k$ by \autoref{tenseurpuk}. We then deduce from \autoref{puksemireg} that $p\vn{A} p\otimes \vn{Q}$ is not a MASA in $\B{L^{2}(p\infhyp p\otimes \finhyp)}$. But $\pi(\vn{A})=L^{\infty}(\realN)$ is a MASA in $\B{L^{2}(\realN)}$, and therefore $\pi(p\vn{A} p)=L^{\infty}(X)$ for $X$ a measurable subset of finite measure. We deduce that $\pi(p\vn{A} p)$ is a MASA in $\B{L^{2}(X)}$, and thus unitarily equivalent (we write $u$ the said unitary) to a MASA in $\B{L^{2}(p\infhyp p)\otimes\finhyp}$ (\autoref{masasbh}). The morphism $\phi: a\mapsto u\pi(a)u^{\ast}$ is then a morphism from $p\infhyp p\otimes\vn{Q}$ into $\phi(p\infhyp p \otimes \vn{Q})$ such that $\puk{\phi(\vn{A})}=\{1\}$ which contradicts the fact that $k$ belongs to $\puk{\vn{A}}$.
\end{proof}

The previous theorem leads to the natural question of the existence of non-regular viewpoints. The following proposition partially answers this question.

\begin{proposition}
There exist regular viewpoints, semi-regular viewpoints, as well as non-Dixmier-classifiable viewpoints.
\end{proposition}

\begin{proof}
We consider $\finhyp$ in its standard representation, and we fix $\vn{A},\vn{B},\vn{C}$ three MASAs in $\finhyp$ such that $\vn{A}$ is regular, $\vn{B}$ is semi-regular, $\vn{C}$ is singular, and $\puk{\vn{A}}=\puk{\vn{B}}=\puk{\vn{C}}=\{1\}$. We recall that we showed this is possible at the end of Section \ref{refmasaspuk1-1}.

Since $\vn{A}$, $\vn{B}$ and $\vn{C}$ have their Pukansky invariant equal to the singleton $\{1\}$, they are MASAs in $\B{L^{2}(\finhyp)}$ by \autoref{puk1}. Moreover, they are diffuse since $\finhyp$ is of type $\text{II}$ and there consequently exist unitaries $u,v,w$ such that $u^{\ast}\vn{A}u=L^{\infty}([0,1])$, $v^{\ast}\vn{B}v=L^{\infty}([0,1])$ and $w^{\ast}\vn{C}w=L^{\infty}([0,1])$. We now pick a diagonal MASA $\vn{D}$ in $\B{\hil{H}}$ induced by a basis, and define the following MASAs in $\finhyp\otimes\B{\hil{H}}$:
$$\vn{E_{A}}=\vn{A}\otimes\vn{D}~~~~~\vn{E_{B}}=\vn{B}\otimes\vn{D}~~~~~\vn{E_{C}}=\vn{C}\otimes\vn{D}$$
The unitaries $u\otimes 1, v\otimes 1, w\otimes 1:L^{2}(\finhyp)\otimes\B{\hil{H}}\rightarrow L^{2}([0,1])\otimes \B{\hil{H}}$ define -- modulo the isomorphism between $L^{2}([0,1])\otimes\B{\hil{H}}$ and $L^{2}(\realN)$ -- representations $\pi_{\vn{A}}: x\mapsto  (u\otimes 1)^{\ast}x(u\otimes 1)$, $\pi_{\vn{B}}:x\mapsto (v\otimes 1)^{\ast}x(v\otimes 1)$ and $\pi_{\vn{C}}:x\mapsto (w\otimes 1)^{\ast}x(w\otimes 1)$ of $\infhyp$ onto $L^{2}(\realN)$ such that:
$$\pi_{\vn{A}}(\vn{E}_{\vn{A}})=L^{\infty}(\realN)~~~~~\pi_{\vn{B}}(\vn{E}_{\vn{B}})=L^{\infty}(\realN)~~~~~\pi_{\vn{C}}(\vn{E}_{\vn{C}})=L^{\infty}(\realN)$$

We now show that the outlooks associated to $\pi_{\vn{A}}$, $\pi_{\vn{B}}$ and $\pi_{\vn{C}}$, i.e. the MASAs $\vn{E}_{\vn{A}}$, $\vn{E}_{\vn{B}}$ and $\vn{E}_{\vn{C}}$, are respectively regular, semi-regular and non-Dixmier-classifiable. It is in fact a simple application of \autoref{chifan}. In the case of $\vn{E}_{\vn{A}}=\vn{A}\otimes\vn{D}$, we have:
\begin{equation*}
\noralg[\finhyp\otimes\B{\hil{H}}]{\vn{A}\otimes\vn{D}}=\noralg[\finhyp]{\vn{A}}\otimes\noralg[\B{\hil{H}}]{\vn{D}}=\finhyp\otimes\B{\hil{H}}
\end{equation*}
Thus $\vn{E}_{\vn{A}}$ is regular. 

In the case of $\vn{E}_{\vn{B}}=\vn{B}\otimes\vn{D}$, we obtain:
\begin{equation*}
\noralg[\finhyp\otimes\B{\hil{H}}]{\vn{B}\otimes\vn{D}}=\noralg[\finhyp]{\vn{B}}\otimes\noralg[\B{\hil{H}}]{\vn{D}}=\vn{N}\otimes\B{\hil{H}}
\end{equation*}
Since $\vn{B}$ is semi-regular, $\vn{N}$ is a proper sub factor of $\finhyp$, and therefore $\vn{N}\otimes\B{\hil{H}}$ is a proper sub factor of $\finhyp\otimes\B{\hil{H}}$ (we use here a theorem\footnote{This theorem can be found with its proof in Takesaki's series \cite{takesaki1,takesaki2,takesaki3}.} due to Tomita stating that the commutant of a tensor product is the tensor product of the commutants). Thus $\vn{E}_{\vn{B}}$ is semi-regular.

In the case of $\vn{E}_{\vn{C}}=\vn{C}\otimes\vn{D}$, we have:
\begin{equation*}
\noralg[\finhyp\otimes\B{\hil{H}}]{\vn{C}\otimes\vn{D}}=\noralg[\finhyp]{\vn{C}}\otimes\noralg[\B{\hil{H}}]{\vn{D}}=\vn{C}\otimes\B{\hil{H}}
\end{equation*}
Since the commutant of $\vn{C}\otimes\B{\hil{H}}$ is equal to $\vn{C}\otimes\complexN$, the center of $\vn{C}\otimes\B{\hil{H}}$ is equal to $\vn{C}\otimes\complexN$. We deduce that $\noralg[\finhyp\otimes\B{\hil{H}}]{\vn{E}_{\vn{C}}}$ is neither equal to $\vn{E}_{\vn{C}}$, neither a factor. Therefore $\vn{E}_{\vn{C}}$ is non-Dixmier-classifiable.
\end{proof}

The question of the existence of singular viewpoints is still open. Indeed, the method used in the preceding proof does not apply for showing the existence of singular MASAs: by writing $\infhyp=\finhyp\otimes\B{\hil{H}}$ and choosing a MASA of the form $\vn{A}\otimes \vn{D}$, we impose ourselves a certain \enquote{regularity}, since $\vn{D}$ is necessarily a regular MASA ($\B{\hil{H}}$ is of type {I}). The existence of singular MASAs with Pukansky invariant equal to $\{1\}$ in the hyperfinite factor $\finhyp$ of type $\text{II}_{1}$ obtained by White \cite{tauermasas} suggests however that such viewpoints exist.

\section{Dixmier's Classification and Linear Logic}

In this section, we will prove the main technical result of this paper. We first prove that no non-trivial interpretation of linear logic proofs exists w.r.t. a singular outlook. We then show that semi-regular outlooks provide enough structure to interpret the exponential-free fragment of linear logic. Lastly, we show that regular outlooks provide the structure for the interpretation of exponential connectives.

\subsection{Singular MASAs}

First, we show that every promising project w.r.t. an outlook $\vn{P}$ which is a singular MASA in $\infhyp$ is trivial, i.e. its operator is equal to $0$. We will first show two lemmas that will be of use afterwards. We will write $\vn{A}_{\proj{p}}$ the von Neumann algebra $\proj{p}\vn{A}\proj{p}$ where $\proj{p}$ is a projection, i.e. the restriction of $\vn{A}$ to the subspace corresponding to $p$.

\begin{lemma}\label{reducingnormaliser}
Let $\vn{A}$ be a MASA in a factor $\vn{M}$, and let $\proj{p}$ be a projection in $\vn{A}$. If $A\in\vn{M}$ normalises $\vn{A}$, and $A\proj{p}=\proj{p}A$, then $A$ normalises $\vn{A}_{\proj{p}}$.
\end{lemma}

\begin{proof}
We pick $x$ in $\vn{A}_{\proj{p}}\subset\vn{A}$. Then $AxA^{\ast}=y\in\vn{A}$ since $A$ normalises $\vn{A}$. Moreover, $y\proj{p}=AxA^{\ast}\proj{p}=Ax(\proj{p}A)^{\ast}=Ax(A\proj{p})^{\ast}=Ax\proj{p}A^{\ast}=AxA^{\ast}=y$, and a similar argument shows that $\proj{p}y=y$. Thus $y=\proj{p}y\proj{p}$ and $y\in\vn{A}_{\proj{p}}$.
\end{proof}

\begin{remark}
This result implies that $\proj{p}\textnormal{G}_{\vn{M}}(\vn{A})\proj{p}\subset \textnormal{G}_{\vn{M}_{\proj{p}}}(\vn{A}_{\proj{p}})$.
\end{remark}

The following lemma is of particular importance since it will allow us to reduce our study to the case of a factor of type $\text{II}_{1}$, and thus to use Chifan's result (\autoref{chifan}). The fact that $\vn{A}$ is abelian is essential here. Indeed, one can even find finite-dimensional counter-examples in the case $\vn{A}$ is a non-commutative singular von Neumann sub-algebra. For instance, the subfactor $\vn{A}=\vn{M}_{2}(\complexN)\oplus\complexN$ of $\vn{M}_{3}(\complexN)$ is singular : $\vn{N}_{\vn{M}_{3}(\complexN)}(\vn{A})=\vn{A}$. Picking the projection $\proj{p}=0\oplus1\oplus1$ in $\vn{M}_{3}(\complexN)$, we have that $\vn{A}_{\proj{p}}$ is not singular in $(\vn{M}_{3}(\complexN))_{\proj{p}}$ -- it is even a regular sub-algebra -- since $\vn{N}_{\vn{M}_{2}(\complexN)}(\vn{A}_{\proj{p}})=\vn{M}_{2}(\complexN)$.

\begin{lemma}\label{singularprojection}
Let $\vn{A}$ be a MASA in a von Neumann algebra $\vn{M}$, and $\proj{p}$ a projection in $\vn{A}$. Then $\vn{A}_{\proj{p}}$ is a maximal abelian sub-algebra of $\vn{M}_{\proj{p}}$. Moreover, if $\vn{A}$ is singular, $\vn{A}_{p}$ is singular.
\end{lemma}

\begin{proof}
Let $v$ be an element of $(\vn{A}_{\proj{p}})'$, the commutant of $\vn{A}_{\proj{p}}$ in $\vn{M}_{\proj{p}}$, and let $a\in\vn{A}$. We have $\proj{p}a=\proj{p}a\proj{p}$ and $(1-\proj{p})a=(1-\proj{p})a(1-\proj{p})$ since $\proj{p}$ and $(1-\proj{p})$ are elements of $\vn{A}$ -- which is commutative. Using the fact that $v\proj{p}=\proj{p}v=v$, we obtain:
\begin{eqnarray*}
va&=&v(\proj{p}a+(1-\proj{p})a)\\
&=&v\proj{p}a+v(1-\proj{p})a\\
&=&v\proj{p}a\proj{p}\\
&=&\proj{p}a\proj{p}v\\
&=&av
\end{eqnarray*}
As a consequence, $(\vn{A}_{p})'\subset(\vn{A}')_{\proj{p}}=\vn{A}_{p}$ from the maximality of $\vn{A}$. This gives us that $(\vn{A}_{p})'\subset\vn{A}_{p}$, i.e. $\vn{A}_{\proj{p}}$ is a MASA in $\vn{M}_{\proj{p}}$. 

Suppose now that $\vn{A}$ is a singular MASA. Pick $u\in N_{\vn{M}_{\proj{p}}}(\vn{A}_{\proj{p}})$. By definition, $u$ is a unitary in $\vn{M}_{\proj{p}}$, meaning that $uu^{\ast}=u^{\ast}u=\proj{p}$, and therefore $u\proj{p}=\proj{p}u=u$. Let us define $v=u+(1-\proj{p})$, which is an element of $\vn{M}$. Then, since $\proj{p}(1-\proj{p})=0$ and $v$ is a unitary: $vv^{\ast}=(u+(1-\proj{p}))(u^{\ast}+(1-\proj{p}))=uu^{\ast}+(1-\proj{p})=1=v^{\ast}v$, and $\proj{p}v\proj{p}=u$.\\
Let us now chose $x\in\vn{A}$. Since $x$ and $1-\proj{p}$ are elements of $\vn{A}$ and $\vn{A}$ is commutative, $ux(1-\proj{p})=u(1-\proj{p})x=u\proj{p}(1-\proj{p})x=0$. Similarly, we show that $(1-\proj{p})xu=0$. This implies that $vxv^{\ast}=uxu^{\ast}+(1-\proj{p})x(1-\proj{p})$. But, since $u$ normalises $\vn{A}$, $uxu^{\ast}=u\proj{p}x\proj{p}u^{\ast}$ is an element $y$ in $\vn{A}$. Moreover, since $1-\proj{p}$ is in $\vn{A}$, the element $(1-\proj{p})x(1-\proj{p})$ also lives in $\vn{A}$. Finally, $vxv^{\ast}$ is in $\vn{A}$ and $v\in N_{\vn{M}}(\vn{A})$. Thus $u\in\proj{p}N_{\vn{M}}(\vn{A})\proj{p}$. But $N_{\vn{M}}(\vn{A})\subset\vn{N}_{\vn{M}}(\vn{A})=\vn{A}$ ($\vn{A}$ is singular), and therefore $u\in\vn{A}_{\proj{p}}$.\\
Since $N_{\vn{M}_{\proj{p}}}(\vn{A}_{\proj{p}})\subset\vn{A}_{\proj{p}}$ this gives us that $\vn{N}_{\vn{M}_{\proj{p}}}(\vn{A}_{\proj{p}})=\vn{A}_{\proj{p}}$.
\end{proof}

\begin{theorem}[Singular Outlooks and Soundness]\label{singtruth}
If $\vn{P}$ is a singular MASA in $\infhyp$, then every promising project w.r.t. the outlook $\vn{P}$ is trivial.
\end{theorem}

\begin{proof}
Let $\de{a}=(\proj{p},0,\finhyp,tr,A)$ be a promising hyperfinite project w.r.t. $\vn{P}$ a singular MASA in $\infhyp$.\\
We want to show that $A\in\vn{P}\vntimes\finhyp$. Since $A$ is an element of $(\infhyp)_{\proj{p}}\vntimes\finhyp$, we know that $A\proj{\overline{p}}=\proj{\overline{p}}A=A$ where $\proj{\overline{p}}$ is the projection $\proj{p}\vntimes 1$. Since $A$ is in the normalising groupoid $G_{\infhyp\vntimes\finhyp}(\vn{P}\vntimes\vn{Q})$, it is also in $G_{(\infhyp)_{\proj{p}}\vntimes\finhyp}(\vn{P}_{\proj{p}}\vntimes\vn{Q})$ by \autoref{reducingnormaliser}, and it is therefore an element of the algebra $\vn{N}_{(\infhyp)_{\proj{p}}\vntimes\finhyp}(\vn{P}_{\proj{p}}\vntimes\vn{Q})$ by \autoref{dyestheorem}. Now, by \autoref{chifan}, we obtain that $A\in\vn{N}_{\infhyp}(\vn{P}_{\proj{p}})\vntimes\vn{N}_{\finhyp}(\vn{Q})$. Since $\vn{P}$ is singular in $\infhyp$, $\vn{P}_{\proj{p}}$ is singular in $(\infhyp)_{\proj{p}}$ by \autoref{singularprojection}, which shows that $A\in\vn{P}_{\proj{p}}\vntimes\vn{N}_{\finhyp}(\vn{Q})$.\\
Since $A$ is promising w.r.t. $\vn{P}$, we have $\tr(\proj{\pi}A)=0$. Since the set of projections in $\vn{P}\vntimes\vn{N}_{\finhyp}(\vn{Q})$ generates the algebra $\vn{P}\vntimes\vn{N}_{\finhyp}(\vn{Q})$ and since the trace is continuous, we deduce that for all $B\in\vn{P}\vntimes\vn{N}_{\finhyp}(\vn{Q})$, we have $\tr(BA)=0$.\\
Finally, since $A^{3}=A=A^{\ast}$, $A^{2}=\proj{r}$ is a projection, with $\proj{r}A=A\proj{r}=A$. Since $A\in\vn{P}\vntimes\vn{N}_{\finhyp}(\vn{Q})$, we can conclude that $\tr(A^{2})=0$, i.e. $A^{2}=0$ and therefore $A=A^{3}=0$.
\end{proof}

\begin{remark}\label{nontrivialsing}
Without the additional condition (about projections) in the definition of promising projects, it would be easy to find non-trivial hyperfinite projects which are promising w.r.t. a singular MASA $\vn{P}$ in $\infhyp$. Indeed, let $\proj{p,q}$ be two projections in $\vn{P}$. Then the project $\de{a}=(\proj{p+q},0,1,\complexN,\proj{p}+\proj{q})$ would then clearly be promising $\vn{P}$.~\\
One might wonder however if this condition could be weakened, asking for instance that the trace of $A$ be zero. This condition would not be sufficient, since for all projections $\proj{p,q}\in\vn{P}$ such that $\tr(\proj{p})=\tr(\proj{q})$, the hyperfinite project $\de{b}$ defined as $(\proj{p+q},0,1,\complexN,\proj{p-q})$ would then be promising as $\tr(\proj{p-q})=\tr(\proj{p})-\tr(\proj{q})=0$.~\\
Another weaker condition would be: for all projections $\pi\in\vn{P}\vntimes\vn{Q}$, $\tr(\pi A)=0$. However, the following project would then be promising w.r.t. $\vn{P}$ when $\proj{p},\proj{q}\in\vn{P}$ are projections:
\begin{equation*}
\de{c}=(\proj{p+q},0,tr,\vn{M}_{2}(\complexN),\left(\begin{array}{cc}0&(\proj{p+q})_{\otimes 1_{\finhyp}}\\(\proj{p+q})_{\otimes 1_{\finhyp}}&0\end{array}\right) )
\end{equation*}
All those projects may be considered as successful, so why do we want to exclude them? The reason can be found in the relationship between the GoI interpretation of proofs and the theory of proof nets. Indeed, as it is explained in both Girard and the author's work on the interpretation of multiplicatives \cite{multiplicatives,seiller-goim}, the GoI interpretation of a proof corresponds to a representation of the axiom links of the corresponding proof net. As a consequence, a successful project should be understood intuitively as a set of axiom links, i.e. a partial symmetry not containing any fixed point -- something that corresponds to the fact that for all non-zero vector $\xi$ the symmetry $S$ satisfies $S\xi\neq\xi$. In this respect, the first projects considered above should therefore not be considered as successful as they obviously do not satisfy this property. The reason why last project should also not be considered as successful is, however, more involved since it is a symmetry not containing fixed points. In this case, however, the vectors $\xi$ and $S\xi$ differ only from the \emph{dialect}, i.e. the second projection of the vector. Thinking about proof nets again, this second projection, the dialect, corresponds to \emph{slices} in additive proof nets \cite{proofnets}. This last project represents, in this respect, an axiom link between a formula $A$ in a slice $s_{1}$ and the same formula $A$ in a different slice $s_{2}$. The reader familiar with additive proof nets should now be convinced that such a project should not be successful, as it represents something which is not a valid axiom link. 
\end{remark}

\subsection{Non-Singular MASAs}

In this section, we consider chosen an outlook $\vn{P}$ which is either a regular or a semi-regular MASA in $\infhyp$. We will show a full soundness result for the sequent calculus $\textnormal{MALL}_{\cond{T},\cond{0}}$ (\autoref{MALL}), i.e. we interpret formulas and sequents as conducts and proofs as hyperfinite projects and we show that for all proof $\pi$ of a sequent $\vdash \Gamma$, the interpretation $\Int{\pi}{}$ is a promising project which belongs to $\Int{\Gamma}{}$.

\subsubsection{The Sequent Calculi $\textnormal{MALL}_{\cond{T},\cond{0}}$}


We will briefly define the sequent calculus $\textnormal{MALL}_{\cond{T},\cond{0}}$ for which we show a soundness result. This sequent calculus was defined in order to prove a soundness result for interaction graphs \cite{seiller-goia}. This is the usual sequent calculus for multiplicative-additive linear logic without multiplicative units (but including additive units). Although multiplicative units can be dealt with, they need a more involved sequent calculus with \emph{polarised formulas} that deals with exponential connectives \cite{seiller-goie}. A soundness result for this more involved calculus exists, but the result does not justifies the amount of work needed to correctly define the calculus. The interested reader can have a look at the author's work on exponentials \cite{seiller-goie} to persuade herself that this extended result holds as well.

In earlier works \cite{seiller-goim,seiller-goia,seiller-goig,seiller-goie}, we took into account the locativity of the framework by defining a \emph{localised sequent calculus} $\textnormal{locMALL}_{\cond{T},\cond{0}}$ for which formulas have a specific location and rules are subject to constraints on the locations of the formulas appearing in the sequents. This localised version of the sequent calculus is used in order to prove a soundness result more easily as it presupposes the locativity constraints of the GoI model. The soundness result for the usual non-localised calculus is then obtained by noticing that every formula, thus sequent, and every proof can be \enquote{localised}, i.e. interpreted as a formula, sequent or proof of the localised calculus. We will here define directly localised interpretations of the non-localised sequent calculus in order to limit the space needed to show the results. 

Let us fix $\mathcal{V}=\{X_{i}\}_{i\in\naturalN}$ a set of variables.

\begin{definition}[Formulas of MALL$_{\cond{T,0}}$]
The formulas of MALL$_{\cond{T,0}}$ are defined by the following grammar:
\begin{equation*}
F:=X_{i}~|~X_{i}^{\pol}~|~F\otimes F~|~F\parr F~|~F\with F~|~F\oplus F~|~\cond{0}~|~\cond{T}
\end{equation*}
where the $X_{i}$ are variables.
\end{definition}

\begin{definition}[Proofs of MALL$_{\cond{T,0}}$]
A proof of MALL$_{\cond{T,0}}$ is a derivation obtained from the sequent calculus rules shown in Figure \ref{MALL}.
\end{definition}

\begin{figure}
\begin{center}
$\begin{array}{cc}
\begin{minipage}{5cm}
\begin{prooftree}
\AxiomC{}
\RightLabel{\scriptsize{Ax}}
\UnaryInfC{$\vdash X_{i}^{\pol},X_{i}$}
\end{prooftree}
\end{minipage}
&
\begin{minipage}{5cm}
\begin{prooftree}
\AxiomC{$\vdash A,\Delta$}
\AxiomC{$\vdash A^{\pol},\Gamma$}
\RightLabel{\scriptsize{Cut}}
\BinaryInfC{$\vdash \Delta,\Gamma$}
\end{prooftree}
\end{minipage}
\\~\\
\begin{minipage}{5,5cm}
\begin{prooftree}
\AxiomC{$\vdash A,\Delta$}
\AxiomC{$\vdash B,\Gamma$}
\RightLabel{\scriptsize{$\otimes$}}
\BinaryInfC{$\vdash A\otimes B,\Delta,\Gamma$}
\end{prooftree}
\end{minipage}
&
\begin{minipage}{5cm}
\begin{prooftree}
\AxiomC{$\vdash A,B,\Gamma$}
\RightLabel{\scriptsize{$\parr$}}
\UnaryInfC{$\vdash A\parr B,\Gamma$}
\end{prooftree}
\end{minipage}
\\~\\
\begin{minipage}{5cm}
\begin{prooftree}
\AxiomC{$\vdash A_{i},\Gamma$}
\RightLabel{\scriptsize{$\oplus^{i}$}}
\UnaryInfC{$\vdash A_{0}\oplus A_{1},\Gamma$}
\end{prooftree}
\end{minipage}
&
\begin{minipage}{5cm}
\begin{prooftree}
\AxiomC{$\vdash \Gamma,A$}
\AxiomC{$\vdash \Gamma,B$}
\RightLabel{\scriptsize{$\with$}}
\BinaryInfC{$\vdash \Gamma,A\with B$}
\end{prooftree}
\end{minipage}
\\~\\
\begin{minipage}{5cm}
\begin{prooftree}
\AxiomC{}
\RightLabel{\scriptsize{$\top$}}
\UnaryInfC{$\vdash \top,\Gamma$}
\end{prooftree}
\end{minipage}
&
\begin{minipage}{5cm}
\centering
\begin{prooftree}
\AxiomC{}
\noLine
\UnaryInfC{No rules for $\cond{0}$.}
\end{prooftree}
\end{minipage}
\end{array}
$\end{center}
\caption{Sequent calculus MALL$_{\cond{T,0}}$}\label{MALL}
\end{figure}

\subsubsection{Interpretation of Formulas}

\begin{definition}[Delocations]
Let $\proj{p},\proj{q}$ be projections in $\vn{P}$. A \emph{delocation} from $\proj{p}$ onto $\proj{q}$ is a partial isometry $\theta : \proj{p}\rightarrow \proj{q}$ such that $\theta\in G_{\infhyp}(\vn{P})$.
\end{definition}

To interpret the sequent calculus, we will actually work with the MASA $\vn{P}\oplus\vn{P}$ of the algebra $\vn{M}_{2}(\infhyp)$ in order to distinguish a \emph{primitive space} (the first component of the direct sum $\vn{P}\oplus\vn{P}$) and an \emph{interpretation space} (the second component of the direct sum). Interpretations of proofs and formulas will be elements of the interpretation space, hence the interpretation will in fact take place in $\infhyp$, while the primitive space will be used in order to define correctly the syntax. The following proposition shows that, since the interpretations will be hyperfinite projects defined in the second component of the sum $\vn{P\oplus P}$, the fact that they are promising w.r.t. $\vn{P\oplus P}$ in $\vn{M}_{2}(\infhyp)$ implies that their restriction to $\infhyp$ (the second component) is promising w.r.t. $\vn{P}$.

\begin{proposition}[Restriction]
Let $\de{a}=(\proj{p},0,tr,\finhyp, A)$ be a promising project w.r.t. $\vn{P\oplus P}\masa\vn{M}_{2}(\infhyp)$ such that $\proj{p}\leqslant 0\oplus 1$.
Then $A(0\oplus 1)=(0\oplus 1)A=A$, and $\de{a}$ is a promising project w.r.t. $\vn{P}\masa(\vn{M}_{2}(\infhyp))_{0\oplus 1}\simeq \infhyp$.
\end{proposition}

\begin{proof}
It is clear that $\vn{P}$ is a MASA in $(\vn{M}_{2}(\infhyp))_{0\oplus 1}$. The result is then a direct consequence of \autoref{reducingnormaliser}.
\end{proof}

Let us now define variables. We pick a family of pairwise disjoint projections $(\proj{p}_{i})_{i\in\naturalN}$. The projections $\proj{p}_{i}\oplus 0$ will be called the \emph{primitive locations} of the variables, and one should think of this as our actual set of variables.

\begin{definition}[Variable names]
A \emph{variable name} is an integer $i\in\naturalN$ denoted by capital letters $X,Y,Z$, etc. 
A \emph{variable} is a pair $X_{\theta}=(X,\theta)$ where $X$ is a variable name, i.e. an integer $i$, and $\theta$ is a relocation of $\proj{p}_{i}\oplus 0$ onto a projection $0\oplus\proj{q}_{X_{\theta}}$. The projection $0\oplus\proj{q}_{X_{\theta}}$ is referred to as the \emph{location} of the variable, and we will sometimes allow ourselves to forget about the first component and simply write $\proj{q}_{X_{\theta}}$.
\end{definition}

We now define the interpretation of formulas.

\begin{definition}[Interpretation Basis]
An \emph{interpretation basis} is a map $\delta$ associating to each variable name $X=i$ a dichology $\delta(X)$ of carrier the primitive location $\proj{p}_{i}$ of $X$. This map extends to a function $\bar{\delta}$ which associates, to each variable $X_{\theta}$, the dichology $\bar{\delta}(X_{\theta})=\theta(\delta(X))$ of carrier $\proj{q}_{X_{\theta}}$ -- the location of $X_{\theta}$. The term \emph{interpretation basis} will abusively refer to this extension in the following.
\end{definition}

\begin{definition}[Interpretation of Formulas]
The interpretation $\Int{F}{\delta}$ of a formula $F$ along the interpretation basis $\delta$ is defined inductively as follows:
\begin{itemize}
\item $F=X_{\theta}$. We define $\Int{F}{\delta}$ as the dichology $\bar{\delta}(X_{\theta})$ of carrier $\proj{q}_{X_{\theta}}$;
\item $F=X_{\theta}^{\pol}$. We define $\Int{F}{\delta}=(\Int{X_{\theta}}{\delta})^{\pol}$, a dichology of carrier $\proj{q}_{X_{\theta}}$;
\item $F=A\star B$ ($\star\in\{\otimes,\parr,\with,\oplus$). We define $\Int{F}{\delta}=\Int{A}{\delta}\star\Int{B}{\delta}$, a dichology of carrier $\proj{p+q}$, where $\proj{p}$ and $\proj{q}$  are the respective carriers of $\Int{A}{\delta}$ and $\Int{B}{\delta}$;
\item $F=\cond{T}$ (resp. $F=\cond{0}$). We define $\Int{F}{\delta}=\cond{T}_{0}$ (resp. $\cond{0}_{0}$), the \emph{full conduct} (resp. the \emph{empty conduct} of carrier $0$.
\end{itemize}
\end{definition}

\begin{definition}[Interpretation of Sequents]
A sequent $\vdash \Gamma$ will be interpreted as the $\parr$ of formulas in $\Gamma$, denoted by $\bigparr\Gamma$.
\end{definition}

\subsubsection{Interpretation of proofs}

The introduction rule of the $\parr$ as well as the exchange rule will have a trivial interpretation, since premise and conclusion sequents are interpreted by the same dichology: due to locativity, the commutativity and associativity of $\parr$ are real equalities and not morphisms. Similarly, rules for $\oplus$ have an easy interpretation as it suffices to extend the carrier of the project interpreting the premise to define the interpretation of the conclusion. Moreover, the rule $\top$ has a straightforward interpretation as the project $(0,0,1_{\complexN},\complexN,0)$. Axioms will be easily interpreted by delocations, whose existence is ensured by \autoref{popajones}. The case of cut has already been treated in \autoref{composi50}, and we therefore only need to deal with the introduction rules of $\otimes$ and $\with$.

Given two hyperfinite projects $\de{f}$ and $\de{g}$ in the interpretations of the premises of a tensor ($\otimes$) introduction rule, we will define a hyperfinite project $\de{h}$ in the interpretation of the conclusion. The operation that naturally comes to mind is to define this project as the tensor product of the projects $\de{f}$ and $\de{g}$. It turns out that this interpretation of the $\otimes$ introduction rule is perfectly satisfactory: the following proposition shows that the project $\de{h}$ defined as $\de{f}\otimes\de{g}$ is a project in the interpretation of the conclusion.

\begin{proposition}[Interpretation of the Tensor Rule]\label{tensorprop}
Let $\cond{A,B,C,D}$ be conducts of respective carriers $\proj{p}_{A},\proj{p}_{B},\proj{p}_{C},\proj{p}_{D}$. We have the following inclusion: 
\begin{equation*}
((\cond{A\multimap B})\otimes(\cond{C\multimap D}))\subset(\cond{(A\otimes C})\multimap(\cond{B\otimes D}))
\end{equation*}
\end{proposition}

\begin{proof}
We show that $\cond{(A\otimes C)\multimap(B\otimes D)}$ contains the conduct $\cond{(A\multimap B)\otimes(C\multimap D)}$, for all conducts $\cond{A,B,C,D}$.

We denote by $\proj{p}_{A}$, $\proj{p}_{B}$, $\proj{p}_{C}$ and $\proj{p}_{D}$ the respective carriers of the conducts $\cond{A,B,C,D}$. Let $\de{f}\in\cond{A\multimap B}$ and $\de{g}\in\cond{C\multimap D}$ be the projects:
\begin{eqnarray*}
\de{f}&=&(\proj{p}_{\cond{A}}+\proj{p}_{\cond{B}},f,\phi,\vn{F},F)\\
\de{g}&=&(\proj{p}_{\cond{C}}+\proj{p}_{\cond{D}},g,\psi,\vn{G},G)
\end{eqnarray*}
We will show that for all projects $\de{a}=(\proj{p}_{\cond{A}},a,\alpha,\vn{A},A)$ and $\de{c}=(\proj{p}_{\cond{C}},c,\gamma,\vn{C},C)$ in $\cond{A}$ and $\cond{C}$ respectively, the result of the execution $(\de{f}\otimes\de{g})\plug(\de{a}\otimes\de{c})$ is an element of the conduct $\cond{B\otimes D}$. This will show that any element of $\cond{(A\multimap B)\otimes(C\multimap D)}$ is also an element of $\cond{(A\otimes C)\multimap (B\otimes D)}$.

We have: 
\begin{equation*}
(\de{f}\otimes\de{g})\plug(\de{a}\otimes\de{c})=(\proj{p}_{B}+\proj{p}_{D},w,\nu,\vn{N},P)
\end{equation*}
where 
\begin{eqnarray*}
w&=&\phi\otimes\psi(1_{\vn{F}\vntimes\vn{G}})(a\gamma(1_{\vn{C}})+c\alpha(1_{\vn{A}}))\\
&&~~~~~~+\alpha\otimes\gamma(1_{\vn{A}\vntimes\vn{C}})(f\psi(1_{\vn{G}})+g\phi(1_{\vn{F}}))\\
&&~~~~~~~~~~~~+\ldet(1-(F^{\dagger_{\vn{G}}}+G^{\ddagger_{\vn{F}}})^{\dagger_{\vn{A}\vntimes\vn{C}}}(A^{\dagger_{\vn{C}}}+C^{\ddagger_{\vn{A}}})^{\ddagger_{\vn{F}\vntimes\vn{G}}}) \\
\nu&=&\phi\otimes\psi\otimes\alpha\otimes\gamma\\
\vn{N}&=&\vn{F}\vntimes\vn{G}\vntimes\vn{A}\vntimes\vn{C}\\
P&=&(F^{\dagger_{\vn{G}}}+G^{\ddagger_{\vn{F}}})^{\dagger_{\vn{A}\vntimes\vn{C}}}\plug(A^{\dagger_{\vn{C}}}+C^{\ddagger_{\vn{A}}})^{\ddagger_{\vn{F}\vntimes\vn{G}}}
\end{eqnarray*}
We are going to show that this project is a variant of the project $(\de{f}\plug\de{a})\otimes(\de{g}\plug\de{c})$. We will then conclude that $(\de{f}\otimes\de{g})\plug(\de{a}\otimes\de{c})\in\cond{(A\otimes C)\multimap(B\otimes D)}$ using \autoref{variantgdi50}.
\begin{itemize}
\item It is clear that $\vn{N}$ is equal to the dialect of $(\de{f}\plug\de{a})\otimes(\de{g}\plug\de{c})$ up to a commutativity isomorphism. Indeed, the dialect of $(\de{f}\plug\de{a})\otimes(\de{g}\plug\de{c})$ is equal to $\vn{F}\otimes\vn{A}\otimes\vn{G}\otimes\vn{C}$ and the morphism $\phi=\text{Id}_{\vn{F}}\otimes\tau\otimes\text{Id}_{\vn{C}}$, where $\tau:\vn{A}\otimes\vn{G}\rightarrow\vn{G}\otimes\vn{A}, x\otimes y\mapsto y\otimes x$, is a isomorphism between this algebra and $\vn{N}$;
\item It is also clear that $\nu$ is equal to $\mu\circ\phi^{-1}$ where $\mu$ is the pseudo-trace of the project $(\de{f}\plug\de{a})\otimes(\de{g}\plug\de{c})$;
\item Since $F$ and $A$ are elements of $(\infhyp)_{\proj{p}_{A}+\proj{p}_{B}}$ and $G$ and $C$ are elements of $(\infhyp)_{\proj{p}_{C}+\proj{p}_{D}}$, we deduce from the pairwise disjointness of projections that:
\begin{equation*}
P = (F^{\dagger_{\vn{G}}})^{\dagger_{\vn{A}\vntimes\vn{C}}}\plug(A^{\dagger_{\vn{C}}})^{\ddagger_{\vn{F}\vntimes\vn{G}}}+(G^{\dagger_{\vn{F}}})^{\dagger_{\vn{A}\vntimes\vn{C}}}\plug(C^{\dagger_{\vn{A}}})^{\ddagger_{\vn{F}\vntimes\vn{G}}}
\end{equation*}
Once again, this is equal, modulo $\phi$, to the operator of the project $(\de{f}\plug\de{a})\otimes(\de{g}\plug\de{c})$:
\begin{equation*}
(F^{\dagger_{\vn{A}}}\plug A^{\dagger_{\vn{F}}})^{\dagger_{\vn{C}\vntimes\vn{G}}}+(G^{\dagger_{\vn{C}}}\plug C^{\dagger_{\vn{G}}})^{\ddagger_{\vn{A}\vntimes\vn{F}}}
\end{equation*}
\item Using the fact that $\phi\otimes\psi(1_{\vn{F}\vntimes\vn{G}})=\phi\otimes\psi(1_{\vn{F}}\otimes1_{\vn{G}})=\phi(1_{\vn{F}})\psi(1_{\vn{G}})$ and that $\alpha\otimes\gamma(1_{\vn{A}\vntimes\vn{C}})=\alpha(1_{\vn{A}})\gamma(1_{\vn{C}})$, we obtain:
\begin{eqnarray*}
w&=& \gamma(1_{\vn{C}})\psi(1_{\vn{G}})(a\phi(1_{\vn{F}})+f\alpha(1_{\vn{A}}))\\&&~~~~~~+\alpha(1_{\vn{A}})\phi(1_{\vn{F}})(c\psi(1_{\vn{G}})+g\gamma(1_{\vn{C}}))\\
&&~~~~~~~~~~~~+\ldet(1-(F^{\dagger_{\vn{G}}}+G^{\ddagger_{\vn{F}}})^{\dagger_{\vn{A}\vntimes\vn{C}}}(A^{\dagger_{\vn{C}}}+C^{\ddagger_{\vn{A}}})^{\ddagger_{\vn{F}\vntimes\vn{G}}}) 
\end{eqnarray*}
Moreover, since $AG=0$, \autoref{sumofldet} allows us to conclude: 
\begin{eqnarray*}
&&\ldet(1-(F^{\dagger_{\vn{G}}}+G^{\ddagger_{\vn{F}}})^{\dagger_{\vn{A}\vntimes\vn{C}}}(A^{\dagger_{\vn{C}}}+C^{\ddagger_{\vn{A}}})^{\ddagger_{\vn{F}\vntimes\vn{G}}})\\
&=&\ldet(1-((F^{\dagger_{\vn{G}}})^{\dagger_{\vn{A}\vntimes\vn{C}}}(A^{\dagger_{\vn{C}}})^{\ddagger_{\vn{F}\vntimes\vn{G}}}+(G^{\ddagger_{\vn{F}}})^{\dagger_{\vn{A}\vntimes\vn{C}}}(C^{\ddagger_{\vn{A}}})^{\ddagger_{\vn{F}\vntimes\vn{G}}}))\\
&=&\ldet(1-(F^{\dagger_{\vn{G}}})^{\dagger_{\vn{A}\vntimes\vn{C}}}(A^{\dagger_{\vn{C}}})^{\ddagger_{\vn{F}\vntimes\vn{G}}})+\ldet(1-(G^{\ddagger_{\vn{F}}})^{\dagger_{\vn{A}\vntimes\vn{C}}}(C^{\ddagger_{\vn{A}}})^{\ddagger_{\vn{F}\vntimes\vn{G}}})\\
&=&\gamma(1_{\vn{C}})\psi(1_{\vn{G}})\ldet(1-F^{\dagger_{\vn{A}}}A^{\dagger_{\vn{F}}})+ \alpha(1_{\vn{A}})\phi(1_{\vn{F}})\ldet(1-G^{\ddagger_{\vn{C}}}C^{\ddagger_{\vn{G}}})
\end{eqnarray*}
from which we can conclude:
\begin{eqnarray*}
w&=&\gamma(1_{\vn{C}})\psi(1_{\vn{G}})[a\phi(1_{\vn{F}})+f\alpha(1_{\vn{A}})+\ldet(1-F^{\dagger_{\vn{A}}}A^{\dagger_{\vn{F}}})]\\
&&~~~+\alpha(1_{\vn{A}})\phi(1_{\vn{F}})[c\psi(1_{\vn{G}})+g\gamma(1_{\vn{C}})+\ldet(1-G^{\ddagger_{\vn{C}}}C^{\ddagger_{\vn{G}}})]
\end{eqnarray*}
which is the wager of the project  $(\de{f}\plug\de{a})\otimes(\de{g}\plug\de{c})$.
\end{itemize}
We thus deduce that for all $\de{f,g,a,c}$, we have $(\de{f}\otimes\de{g})\plug(\de{a}\otimes\de{c})\in\cond{(A\otimes C)\multimap(B\otimes D)}$ by \autoref{variantgdi50} page \pageref{variantgdi50}. Finally, we showed that $\cond{(A \multimap B)\otimes(C\multimap D)}\subset\cond{(A\otimes C)\multimap(B\otimes D)}$.
\end{proof}

We will now interpret the introduction rule for $\with$. We will interpret a proof ending with a $\with$ introduction rule by the sum of the projects $\de{f}_{p+q}$ and $\de{g}_{p+q}$, where $\de{f}$ and $\de{g}$ -- of respective carriers $p$ and $q$ -- are the interpretations of the sub-proofs whose conclusions are the premises of the $\with$ rule. In order to perform this operation, it is necessary to first delocalise the interpretations of the premises as the premises do not have disjoint locations. Once this relocation is done, we can define the project $\de{h}$ as $\theta_{1}(\de{f})\with\theta_{2}(\de{g})$ -- where $\theta_{1}$ and $\theta_{2}$ are the delocations just mentioned. We then apply the project implementing distributivity in order to superpose the contexts. We refer the reader to the interpretation of proofs of $MALL_{\cond{T},\cond{0}}$ in interaction graphs \cite{seiller-goia} for a more thorough explanation of this.

For the next result, we will be needing a proposition shown in earlier work \cite{seiller-goia} in a different setting, but whose proof easily adapts to the matricial (as well as the hyperfinite) GoI model. This proposition states any element of a dichology\footnote{We must inform the reader that the terminology here differs from the cited paper: what we call here a dichology is called a \emph{behaviour} in the interaction graphs constructions \cite{seiller-goia}.} $\cond{A\with B}$ is \emph{observationally equivalent} to a sum $\de{a}_{1}\otimes\de{0}_{q}+\de{a}_{2}\otimes\de{0}_{p}$ with $\de{a}_{1}\in\cond{A}, \de{a}_{2}\in\cond{B}$, where observational equivalence is defined, e.g. on elements of a conduct $\cond{A}$, as follows:
$$\de{a}\sim_{\cond{A}}\de{a'}\Leftrightarrow \forall\de{t}\in\cond{A}^{\pol}, \sca[\textnormal{mat}]{a}{t}=\sca[\textnormal{mat}]{a'}{t}$$
We recall that this notion of equivalence is a congruence, i.e. if $\de{a}\sim_{\cond{A}} \de{a'}$, then for all $\de{f}\in\cond{A\multimap B}$ we have $\de{f\plug a}\sim_{\cond{B}}\de{f\plug a'}$. In particular, if $\de{a}\sim_{\cond{A}} \de{a'}$ then $\de{a\otimes b}\sim_{\cond{A\otimes B}} \de{a'\otimes b}$ for all $\de{b}\in\cond{B}$.

\begin{proposition}\label{lemmawith}
Let $\cond{A,B}$ be dichologies of respective carriers $p,q$. For any element $\de{a}\in\cond{A\with B}$, there exists elements $\de{a}_{1}\in\cond{A}$ and $\de{a}_{2}\in\cond{B}$ such that $\de{a}_{1}\otimes\de{0}_{q}+\de{a}_{2}\otimes\de{0}_{p}\sim_{\cond{A\with B}}\de{a}$.
\end{proposition}

\begin{corollary}\label{corowith}
Let $\cond{A,B,C}$ be dichologies of carriers $p,q,r$ respectively, and $\de{f}$ a project of carrier $p+q+r$. If $\de{a}$ maps every sum $\de{a}_{1}\otimes\de{0}_{q}+\de{a}_{2}\otimes\de{0}_{p}$ to an element of $\cond{C}$, then $\de{f}$ belongs to $\cond{(A\with B)\multimap C}$.
\end{corollary}

\begin{proof}
Let $\de{a}$ be an element of $\cond{A\with B}$. Then $\de{a}$ is equivalent to some sum $\de{a}_{1}\otimes\de{0}_{q}+\de{a}_{2}\otimes\de{0}_{p}$. As a consequence, $\de{f}\plug\de{a}$ is equivalent to $\de{f}\plug (\de{a}_{1}\otimes\de{0}_{q}+\de{a}_{2}\otimes\de{0}_{p})$. Since the latter is an element of $\cond{C}$, we can deduce that $\de{f}\plug\de{a}$ belongs to $\cond{C}$. Indeed, for all $\de{t}\in\cond{C^{\pol}}$, $\sca[\textnormal{mat}]{f\plug a}{t}=\sca[\textnormal{mat}]{f}{a\otimes t}$. Moreover, $\de{a\otimes t}$ is equivalent to $(\de{a}_{1}\otimes\de{0}_{q}+\de{a}_{2}\otimes\de{0}_{p})\otimes \de{t}$ since $\de{a}$ is equivalent to $(\de{a}_{1}\otimes\de{0}_{q}+\de{a}_{2}\otimes\de{0}_{p})$. Hence $\sca[\textnormal{mat}]{f\plug a}{t}=\sca[\textnormal{mat}]{f}{(\de{a}_{1}\otimes\de{0}_{q}+\de{a}_{2}\otimes\de{0}_{p})\otimes \de{t}}=\sca[\textnormal{mat}]{f\plug (\de{a}_{1}\otimes\de{0}_{q}+\de{a}_{2}\otimes\de{0}_{p})}{t}$. Since the latter is different from $0,\infty$, we have shown that $\de{f\plug a}$ is orthogonal to $\de{t}$. This shows that $\de{f\plug a}$ is an element of $\cond{C}$ for all $\de{a}\in\cond{\with B}$, i.e. $\de{f}\in\cond{(A\with B)\multimap C}$.
\end{proof}


\begin{proposition}[Interpretation of the $\with$ Rule]\label{withprop}
Let $\cond{A,B,C}$ be dichologies of respective pairwise disjoint carriers $\proj{p}_{A},\proj{p}_{B},\proj{p}_{C}$, and let $\phi(\cond{A})$ be a delocation of $\cond{A}$, with $\phi\in G_{\infhyp}(\vn{P})$, whose carrier is a projection disjoint from the projections $\cond{A,B,C}$. Then for all delocations\footnote{Supposing of course that the carriers are pairwise disjoint.} $\theta_{1},\theta_{2},\theta_{3}$ in $G_{\infhyp}(\vn{P})$, there exists a project $\de{With}$ in the dichology:
\begin{equation*}
((\cond{A\multimap B})\with(\cond{\phi(A)\multimap C}))\multimap(\cond{\theta_{1}(A)}\multimap(\cond{\theta_{2}(B)\with \theta_{3}(C)}))
\end{equation*}
Moreover, $\de{With}$ is promising w.r.t. the outlook $\vn{P}$.
\end{proposition}

\begin{proof}
We chose projections $\proj{p}'_{A}\eqp \proj{p}_{A}$, $\proj{p}''_{A}\eqp \proj{p}_{A}$, $\proj{p}'_{B}\eqp \proj{p}_{B}$ and $\proj{p}'_{C}\eqp \proj{p}_{C}$ which are pairwise disjoint and disjoint from the carriers of $\cond{A,B,C}$. Then, since $\vn{P}$ is either regular or semi-regular there exists by \autoref{popajones} partial isometries $\phi$, $\theta_{1}$, $\theta_{2}$ and $\theta_{3}$ in the normalising groupoid $G_{\infhyp}(\vn{P})$:
\begin{equation*}
\left\{\begin{array}{rcrcl}
\phi &:& \proj{p}_{A}&\rightarrow &\proj{p}'_{A}\\
\theta_{1} &: &\proj{p}_{A}&\rightarrow &\proj{p}''_{A}\\
\theta_{2} &: & \proj{p}_{B}&\rightarrow &\proj{p}'_{B}\\
\theta_{3} &: & \proj{p}_{C}&\rightarrow &\proj{p}'_{C}
\end{array}\right.
\end{equation*}
We will write $p=p_{A}+p_{B}+p_{A}'+p_{C}$ and $p'=p_{A}''+p_{B}'+p_{C}'$, and we define $\de{distr}$ as the project $(\proj{p+p}',0,\kappa,\vn{K},K)$ where:
\begin{eqnarray*}
\kappa&=&\frac{1_{\complexN}\oplus 1_{\complexN}}{2}\\
\vn{K}&=&\complexN\oplus\complexN\\
K&=&(\theta_{1}+\theta_{1}^{\ast}+\theta_{2}+\theta_{2}^{\ast})\oplus(\theta_{1}\phi^{\ast}+\phi\theta_{1}^{\ast}+\theta_{3}+\theta_{3}^{\ast})
\end{eqnarray*}
We now show that $\de{distr}$ is promising w.r.t. the outlook $\vn{P}$ and that it is an element of the dichology:
\begin{equation*}
\cond{((A\multimap B)\with(\phi(A)\multimap C))\multimap (\theta_{1}(A)\multimap(\theta_{2}(B)\with \theta_{3}(C)))}
\end{equation*}
To show that it belongs to the latter dichology, we will compute the execution of $\de{distr}$ with an element $\de{f}\in\cond{(A\multimap B)\with (\phi(A)\multimap C)}$, and then compute the execution of $\de{distr \plug f}$ with an element $\theta_{1}(\de{a})\in\theta_{1}(\cond{A})$. The result, i.e. the project $\de{(distr \plug f)\plug}\theta_{1}(\de{a})$ will then be shown to belong to the dichology $\theta_{2}(\cond{B})\with \theta_{3}(\cond{C})$. 

To ease the computations, we will write $\proj{q}=\proj{p}_{A}+\proj{p}_{B}$ and $\proj{r}=\proj{p}_{A}'+\proj{p}_{C}$.

Let us pick $\de{f}\in\cond{(A\multimap B)\with (\phi(A)\multimap C)}$ and $\theta_{1}(\de{a})\in\theta_{1}(\cond{A})$ as follows. Notice that based on \autoref{corowith} we can suppose, without loss of generality, that $\de{f}$ is of the form $\de{f}_{1}\otimes\de{0}_{r}+\de{0}_{q}\otimes\de{f}_{2}$.
\begin{eqnarray*}
\de{f}&=&(\proj{q+r},0,\phi_{1}\oplus\phi_{2},\vn{F}_{1}\oplus\vn{F}_{2},F_{1}\oplus F_{2})\\
\theta_{1}(\de{a})&=&(\proj{p}_{A}'',0,\alpha,\vn{A},\theta_{1}A\theta_{1}^{\ast})
\end{eqnarray*}
Let us stress that $F_{1}=qF_{1}=qF_{1}q$ and $F_{2}=rF_{2}=rF_{2}r$. We will write $\vn{F}=\vn{F}_{1}\oplus\vn{F}_{2}$, $\phi=\phi_{1}\oplus\phi_{2}$ and $F=F_{1}\oplus F_{2}$.

The computation of $\de{w}=\de{distr}\plug\de{f}$ gives $\de{w}=(\proj{p'},w,\xi,\vn{W},W)$ 
where:
\begin{eqnarray*}
w&=&0\\
\xi&=&\frac{1\oplus 1}{2}\otimes(\phi_{1}\oplus \phi_{2})\equiv \frac{\phi_{1}\oplus\phi_{2}\oplus \phi_{1}\oplus \phi_{2}}{2}=\frac{\phi\oplus\phi}{2}\\
\vn{W}&=&(\complexN\oplus\complexN)\vntimes\vn{F}\equiv\vn{F}_{1}\oplus\vn{F}_{2}\oplus\vn{F}_{1}\oplus\vn{F}_{2}=\vn{F}\oplus\vn{F}
\end{eqnarray*}
and 
\begin{eqnarray*}
W&=&K^{\dagger_{\vn{F}}}\plug F^{\ddagger_{\vn{K}}}\\
&=&(\theta_{1}+\theta_{1}^{\ast}+\theta_{2}+\theta_{2}^{\ast})^{\dagger_{\vn{F}}}F(\theta_{1}+\theta_{1}^{\ast}+\theta_{2}+\theta_{2}^{\ast})^{\dagger_{\vn{F}}}\\
&&~~~~~\oplus(\theta_{1}\phi^{\ast}+\phi\theta_{1}^{\ast}+\theta_{3}+\theta_{3}^{\ast})^{\dagger_{\vn{F}}}F(\theta_{1}\phi^{\ast}+\phi\theta_{1}^{\ast}+\theta_{3}+\theta_{3}^{\ast})^{\dagger_{\vn{F}}}\\
&=&(\theta_{1}+\theta_{1}^{\ast}+\theta_{2}+\theta_{2}^{\ast})F_{1}(\theta_{1}+\theta_{1}^{\ast}+\theta_{2}+\theta_{2}^{\ast})\oplus 0\\
&&~~~~~\oplus 0\oplus(\theta_{1}\phi^{\ast}+\phi\theta_{1}^{\ast}+\theta_{3}+\theta_{3}^{\ast})F_{2}(\theta_{1}\phi^{\ast}+\phi\theta_{1}^{\ast}+\theta_{3}+\theta_{3}^{\ast})
\end{eqnarray*}
We will write $W_{1}$ (resp. $W_{2}$) the first (resp. the second) non-zero component of $W$; i.e. $W=W_{1}\oplus 0\oplus 0\oplus  W_{2}$. 

We now compute $\de{r}=\de{w}\plug \theta_{1}({a})$:
\begin{equation*}
\de{r}=(p'_{B}+p'_{C},m,\frac{1}{2}(\phi\otimes\alpha)\oplus\frac{1}{2}(\phi\otimes\alpha),\vn{F\otimes A}\oplus\vn{F\otimes A},W_{1}^{\dagger_{\vn{A}}}\plug \theta_{1}A^{\ddagger_{\vn{F}}}\theta_{1}^{\ast}\oplus 0\oplus 0\oplus W_{2}^{\dagger_{\vn{A}}}\plug \theta_{1}A^{\ddagger_{\vn{F}}}\theta_{1}^{\ast})
\end{equation*}
where the wager $m$ is computed as follows:
\begin{eqnarray*}
m&=&\ldet(1-(W_{1}\oplus 0\oplus 0\oplus W_{2})^{\dagger_{\vn{A}}}\theta_{1}A^{\ddagger_{\vn{F\oplus F}}}\theta_{1}^{\ast})\\
&=&\ldet(1-W_{1}^{\dagger_{\vn{A}}}\theta_{1}A^{\ddagger_{\vn{F}}}\theta_{1}^{\ast}\oplus 0\oplus 0\oplus  W_{2}^{\dagger_{\vn{A}}}\theta_{1}A^{\ddagger_{\vn{F}}}\theta_{1}^{\ast})\\
&=&\ldet_{\tr\otimes \frac{1}{2}(\phi_{1}\otimes \alpha)}(1-W_{1}^{\dagger_{\vn{A}}}\theta_{1}A^{\ddagger_{\vn{F}}}\theta_{1}^{\ast})+\ldet_{\tr\otimes \frac{1}{2}(\phi_{2}\otimes \alpha)}(1-W_{2}^{\dagger_{\vn{A}}}\theta_{1}A^{\ddagger_{\vn{F}}}\theta_{1}^{\ast})\\
&=&\frac{1}{2}\ldet_{\tr\otimes \phi_{1}\otimes \alpha}(1-W_{1}^{\dagger_{\vn{A}}}\theta_{1}A^{\ddagger_{\vn{F}}}\theta_{1}^{\ast})+\frac{1}{2}\ldet_{\tr\otimes \phi_{2}\otimes \alpha}(1-W_{2}^{\dagger_{\vn{A}}}\theta_{1}A^{\ddagger_{\vn{F}}}\theta_{1}^{\ast})
\end{eqnarray*}

We can also rewrite the last component of $\de{r}$ as follows:
$$W_{1}^{\dagger_{\vn{A}}}\plug \theta_{1}A^{\ddagger_{\vn{F}}}\theta_{1}^{\ast}\oplus 0\oplus 0\oplus W_{2}^{\dagger_{\vn{A}}}\plug \theta_{1}A^{\ddagger_{\vn{F}}}\theta_{1}^{\ast}
=
\theta_{2}(F_{1}^{\dagger_{\vn{A}}}\plug A^{\ddagger_{\vn{F}}})\oplus 0\oplus 0\oplus \theta_{3}(F_{2}^{\dagger_{\vn{A}}}\plug \phi(A^{\ddagger_{\vn{F}}}))$$

In other words, we have that $\de{r}$ is equal, up to inflation, to $\frac{1}{2}\theta_{2}(\de{f}_{1}\plug\de{a})\otimes\de{0}_{p'_{C}}+\frac{1}{2}\theta_{3}(\de{f}_{2}\plug\phi(\de{a}))\otimes\de{0}_{p'_{B}}$. But we picked $\de{f}_{1}$ in $\cond{A\multimap B}$, so $\de{f}_{1}\plug \de{a}$ is an element of $\cond{B}$, hence $\frac{1}{2}\theta_{2}(\de{f}_{1}\plug\de{a})$ is an element of $\theta_{2}(\de{B})$ (notice that if $\de{a}$ belongs to a conduct $\cond{A}$ and $\lambda\neq 0$ is a real number, then $\lambda\de{a}$ is an element of $\cond{A}$). Similarly, we can show that $\frac{1}{2}\theta_{3}(\de{f}_{2}\plug\phi(\de{a}))$ is an element of $\theta_{3}(\cond{C})$. Hence $\theta_{2}(\de{f}_{1}\plug\de{a})\otimes\de{0}_{p'_{C}}+\theta_{3}(\de{f}_{2}\plug\phi(\de{a}))\otimes\de{0}_{p'_{B}}$ is an element of $\theta_{2}(\cond{B})\with\theta_{3}(\cond{C})$, which ends the proof that $\de{distr}$ belongs to the right dichology.

Lastly, we need to show that $\de{distr}$ is promising w.r.t. $\vn{P}$. However, since the project is constructed from delocations in the normalising groupoid of $\vn{P}$, this is immediate.

\end{proof}

\begin{remark}\label{withbar}
The interpretation of the $\with$ introduction rule will therefore be defined as the relatively complex construction\footnote{We recall that $\theta_{1}$ and $\theta_{2}$ are well-chosen delocations.} $\de{f}\overline{\with}\de{g}=[\de{With}](\theta_{1}(\de{f})\with\theta_{2}(\de{g}))$. This construction should not hide however the simplicity of the underlying idea. Indeed, given two projects $\de{f}=(\proj{p+r},0,\phi,\vn{F},F)$ and $\de{g}=(\proj{q+r},0,\gamma,\vn{G},G)$, we are just constructing the project: 
\begin{equation*}
\de{f}\overline{\with}\de{g}=(\proj{p+q+r},0,\frac{1}{2}(\phi\oplus \gamma), \vn{F\oplus G}, F\oplus G)
\end{equation*} 
\end{remark}

\subsubsection{Soundness}

In order to state and show the full soundness result, we first define the interpretation of proofs.

\begin{definition}[Interpretations of Proofs]\label{intproof}
We inductively define the interpretation of a proof $\Pi$:
\begin{itemize}
\item if $\Pi$ is an axiom rule introducing the sequent $\vdash X_{\theta};X_{\phi}$ ($\theta,\phi$ are disjoint), we define the interpretation $\Pi^{\bullet}$ as the project $(\proj{q}_{X_{\theta}}+\proj{q}_{X_{\phi}},0,\tr,\finhyp,\theta\phi^{\ast}+\phi\theta^{\ast})$;
\item if $\Pi$ is obtained by application of a rule $\parr$, or an exchange rule, to a proof $\Pi_{1}$, we define $\Pi^{\bullet}=\Pi_{1}^{\bullet}$;
\item if $\Pi$ is obtained by applying a $\oplus$ rule to a proof $\Pi_{1}$ whose interpretation's carrier is $\proj{p}$, then $\Pi^{\bullet}=(\Pi_{1}^{\bullet})_{\proj{p+q}}$ where $\proj{q}$ is the carrier of the interpretation of the introduced formula;
\item if $\Pi$ is obtained by applying a cut rule between two proofs $\Pi_{1}$ and $\Pi_{2}$, then $\Pi^{\bullet}=\Pi^{\bullet}_{1}\plug \Pi^{\bullet}_{2}$;
\item if $\Pi$ is obtained by applying a $\otimes$ introduction rule to the proofs $\Pi_{1}$ and $\Pi_{2}$, then $\Pi^{\bullet}=\Pi^{\bullet}_{1}\otimes\Pi^{\bullet}_{2}$;
\item if $\Pi$ is obtained by the application of a $\with$ rule on the proofs $\Pi_{1}$ and $\Pi_{2}$ interpreted by projects $\Pi_{1}^{\bullet}$ and $\Pi_{2}^{\bullet}$, we then define $\Pi^{\bullet}=[\de{With}](\theta_{1}(\Pi^{\bullet}_{1})\with\theta_{2}(\Pi^{\bullet}_{2}))$ where $\theta_{1},\theta_{2}$ are delocations of $\Pi_{1}^{\bullet}$ and $\Pi_{2}^{\bullet}$ onto disjoint projections, and where $\de{With}$ is the project whose existence is ensured by \autoref{withprop}.
\end{itemize}
\end{definition}

\begin{theorem}[Full Soundness]
Let $\pi$ be a proof of the sequent $\vdash \Gamma$ in $MALL_{\cond{T},\cond{0}}$, and $\delta$ an interpretation basis. Then the interpretation $\pi^{\bullet}$ of $\pi$ is a promising project w.r.t. $\vn{P}$ in the interpretation $\Int{\vdash \Gamma; A}{\delta}$ of $\vdash \Gamma$.
\end{theorem}

\begin{proof}
The proof is a simple induction. The base case is the simple observation that a fax -- the project $(\proj{q}_{X_{\theta}}+\proj{q}_{X_{\phi}},0,\tr,\finhyp,\theta\phi^{\ast}+\phi\theta^{\ast})$ -- constructed from partial isometries $\theta,\phi$ in the normalising groupoid of $\vn{P}$ is promising w.r.t. $\vn{P}$. The induction steps are then consequences of \autoref{tensorprop} and \autoref{withprop}, as well as \autoref{composi50}.
\end{proof}

\subsection{Regular MASAs}

To interpret exponentials of Elementary Linear Logic (ELL), we will consider the construction proposed by Girard \cite{goi5} and exposed above. There is one major problem with this construction, however. Indeed, if $\de{a}$ is a promising hyperfinite project w.r.t. the outlook $\vn{P}$, it is clear that the hyperfinite project $\de{\oc_{\Omega} a}$ is promising w.r.t the outlook $\Omega(\vn{P}\otimes\vn{Q})$ where $\vn{Q}$ is a MASA in $\finhyp$. However, if it is obvious that $\Omega(\vn{P}\otimes \vn{Q})$ is a MASA in $\infhyp$, it won't be true, in general, that $\Omega(\vn{P}\otimes \vn{Q})=\vn{P}$. As those are both MASAs in $\infhyp$, the two algebras $\Omega(\vn{P}\otimes \vn{Q})$ and $\vn{P}$ are diffuse abelian von Neumann algebras, thus isomorphic \emph{as von Neumann algebras}. This is however too weak a result as this isomorphism is not in general realised by a unitary operator, a necessary condition for an adequate interpretation of the promotion rule.

\begin{proposition}\label{regexp}
Let $\de{a}$ be a promising project w.r.t. the outlook $\vn{P}$. Suppose that $\vn{P}$ is a regular MASA in $\infhyp$. Then there exists a partial isometry $u$ such that $u \Omega(A) u^{\ast}$ is a partial symmetry in the normalising groupoid of $\vn{P}$.
\end{proposition}

\begin{proof}
This proof relies on two hypotheses: the fact that the outlook is a regular MASA, and the fact that the operator $A$ is an element of $p\infhyp p$, i.e. an element of a type $\text{II}_{1}$ hyperfinite factor. Indeed, since $\de{a}$ is promising w.r.t. $\vn{P}$, $A$ is an element of the normalising groupoid of $\vn{P}\otimes\vn{Q}$, where $\vn{Q}$ is a maximal abelian sub-algebra of $\vn{A}$ which is obviously \enquote{regular} as $\vn{A}$ is a finite factor of type $\text{I}$. Since $\vn{P}$ and $\vn{Q}$ are regular, their tensor product is a regular MASA in $\infhyp\otimes\vn{A}$, and therefore\footnote{We allow ourselves a small abuse of notations here, as $\vn{A}$ is not a MASA in $\finhyp$. However, one can chose the embedding of $\vn{A}$ into $\finhyp$ in such a way that the image of $\vn{A}$ is a regular MASA in $\finhyp$.} $\Omega(\vn{P}\otimes\vn{Q})$ is a regular MASA in $\infhyp$.

Moreover, $A=pAp$ where $p$ is a finite projection in $\vn{P}$. Then $\Omega(A)=\Omega(p)\Omega(A)\Omega(p)$, which implies that $\Omega(A)$ is an element of the normalising groupoid of $\Omega(p)\Omega(\vn{P}\otimes\vn{Q})\Omega(p)$, which is a MASA in $\Omega(p)\infhyp\Omega(p)$ by \autoref{singularprojection}. Let us pick $p'$ a projection in $\infhyp$ with the same trace as $p$  -- i.e. $p'$ is Murray von Neumann equivalent to $p$ -- and therefore with the same trace as $\Omega(p)$. We can then consider the regular MASA $p'\vn{P}p'$ in $p\infhyp p$. Since $p'$ and $\Omega(p)$ have equal traces, there exists a partial isometry $u$ such that $uu^{\ast}=p'$ and $u^{\ast}u=\Omega(p)$. Then $u^{\ast}p'\vn{P}p'u$ is a MASA in $\Omega(p)\infhyp\Omega(p)$. Indeed, the image of a MASA $\vn{A}\subset \vn{R}\subset\B{\hil{H}}$ through unitary conjugation by $u:\hil{H}\rightarrow\hil{K}$ is a MASA $u\vn{A}u^{\ast}\subset u\vn{R}u^{\ast}\subset\B{\hil{K}}$. The commutativity of $u\vn{A}u^{\ast}$ is a consequence of the commutativity of $\vn{A}$. Moreover, if it were not maximal, this would contradict the maximality of $\vn{A}$: if there exists an abelian $\vn{B}$ such that $u\vn{A}u^{\ast}\subsetneq \vn{B}\subsetneq u\vn{R}u^{\ast}$, then we could deduce the inclusions $\vn{A}\subsetneq u^{\ast}\vn{B}u\subsetneq \vn{R}$ by conjugation w.r.t $u^{\ast}$.

We now use the result of Connes, Feldman and Weiss (\autoref{connes}) showing that two regular MASAs in the hyperfinite factor of type $\text{II}_{1}$ are unitarily equivalent. Therefore, there exists a unitary $v\in \Omega(p)\infhyp\Omega(p)$ such that $v^{\ast}(u^{\ast}p'\vn{P}p'u)v=\Omega(p)\Omega(\vn{P}\otimes\vn{Q})\Omega(p)$. We can show that the product $uv$ is a partial isometry: $uv(uv)^{\ast}uv=uvv^{\ast}u^{\ast}uv=u\Omega(p)\Omega(p)v=uv$. We will show that this is the partial isometry we were after.

To end the proof, we now show that $(uv)^{\ast}A(uv)$ is a partial symmetry in the normalising groupoid of $\vn{P}$. For this, we first show that it is a hermitian, i.e. $(uv)^{\ast}A(uv)=((uv)^{\ast}A(uv))^{\ast}$, by using that $A=A^{\ast}$.
\begin{eqnarray*}
(v^{\ast}u^{\ast}Auv)^{\ast}&=&v^{\ast}u^{\ast}A^{\ast}uv\\
&=&v^{\ast}u^{\ast}Auv
\end{eqnarray*}
Then, we need to show that it is a partial isometry by proving that e.g. $v^{\ast}u^{\ast}Auv(v^{\ast}u^{\ast}Auv)^{\ast}$ is a projection. We use in the computation the facts that $AA^{\ast}=q$ is a projection, that $Ap=A$ and that the unitary conjugate of a projection is a projection.
\begin{eqnarray*}
v^{\ast}u^{\ast}Auv(v^{\ast}u^{\ast}Auv)^{\ast}&=&v^{\ast}u^{\ast}Auvv^{\ast}u^{\ast}A^{\ast}uv\\
&=&v^{\ast}u^{\ast}Au\Omega(p)u^{\ast}A^{\ast}uv\\
&=&v^{\ast}u^{\ast}ApA^{\ast}uv\\
&=&v^{\ast}u^{\ast}AA^{\ast}uv\\
&=&v^{\ast}u^{\ast}quv
\end{eqnarray*}
The partial isometry $uv$ is therefore the one we were after.
\end{proof}

One can notice that the interpretations of the contraction and functorial promotion rules only use promising projects w.r.t. $\vn{P}$. From this and the preceding proposition, one can easily show an extension of the soundness result stated above for the sequent calculi ELL$_{\textnormal{pol}}$ and ELL$_{\textnormal{comp}}$ considered in the author's work on interaction graphs \cite{seiller-goie} as soon as the outlook is a regular MASA. Let us notice that this proposition do not depend on the morphism $\Omega$ chosen to define the exponentials (the soundness result however depends on $\Omega$ since not all choices of morphisms would allow for the interpretation of functorial promotion).

It is then natural to ask oneself if the converse of this result holds, i.e. if the fact that $\vn{P}$ is not regular implies that one cannot interpret (at least one) exponential connective. We will not fully answer this question in this paper, but we will discuss it anyway.

Let us first consider the Pukansky invariant of the outlook $\vn{P}$ and of the sub-algebra $\Omega(\vn{P}\otimes \vn{Q})$ (using the same notations as in the preceding proof). It is known\footnote{White \cite{white2} showed that all subset of $\naturalN\cup\{\infty\}$ is the Pukansky invariant of a MASA in $\finhyp$.} that there exists singular MASAs in $\finhyp$ whose Pukansky invariant is included in $\{2,3,\dots,\infty\}$, and the sub-algebra $\vn{Q}$ satisfies $\puk{\vn{Q}}=\{1\}$ since it is regular (\autoref{pukreg}). Using \autoref{tenseurpuk}, we get that $\puk{\Omega(\vn{P}\otimes\vn{Q})}$ contains $1$, and it is therefore impossible in this case that $\Omega(\vn{P}\otimes\vn{Q})$ and $\vn{P}$ be unitarily equivalent.

However, the Pukansky invariant of a semi-regular MASA is a subset of $\naturalN\cup\{\infty\}$ that contains $1$ (from \autoref{puksemireg}).  Then, by using \autoref{tenseurpuk}, one shows that in this case $\puk{\Omega(\vn{P\otimes Q})}=\puk{\vn{P}}$. It is therefore not possible to show the reciprocal statement of \autoref{regexp} in this manner. We conjecture that there exist perennializations $\Omega$ and semi-regular outlooks $\vn{P}$ such that (the equivalent of) \autoref{regexp} holds. We also conjecture that there exists perennializations $\Omega$ and semi-regular outlooks $\vn{P}$ such that the (equivalent of) \autoref{regexp} does not hold. A more interesting question would be to know if for all perennializations (and therefore the one defined by Girard) there exists a semi-regular outlook $\vn{P}$ such that the (equivalent of) \autoref{regexp} does not hold.

\subsection{Conclusion}

The results obtained in this section can be combined into the following theorem, which constitute the main technical result of this paper.
\begin{theorem}\label{dixmierclassll}
Let $\vn{P}$ be a maximal abelian sub-algebra of $\infhyp$. Then:
\begin{itemize}
\item if $\vn{P}$ is singular, there are no non-trivial interpretations of any fragment of linear logic by promising hyperfinite projects w.r.t. $\vn{P}$;
\item if $\vn{P}$ is semi-regular, one can interpret soundly multiplicative-additive linear logic (\mall{}) by promising hyperfinite projects w.r.t. $\vn{P}$;
\item if $\vn{P}$ is regular, one can interpret soundly elementary linear logic (\ells{}) by promising hyperfinite projects w.r.t. $\vn{P}$.
\end{itemize}
\end{theorem}

\section{Conclusion}

This work shows a correspondence between the expressivity of the fragment of linear logic reconstructed from GoI techniques and a classification of maximal abelian sub-algebras proposed by Dixmier \cite{dixmier}. Indeed, it was known that the interpretation of linear logic proofs in GoI models depends on the choice of the algebra $\vn{M}$ in which the GoI construction is performed, i.e. the hyperfinite factor $\infhyp$ in Girard's GoI5 \cite{goi5} or the algebra $\B{\hil{H}}$ of all operators on a separable infinite-dimensional Hilbert space $\hil{H}$ in earlier works \cite{goi1,goi2,goi3}. We showed here that another algebra influence this interpretation of proofs, a maximal abelian sub-algebra $\vn{A}$ of $\vn{M}$: this algebra, which represents the choice of a basis was implicitly fixed in early GoI models but appears explicitly in the hyperfinite GoI model. Dixmier's classification of MASAs specifies three particular types of such inclusions $\vn{A}\subset\vn{M}$ for a maximal abelian sub-algebra of a von Neumann algebra $\vn{M}$. We showed here that the expressivity of the fragment of linear logic interpreted in the model is closely related to the type of the chosen sub-algebra $\vn{A}$.

This work does not provide a complete correspondence, as Dixmier's classification is not an exhaustive one. While we showed that any exponential connective can be interpreted when $\vn{A}$ is a \emph{regular} MASA and that no non-trivial interpretation exists when $\vn{A}$ is a \emph{singular} MASA, the intermediate case of \emph{semi-regular} MASAs is not completely understood. Indeed we proved that one can \emph{at least} interpret multiplicative-additive linear logic in that last case, but some exponential connectives (although probably not all) might be interpreted in some cases. This opens the way for a more complete investigation of this case. A complicated approach would be to study the possible choices of semi-regular sub-algebras and understand how the exponential connectives behave with respect to them. We propose to study these questions using the systematic approach offered by Interaction Graphs \cite{seiller-goim,seiller-goia,seiller-goig,seiller-goie}. In particular, the use of \emph{graphings} allows one to consider a construction of GoI models parametrised by the choice of a so-called \emph{microcosm}: a microcosm is a monoid of measurable transformations of a measured space $X$. Such a monoid defines a measurable equivalence relation, which in turn defines a pair of a von Neumann algebra $\vn{M}$ and a MASA $\vn{A}$ of $\vn{M}$ by Feldman and Moore construction \cite{FeldmanMoore1,FeldmanMoore2}. 
This offers a more practical approach of this problem as it tackles the problem from a different direction: the MASA $\vn{A}$ is fixed, as it is constructed from the measure space $X$ (it is in fact the algebra $L^{\infty}(X)$ acting on the Hilbert space $L^{2}(X)$ by left multiplication), while the algebra $\vn{M}$ varies. Moreover, it allows for subtle distinctions on the inclusions of MASAs considered, distinctions that can be understood as restrictions of the computational principles allowed in the model \cite{seiller-lcc14}. The reader interested in more details about this approach should consult the author's recent perspective paper \cite{seiller-towards}.

\bibliographystyle{alpha}
\bibliography{../../Common/thomas}

\end{document}